\documentclass{Andrea}
\usepackage[latin1]{inputenc}
\usepackage{amsmath}
\usepackage{amssymb}
\usepackage{latexsym}
\usepackage{amscd}
\usepackage{mathrsfs}
\usepackage[all]{xy}
\usepackage{eufrak}
\usepackage{graphicx}
\usepackage[colorlinks=true,citecolor=red,linkcolor=red]{hyperref}

\renewcommand{\a}{\mathcal{A}}

\newcommand{\B}{\mathrm{B}}
\newcommand{\bs}[1]{\boldsymbol{#1}}
\newcommand{\CVD}{$\Box$}

\newcommand{\Def}{\mathrm{Def}}

\newcommand{\e}{\mathbf{e}}

\newcommand{\Ed}{\mathcal{E}^{\dag}}

\newcommand{\Fs}{\mathscr{F}}

\renewcommand{\H}{\mathscr{H}}

\newcommand{\Hom}{\mathrm{Hom}}

\newcommand{\Ka}{\widehat{K^{\mathrm{alg}}}}

\newcommand{\lr}[1]{\langle #1\rangle}
\newcommand{\M}{\mathrm{M}}

\newcommand{\Mod}{\mathrm{Mod}}

\newcommand{\N}{\mathrm{N}}
\renewcommand{\O}{\mathcal{O}}
\newcommand{\Od}{\O^\dag(A_0)}
\renewcommand{\P}{\mathscr{P}}

\newcommand{\Q}{\mathcal{Q}}
\newcommand{\R}{\mathcal{R}}
\newcommand{\Ring}{\mathcal{B}}

\newcommand{\simto}{\xrightarrow{\sim}}

\newcommand{\sm}[1]{\begin{smallmatrix}#1\end{smallmatrix}}

\newcommand{\U}{\mathrm{U}}

\renewcommand{\xi}{x}

\newcommand{\comment}[1]{
\begin{color}{red}{ 
\framebox{\framebox{\framebox{
\begin{minipage}{450pt}#1
\end{minipage}}}}}
\end{color}
}
\newcommand{\smallcomment}[1]{
\begin{color}{red}
{ 
\framebox{\framebox{\framebox{
#1}}}}
\end{color}
}

\usepackage{amsthm}
\swapnumbers

\makeatletter
\def\swappedhead#1#2#3{%
  \thmname{#1}\;%
  \thmnumber{\@upn{\the\thm@headfont#2\@ifnotempty{#1}}}%
  \thmnote{\,{\the\thm@notefont(#3)}}{.~}}
\makeatother

\newtheoremstyle{dotless-thm}
  {10pt}
  {10pt}
  {\itshape}
  {}
  {\bfseries}
  {}
  {.0em}
  {}
\theoremstyle{dotless-thm}

\newtheorem{theorem}{\textbf{Theorem}}[subsection]
\newtheorem{thm-intro}{\textbf{\textsc{Theorem}}}
\newtheorem{Def-intro}[thm-intro]{\textbf{\textsc{Definition}}}
\newtheorem{rk-intro}[thm-intro]{\textbf{\textsc{Remark}}}
\newtheorem{cor-intro}[thm-intro]{\textbf{\textsc{Corollary}}}
\newtheorem{ex-intro}[thm-intro]{\textbf{\textsc{Example}}}
\newtheorem{Crit-intro}[thm-intro]{\textbf{\textsc{Criterion}}}
\newtheorem{proposition}[theorem]{\textbf{Proposition}}
\newtheorem{lemma}[theorem]{\textbf{Lemma}}
\newtheorem{corollary}[theorem]{\textbf{Corollary}}
\newtheorem{definition}[theorem]{\textbf{Definition}}
\newtheorem{remark}[theorem]{\textbf{Remark}}
\newtheorem{example}[theorem]{\textbf{Example}}
\newtheorem{hypothesis}[theorem]{\textbf{Hypothesis}}
\newtheorem{setting}[theorem]{\textbf{Setting}}

\newtheorem{conjecture}[theorem]{\textbf{Conjecture}}
\newtheorem{notation}[theorem]{\textbf{Notation}}
\numberwithin{equation}{section}

\makeindex

\begin{document}

\title[Infinitesimal deformation of $p$-adic differential equations on Berkovich curves]
{Infinitesimal deformation of $p$-adic differential equations on Berkovich curves}

\author{Andrea Pulita}

\email{pulita@math.univ-montp2.fr}

\address{www.math.univ-montp2.fr/$\sim$pulita/\\
Departement de Mathématique, Université de Montpellier II, Bat 9, CC051, Place Eugène Bataillon, 34095 
Montpellier Cedex 05, France\\
\today.}



\begin{abstract}
We show that if a 
differential equation $\Fs$ over a 
quasi-smooth Berkovich curve $X$ has a 
certain compatibility condition with respect to an 
automorphism $\sigma$ of $X$,  
then $\Fs$ acquires a semi-linear action of $\sigma$ 
(i.e. lifting that on $X$). The compatibility 
condition forces the automorphism $\sigma$ to be close 
to the identity of $X$, 
so the above construction applies to a 
certain class of automorphisms called  
\emph{infinitesimal}. 
This generalizes \cite{An-DV} and 
\cite{Pu-q-Diff}. We also obtain an application to Morita's 
$p$-adic Gamma function, and to related values of 
$p$-adic $L$-functions.
\end{abstract}

\maketitle


\if{\makeatletter
\renewcommand\tableofcontents{%
    \subsection*{\contentsname}%
    \@starttoc{toc}%
    }
\makeatother

\begin{small}
\setcounter{tocdepth}{2} \tableofcontents
\end{small}
}\fi

\setcounter{section}{0}
\section*{\textsc{Introduction}}
\if{
\addcontentsline{toc}{section}{\textsc{INTRODUCTION}}
}\fi
Differential equations in the ultrametric world have been 
actively investigated since the proof of the rationality of 
the Zeta function by Dwork \cite{Dwork-zeta}. 
Most of the foundational ideas (as the 
over-convergence, the idea of Frobenius structure, the 
radius of convergence) come from the  
deep pioneering works of B.Dwork and P.Robba.  
Various languages have been used in the past half 
century to deal with the underling topological spaces, 
and  with the several aspects  
of the original ideas of Dwork and Robba. 
We have (among others) 
the rigid geometry of J.Tate, the 
rigid cohomology of P. Berthelot, 
and the language of Meredith.
 
More recently the point of view of Berkovich geometry 
appeared as a 
powerful tool to describe some new features related to 
differential 
equations as shown in the works of F. Baldassarri  
\cite{Balda-Inventiones}, K.Kedlaya \cite{Kedlaya-draft}, 
and the author \cite{NP-I}, \cite{Potentiel}, 
\cite{NP-II}, \cite{NP-III}, \cite{NP-IV}. 

This paper fits in this last sequence of works, 
and it is a generalization of 
\cite{An-DV}, \cite{DV-Dwork}, \cite{Pu-q-Diff} where one deals with 
$q$-difference equations. We generalize these last papers 
from two points of view: 

\begin{enumerate}
\item We consider arbitrary quasi-smooth $K$-analytic 
Berkovich curves;
\item We work with a larger class of automorphisms, 
called $S$-infinitesimal.
\end{enumerate}

Let $K$ be a complete valued ultrametric field, and let $X$ be a 
quasi-smooth\footnote{The terminology quasi-Smooth is that of 
\cite{Duc}, this corresponds to rig-smooth curves in the rigid 
analytic setting.} $K$-analytic Berkovich curve over $K$. 
Such a curve always admits a so called weak 
triangulation $S$, which is a certain locally finite set in 
$X$ such that $X-S$ is a disjoint union of 
virtual open annuli or discs. 
This is a nice way to cut the curve into 
local pieces. 
Each connected component of $X-S$ which is 
isomorphic to an annulus has a skeleton isomorphic to an 
open 
interval. The union of all those intervals in $X$, together 
with the 
points of $S$, form a locally finite graph $\Gamma_S$ 
called the 
$S$-skeleton of $X$. The complement of $\Gamma_S$ in 
$X$ is a 
disjoint union of virtual open discs.

An automorphism $\sigma:X\to X$ of $X$ is called 
\emph{$S$-infinitesimal} if, after all base changes to a 
complete valued field extension 
$L/K$, it induces an automorphism of each such disc. 
We attach to $\sigma$ a function 
$\R_S(-,\sigma):X\to\mathbb{R}_{\geq 0}$ 
that measures how much $\sigma$ is close to the 
identity. More precisely, if $\Omega/K$ is a large field 
extension such that $x\in X$ is $\Omega$-rational, 
and if $t$ is an $\Omega$-rational point of 
$X_\Omega$ over $x$, then $\R_S(x,\sigma)$ controls (aftr a convenient normalization) the radius of 
the smallest closed disc centered at $t$  which is stable under 
$\sigma_\Omega$, and which does not intersect $S$.

A first result is the following:

\begin{thm-intro}[\protect{Thm. \ref{Thm : finiteness sigma}}]
\label{Thm : intro-2}
The function $x\mapsto\R_S(x,\sigma)$ is continuous. Moreover 
there exists a locally finite graph 
$\Gamma_S(\sigma)\subseteq X$, 
containing $\Gamma_S$, such that $\R_S(-,\sigma)$ is 
locally constant outside $\Gamma_S(\sigma)$. 
If $\sigma\neq\mathrm{Id}_X$, the end points of 
$\Gamma_S(\sigma)$ that do not belong to 
$\Gamma_S$ are the rigid points of $X$ that are
fixed by $\sigma$.
\end{thm-intro}
The statement and its proof 
are essentially derived from those of 
\cite{NP-I}, \cite{Potentiel}, and \cite{NP-II} where we have proved that the radius of 
convergence function $\R_{S,1}(-,\Fs):X\to\mathbb{R}_{>0}$ of a 
differential equation $\Fs$ over $X$ satisfies similar finiteness 
properties. We recall that $\R_{S,1}(x,\Fs)$ controls (after normalization) the radius of the largest 
open disc centered at $t$ where $\Fs$ is trivial. 

Now, a differential equation $\Fs$ on $X$ is called 
$\sigma$-compatible if 
for all $x\in X$ we have
\begin{equation}\label{eq : compatibility}
\R_S(x,\sigma)\; < \;\R_{S,1}(x,\Fs)\;.
\end{equation}
This condition 
amounts to say that for all complete valued field extension $L/K$, 
$\sigma$ stabilizes globally each virtual open 
disc $D\subset X_L$ on which $\Fs_L$ is trivial. 

If $\Sigma$ is a family of automorphisms of $X$, a 
$\Sigma$-module over $X$ is a pair 
$(\Fs,\Sigma)$ where $\Fs$ is a locally free 
$\O_X$-module of finite type, 
and $\Sigma$ is a family of $\O_X$-linear isomorphisms 
$\Sigma:=
\{\sigma:\Fs\xrightarrow{\sim}\sigma^*(\Fs)\;,\;
\sigma\in\Sigma\}$. 
We say that a differential equation $(\Fs,\nabla)$ 
is $\Sigma$-compatible if it is 
$\sigma$-compatible for all $\sigma\in\Sigma$.
A second result of this paper is the following:
\begin{thm-intro}[\protect{Thm. \ref{Thm : deformation fredft}}]
\label{Thm : intro-2}
Let $\Sigma$ be a family of $S$-infinitesimal automorphisms of 
$X$, and let $\Fs$ be a $\Sigma$-compatible differential 
equation. 
Then there exists on $\Fs$ a canonical structure of $\Sigma$-module.

If $\alpha:\Fs\to\Fs'$ is an $\O_X$-linear map between 
$\Sigma$-compatible differential equations commuting with 
the connections, then $\alpha$ also commutes with the corresponding 
action of $\Sigma$ on $\Fs$ and $\Fs'$.
\end{thm-intro}\label{Thm-2 intro}
The theorem provides the existence of a faithful functor
\begin{equation}
\Def_{S,\Sigma}\;:\;\{\textrm{$\Sigma$-compatible differential equations}/X\}\;
\xrightarrow{\quad}\;\{\Sigma-\textrm{Modules}/X\}\;.
\end{equation} 
In section \ref{Fully faithfulness and non degeneracy} 
we provide a condition on the 
action of $\Sigma$ on $X$ that guarantee the fully faithfulness.

The existence of the $\Sigma$-module structure on $\Fs$ is obtained 
as a certain pull-back of the stratification associated to it 
(cf. Section \ref{Stratifications.}). 
In down to earth terms if $D\subseteq X$ is a disc stable by $\Sigma$ 
on which the differential equation $(\Fs,\nabla)$ 
admits a complete basis of Taylor solutions $Y\in GL_r(\O(D))$, 
the $\Sigma$-module associated to $\nabla$ also admits the same 
basis of Taylor solutions on $D$.

\begin{ex-intro}
Let $X$ be the open unit disc in 
$\mathbb{A}^{1,\mathrm{an}}_K$, 
with $S=\emptyset$. 
Let $\Fs$ be the differential equation $y'=y$ on $X$. 
For all complete valued field extension $\Omega/K$, and 
all rational point $t$ of $X_\Omega$ 
the solution of $\Fs$ around $t$ is $\exp(T-t)$. Its 
radius function $\R_{S}(-,\Fs):X\to \mathbb{R}_{>0}$ 
is constant, and its value $\omega$ is either 
equal to $|p|^{\frac{1}{p-1}}$ if $K$ is a $p$-adic field, 
or $1$ if the absolute value of $K$ is trivial 
on $\mathbb{Z}$. Now, if $\sigma=\sigma_q:X\to X$ is 
the multiplication by $q\in K$, 
then $\R_{S}(-,\sigma_q):X\to\mathbb{R}_{\geq 0}$ 
coincides with the function $x\mapsto|(q-1)T|(x)$, i.e. 
the function associating to $x\in X$ the 
absolute value of the function 
$(q-1)T:X\to\mathbb{A}^{1,\mathrm{an}}_K$ at $x$. 
Condition \eqref{eq : compatibility} is then equivalent to
$|q-1|\leq \omega$.

The $\sigma_q$-module corresponding to $y'=y$ is 
$y(qT)=A(q,T)y(T)$, where $A(q,T)=\exp((q-1)T)$. 
Indeed $y=\exp(T)$ is a solution, and 
$A(q,T)=\exp((q-1)T)\in\O(X)$ as soon as 
$|q-1|\leq \omega$.
\end{ex-intro}

We also obtain an \emph{analytic} version of the above result.
If we have a $K$-analytic group $G$ acting on $X$ by 
$S$-infinitesimal automorphisms 
(here the general formalism is that of 
\cite[Section 12]{Mumford-AbVar}), 
then the deformation functor furnishes an 
\emph{analytic semi-linear $G$-module} structure of $\Fs$ 
(i.e. lifting that of $G$ on $\O_X$) which is compatible with the connection (cf. Section 
\ref{Analytic semi-linear $G$-modules.} for more details). 
This structure is commonly known as \emph{$G$-equivariant 
$D$-module on $X$} 
(cf. \cite{Mumford-invariants},  
\cite{Kashi-invariant}). 

If the quotient $[X/G]$ of $X$ by $G$ exists, 
Theorem \ref{Thm : intro-2} amounts to 
say that a $G$-compatible differential equation on $X$ gives 
a fiber bundle on $[X/G]$. 
In our context the quotient does not necessarily exist 
in the category of $K$-analytic spaces, 
and it has to be replaced by the usual simplicial 
object attached to the action of $G$ on $X$ 
(cf. \eqref{eq : simplicial object}) i.e. by a Stack 
(as in \cite{Thomason-Alg-K-Th}).

From this point of view the result is similar to that of J.Sauloy 
\cite{Sauloy-elliptic} where in the complex theory of $q$-difference 
equations he shows that a certain class of 
$q$-difference equations is equivalent to a 
category of fibered boundles over the elliptic curve 
$\mathbb{C}^*/q^{\mathbb{Z}}=[X/G]$.

In the case of automorphisms of the 
form $f(T)\mapsto f(qT+h)$ of the affine line 
(this includes as special cases $q$-difference, and difference 
equations) we are able to construct an explicit description 
of the image of the functor $\mathrm{Def}_{S,\Sigma}$, and a 
quasi-inverse functor 
called \emph{Confluence} (cf. Section \ref{confluence}).

As an application, 
in section \ref{Gamma} 
we apply the previous theory to a 
particular difference equation satisfied by $\Gamma_p$. 
We prove the following results:
\begin{thm-intro}[cf. Thm. \ref{Gammmammmma}]
The Morita's $p$-adic Gamma function $\Gamma_p$ is a 
solution of a 
first order differential equation over the open unit disc.   
\end{thm-intro}
Moreover we relate the radius 
of convergence of that differential equation to 
the $p$-adic valuation of certain values at 
\emph{positive integers} of the $L$-functions 
appearing in the Taylor expansion of the function 
$\log(\Gamma_p)$. We then 
prove a 
family of congruences between these values 
(cf. Corollary \ref{Cor : congruences}).

In section \ref{Deformation over the Robba ring}
we focus our attention to a local situation: 
that of differential equations over the Robba ring. 
In this situation an important classification result 
is the so called $p$-adic local monodromy theorem 
\cite{An}, \cite{Ked}, \cite{Me}, saying that each differential equation 
with an (unspecified) Frobenius structure is quasi-unipotent. 

From Theorem \ref{Thm : intro-2} we deduce the 
following analogue of the $p$-adic local monodromy theorem:
\begin{thm-intro}[\protect{Thm. \ref{plmt-diff}}]
\label{Thm-intro : quasi-unipotence}
Each $\Sigma$-module over the Robba ring, 
that is the deformation of a differential equation with a Frobenius 
structure is quasi unipotent.
\end{thm-intro}
In particular this allows us to obtain a characterization of the 
essential image of the deformation functor for equations over the 
Robba ring (cf. Cor. \ref{Corollary : Essential image Case of Robba}). 

We mention that the proof of Theorem 
\ref{Thm-intro : quasi-unipotence} requires a 
$\Sigma$-analogue of the Katz's 
existence of a canonical extension functor \cite{Katz-Can}, 
\cite{Matsuda-unipotent} (cf. Section 
\ref{Section  Katz-Matsuda can ext}). 

In the case of $q$-difference equations, 
Theorem \ref{Thm-intro : quasi-unipotence} is the central 
result of \cite{An-DV}, which is the foundational paper of the 
theory. The difference with \cite{An-DV} is that they deduce the 
existence of the deformation functor from the quasi-unipotence of $q$-difference 
equations by showing that the Tannakian groups of the two categories are isomorphic. 
Our approach goes in the opposite direction, we deduce the quasi 
unipotence from the deformation.

\if{
\subsection*{\textsc{Structure of the paper}} 
In section \ref{Notation} we fix notation about Berkovich 
curves (slopes, graphs, ...). 
In section \ref{Linear differential equations} we recall the 
finiteness results of \cite{NP-I} and \cite{NP-II} about 
differential equations, and the formalism of 
stratifications. 
In section \ref{S-infinitesimal automorphisms} we 
introduce $S$-infinitesimal automorphisms, and we 
prove Theorem \ref{Thm : intro-2}. 
In section \ref{Deformation} we construct the 
deformation, we prove Theorem \ref{Thm : intro-2}. 
In section \ref{Deformation over the Robba ring} 
we deal with the Robba ring and we 
obtain Theorem \ref{Thm-intro : quasi-unipotence}. 
In section \ref{Difference equations over the affine line.} 
we specialize the above results to 
difference equations over 
$\mathbb{P}^{1,\mathrm{an}}_K$, and in section 
\ref{Gamma} 
we obtain the application to the $p$-adic Gamma 
function.
}\fi

Finally we wish to quote two recently appeared papers
of B.Le Stum and A.Quiros \cite{LSQ2,LSQ1}, 
where the authors obtain with different methods
similar results, allowing for instance positive 
characteristic. 
The point of view of those papers is more related to
that of Rigid cohomology of varieties of positive
characteristic. 

\subsection*{\textsc{Acknowledgments}} 
We are grateful to D.Barsky for suggesting and helpful 
discussions, and for guidance and advice in the formulas 
of Section \ref{Section : Applications -SP}. 
We also thank Yves André, Gilles Christol, 
Jerôme Poineau, Bernard Le Stum, Michel Gros,
and Bertrand Toen for 
helpful discussions.

\if{

\subsection{Elementary Stratifications and differential modules}
\label{Elementary Stratifications and differential modules-yhn,ko} 
It is known since \cite[Chapitre II]{Berthelot-these} (see also 
\cite{Deligne-Reg-Sing}, 
\cite[Chap VIII]{Illusie-Cotangent-II}, \cite{Katz-Travaux-de-Dwork}, ...) that the data of a homogeneous linear 
differential equation is equivalent to a so called \emph{stratification}. This equivalence has been established in 
great generality in \cite[Chapitre II]{Berthelot-these} in which one gives the axiomatic presentation of this 
correspondence for objects in a general topos. The formalism works as well for complex and ultrametric differential 
equations (but also for much more general frameworks like crystals of a variety of characteristic $p>0$, or sheaves of 
differential equations over a site). This approach goes back to Grothendieck 
\cite[Appendix]{Grothendieck-10-exposes}, but the 
elementary idea seems to have been already known before (as Grothendieck actually affirms). It is essentially a 
``sublimation'' of the Cauchy's existence theorem of solution of differential equations.
In this paper 
we do not use the formalism of \cite[Chapitre II]{Berthelot-these}, because it is too general and too sophisticated for our 
purposes. In Section \ref{Linear Differential Equations} we simply recall a special case of the above correspondence 
in our context. We also provide all the proofs for the convenience of the reader, in order to make the paper self 
contained, but also because there are many statements that are not explicitly written in the literature. 

Fix a field $K$ of characteristic $0$, complete with respect to an ultrametric absolute value $|.|$. In this 
introduction we will consider (homogeneous linear) differential equations over an affinoid $X = 
\mathrm{D}^+(c_0,R_0)-\cup_{i=1,\ldots,n}\mathrm{D}^-(c_i,R_i)$ (cf. Section \ref{Affionids section ertyxccvza}), but 
everything can be done for more general kinds of spaces that will be used in the text (like open discs, open annuli, 
germs of open annuli, ...). In this context the data of a differential module $(\M,\nabla^{\M})$ 
over $X$ (cf. Section \ref{Linear Differential Equations}) is equivalent to the data of an isomorphism that we call \emph{stratification}
\begin{equation}
\chi_{\M}\;\;:\;\;(p_2^*\M)_{|_V}\;\;\xrightarrow[]{\;\;\sim\;\;}\;\;(p_1^*\M)_{|_V}
\end{equation}
satisfying a certain cocycle property (cf. Section \ref{cocycles}), where $p_i:X\times X\to X$, $i=1,2$, are the 
projections and $V$ is a convenient neighborhood of the diagonal of $X\times X$. A morphism $\alpha:\M\to\N$ commutes 
with the $\nabla$'s if and only if $\chi_{\M}\circ(p_2^*\alpha)_{|_V}=(p_1^*\alpha)_{|_V}\circ\chi_{\N}$. 

The neighborhood $V$ depends on $\M$. 
In the literature one often choses $V$ as a formal neighborhood of the diagonal, or some 
PD-neighborhood (cf. \cite{Katz-Travaux-de-Dwork}). For our purposes we need the largest kind of neighborhood. 
We then consider the following class of neighborhoods: 
we assume that $V$ contains a tube $\mathcal{T}(X,R):=\{(x,y)\in X\times 
X\;|\;|x-y|<R\}$ for an unspecified real number $R>0$. 

In down to earth terms, if a basis of $\M$ is fixed, the data 
of $\nabla^{\M}$ corresponds to a differential equation $Y'=G(T)Y$, $G(T)\in M_n(\H_K(X))$. On the other hand, the data 
of $\chi_\M$ corresponds (in the same basis of $\M$) to a matrix $Y(z,w)$ with coefficients over a tube around the 
diagonal, that will be the Cauchy's solution of that equation around a ``generic point'' $w$. 
This means that  specializing the second variable of the invertible matrix $Y(z,w)$ 
at an arbitrary point $c\in X(\Omega)$, for some complete valued field extension $\Omega/K$, we obtain the Cauchy's 
solution $Y(z,c)$ of the system $\frac{d}{dz}Y(z,c)=G(z)Y(z,c)$ around the point $c$, with initial value 
$Y(c,c)=\mathrm{Id}$. 

One of the major differences 
between ultrametric and archimedean differential equations is that, in the non-archimedean world, 
the radius of convergence of $Y(x,c)$ depends on $\M$, and on the point $c$, while in the archimedean 
world the radius of 
convergence depends only on $c$ but it does not depend on the particular differential module $\M$. 
Another fundamental concept is that the \emph{convergence locus $\mathcal{U}$ of $Y(x,y)$ (i.e. of $\chi_\M$) 
is usually not reduced to a tube $\mathcal{T}(X,R)$}. Indeed, the geometry of $\mathcal{U}$ encodes some important 
invariants of the equation (like for example the irregularity of $p$-adic differential equations in the context of 
Christol-Mebkhout \cite{Ch-Me}). It is actually quite difficult to describe it, 
because it is defined by some \emph{conditions that are \emph{à priori} not within the framework of (rigid) analytic geometry} 
(cf. Sections \ref{Properties of the x-radius and the y-radius.} 
and \ref{uniform neighborhoods of the diagonal}). 
In order to study the convergence locus of the stratification $\chi_\M$ people 
consider sections of $\mathcal{U}$ (following the Dwork's ideas), 
by specializing the second variable $w$ of the matrix $Y(z,w)$ to the points $c$ 
of $X(\Omega)$, for an unspecified arbitrarily large complete valued field extension 
$\Omega/K$.\footnote{It is understood that the absolute value of $\Omega$ extends that of $K$.} 
One then checks, for all such couple $(c,\Omega)$, with $c\in X(\Omega)$, 
the radius of convergence of the matrix solution $Y(z,c)$ of the 
starting differential system $Y'=GY$. An important fact is that the radius of convergence depends only on the 
semi-norm $|.|_c$ defined by $|f|_c:=|f(c)|_{\Omega}$, and not on the point $c$. 
The semi-morm $|.|_c$ is a point of the Berkovich space $\mathscr{M}(X)$ of $X$, 
and reciprocally every point of the Berkovich space of $X$ is of this type for a convenient choice of $c$ and 
$\Omega$ (cf. section \ref{Dwork generic points}).
In other words the radius of convergence of the system $Y'=GY$ is a well defined real valued function on 
$\mathscr{M}(X)$ that is proved to be continuous by \cite{DV-Balda} (cf. Corollary 
\ref{Continuity of the minimum radius of the columnsnnsn}). In the paper we will need an accurate understanding 
of the behavior properties of the radius of convergence function along the Berkovich space of $X$ (cf. section 
\ref{intro : ber sp - x-y-rad} below).

\subsection{Stratified $\sigma$-modules and $\sigma$-Deformation functor} 
Let now $\sigma:X\simto X$ be an automorphism of $X$. 
A (classical) $\sigma$-module is a finite free module $\M$ over the ring $\H_K(X)$ 
of analytic functions of $X$, together with a $\sigma$-semi-linear isomorphism $\sigma^{\M}:\M\simto\M$. 
This means that $\sigma$ verifies $\sigma^{\M}(fm)=\sigma(f)\sigma^{\M}(m)$ for all $f\in\H_K(X)$ 
and all $m\in\M$. The data of $\sigma^{\M}$ is equivalent to the data of a $\H_K(X)$-linear isomorphism 
(still denoted by $\sigma^{\M}$) 
$\sigma^{\M}:\sigma^*\M\simto\M$, where 
$\sigma^*$ is the pull back of $\M$ by the $K$-linear ring automorphism $\sigma:\H_K(X)\simto\H_K(X)$ 
(cf. section \ref{Difference equations}). 
As for differential equations, once a basis of $\M$ 
is chosen, the data of a $\sigma$-module corresponds to a so called $\sigma$-difference equation 
$\sigma(Y)=A(T)\cdot Y$. %
Now let $(c,\Omega)$, $c\in X(\Omega)$, be as above (cf. section 
\ref{Elementary Stratifications and differential modules-yhn,ko} , 
and section \ref{Affionids section ertyxccvza}). 
We consider a disc $\mathrm{D}^-(c,r)$. 
Contrary to the archimedean framework, if $\sigma$ is close enough to the identity, it happens that  
$\mathrm{D}^-(c,r)\subset X$ is globally stable under the action of $\sigma$ 
(i.e. $\sigma(\mathrm{D}^-(c,r))\subset \mathrm{D}^-(c,r)$, see Lemma 
\ref{invariance of a disc} to have an explicit \emph{characterization} of those $\sigma$ that stabilize a given disc).
It makes hence sense to consider solutions $Y$ of $\sigma$-difference equations with coefficients in 
the algebra of analytic functions over $\mathrm{D}^-(c,r)$. 
Now consider, as above, a differential equation $(\M,\nabla^{\M})$ over $X$,
and its attached stratification $\chi_{\M}$. Fix a basis $\e$ of $\M$ and consider 
the matrix $Y(z,w)$ of $\chi_{\M}$. Roughly speaking the main result of this paper is the following. 
If for all couples $(c,\Omega)$ as above the largest disc $\mathrm{D}^-(c,R)$ contained in $X$ on 
which the function $Y(x,c)$ converges is globally stabilized by $\sigma$, then we prove that 
there is a canonical $\sigma$-semi-linear action $\sigma^{\M}:\M\simto\M$ of $\sigma$ on $(\M,\nabla)$. 
This action is determined by the fact that if $A(x)$ is the matrix of the action of $\sigma$ on 
$\M$, in the chosen basis $\e$, then for a (and hence for every) point $c$ one has $Y(\sigma(x),c)=A(x)\cdot Y(x,c)$. 
In other words, the Taylor solutions of $(\M,\nabla^{\M})$ are also solutions of the 
$\sigma$-difference equation $\sigma(Y)=A(x)Y$ defined by $(\M,\sigma^{\M}:\M\simto\M)$. 
We shall immediately to caution the reader against the fact that it is an open problem to know 
whether a $\sigma$-module $(\M,\sigma^{\M})$ admits an analytic cocycle 
$Y(z,w)$ satisfying the last property (cf. see section \ref{Formal cocycle} 
for a solution of this problem in the case 
of $(q,h)$-difference equations). We call \emph{stratified} the $\sigma$-modules with this property (cf. 
Def. \ref{Stratified sigma modules d}). We have to be careful here since it is actually unknown whether a 
sub-$\sigma$-module of a stratified $\sigma$-module is stratified or not (cf. Remark \ref{sub-strat non strat ?}).
  
In order to stabilize each disc of convergence of the Taylor solution $Y(z,w)$, the minimal assumption that 
we have to consider for $\sigma$ is that at least \emph{each maximal open disc contained in $X$ has 
to be globally stabilized by $\sigma$}. 
This means that $\sigma$ has to be quite close to the identity. We call such kind of automorphisms 
\emph{infinitesimal} 
(cf. Def. \ref{Definition infinitesimal}).\footnote{If $\sigma=\sigma_q$ is the $q$-difference automorphism (i.e. 
$\sigma_q(f(x))=f(qx)$), then infinitesimality implies $|q-1|<1$ (cf. Lemma \ref{infinitesimality of sigma_q,h}).} 
Reciprocally, for a chosen infinitesimal automorphism $\sigma$, we call \emph{$\sigma$-compatible} every differential 
equation such that 
$\sigma$ stabilizes globally each 
disc of convergence of its Taylor solution $Y(z,w)$. In the paper we provide some criteria to test whether an 
automorphism is 
infinitesimal (cf. Lemma \ref{lemma: non so dare nome**}) and if a stratification (or equivalently a cocycle $Y(z,w)$) 
is $\sigma$-compatible (cf. Proposition \ref{compatibility condition lemma}). The principal result of this paper is 
then the following (see  Theorem \ref{Main Theorem over affinoids ....frgh} and 
section \ref{proprreties pf deformation} for a more complete statement.)
\begin{theorem}[(cf. Thm.\ref{Main Theorem over affinoids ....frgh})]\label{th.main intro .}
Let $\sigma$ be an \emph{infinitesimal} automorphism of $X$. 
There exists an exact and faithful functor, called $\sigma$-deformation, associating to a $\sigma$-compatible 
differential module a $\sigma$-module. Its essential image is formed by stratified $\sigma$-modules. 
If moreover $\sigma$ is non degenerate, then the functor is fully faithful and realizes an equivalence between 
stratified $\sigma$-modules and $\sigma$-compatible differential equations. 
\end{theorem} 
The condition of non degeneracy of $\sigma$ appearing in the theorem is technical (cf. Def. \ref{Definition non deg}). 
It forces the category 
of $\sigma$-modules 
to be exactly $K$-linear, and not $R$-linear for a larger ring $R/K$. We provide a criteria to test quite easily the 
non degeneracy of $\sigma$ (cf. Lemma \ref{Criterion of non degeneracy -fififi}). In the practice the most part of 
automorphisms are  non degenerate. 
For example if $\sigma=\sigma_q$ is the $q$-difference automorphism, then non degenerate means $q$ being different 
from a root of unity.

The proof of the above theorem is the following. The $\sigma$-deformation functor is defined 
as the composite of two functors. We start from a differential module $(\M,\nabla)$, 
we  first consider the equivalence of categories associating to $(\M,\nabla)$ its stratification 
$\chi_\M:p_2^*\M\simto p_1^*\M$, that we assume to be defined over $\mathcal{T}(X,R)$, for some $R>0$. 
Then we consider the pull back functor by a morphism $\Delta_\sigma:X\to X\times X$, 
defined by $\Delta_\sigma(x)=(\sigma(x),x)$:
$$\chi_\M:p_2^*\M\xrightarrow[]{\;\sim\;} p_1^*\M \qquad\stackrel{\Delta_\sigma^*}{\leadsto}\qquad 
\M=\Delta_\sigma^*p_2^*\M\;\xrightarrow[\;\Delta^*_\sigma(\chi_\M)\;]{\;\sim\;}\;
\Delta_\sigma^*p_1^*\M=\sigma^{*}\M\;.$$
So we obtain the required isomorphism $\sigma^{\M}:=\Delta^*_\sigma(\chi_\M):\M\simto\sigma^*\M$. Now we encounter the 
following problem. The stratification $\chi_\M$ 
attached to $(\M,\nabla)$ is defined over a tube $\mathcal{T}(X,R)$, and in order to make 
sense to $\Delta_\sigma^*(\chi_\M)$ we need the assumption
$\Delta_\sigma(X)\subset\mathcal{T}(X,R)$, 
otherwise we can not consider the composite function $\Delta^*_\sigma(\chi_\M)$. 
Indeed the matrix $A(x)$ of $\sigma^{\M}:=\Delta_\sigma(\chi_\M)$ is given by the composite of the matrix $Y(x,y)$ of 
$\chi_\M$ with $x\mapsto (\sigma(x),x)$: $A(x)=Y(\sigma(x),x)$.
The difficulty of the proof lies in 
the fact that $\Delta_\sigma(X)$ can be contained in the convergence locus $\mathcal{U}$ of $\chi_\M$, without being 
contained in any tubular neighborhood. 
Indeed, as already mentioned in section \ref{Elementary Stratifications and differential modules-yhn,ko}, 
the convergence locus of $\chi_\M$ is often not reduced to a 
tubular neighborhood, and is a subset of $X\times X$ whose definition \emph{is given by some conditions that are not 
within the framework of rigid analytic geometry}. So we actually  prove the existence of a covering of $X$ by 
sub-affinoids in order that, over each affinoid $X'\subset X$ of the covering, one has 
$\Delta_\sigma(X')\subset\mathcal{T}(X',R')\subset\mathcal{U}$. 
In section \ref{sigma-compatibility cond ition} we provide the existence of such a covering in a greater generality 
for arbitrary analytic 
functions in a neighborhood of the diagonal that are $\sigma$-compatible. 
This provides the good definition of the composite function 
$\Delta_\sigma^*(\chi_\M)$, in other words it proves that the matrix $A(x)=Y(\sigma(x),x)$ is well defined and 
analytic over $X$.

\subsection{Generalized $\Sigma$-modules as glueing data}\label{intro generalized sigma modules sss}
In section \ref{Difference equations} we introduce a generalized formalism that permits to express under 
the same language different structures as $\Sigma$-difference equations, and stratifications. In this way we treat 
under the same formalism stratifications and $\sigma$-modules, so 
that the deformation appears as a pull-back functor by $\Delta_\sigma$ (cf. sections \ref{PULL-BACJK} and 
\ref{Definition of the deformation functor: the tubular case}). We apply this philosophy to the Frobenius automorphism, and in that case we obtain, as an application, 
a sort of ``analytic continuation along Frobenius" that permits to extend the interval of definition of a $(\varphi,\nabla)$-module. 
In this context the data of a Frobenius structure appears to be equal to a certain gluing data 
(see section  \ref{appendix : intro} below, and Appendix \ref{Appendix : Local Frobenius Structure}).

\subsection{The theory over the Robba ring}
We pay particular attention to differential equations over the Robba ring 
(cf. section \ref{sigma-modules and Differential equations over the Robba ring}). One of the most important results 
in $p$-adic differential equation theory is the so called \emph{$p$-adic local monodromy 
theorem} (cf. \cite{An}, \cite{Ked}, \cite{Me}). It affirms that a differential equation over the Robba 
ring (with a Frobenius structure) becomes unipotent after scalar extension to a particular kind of covering called 
étale. We prove, by $\sigma$-deformation, the analogue of that theorem for stratified $\sigma$-modules (cf. 
Thm.\ref{quasi unipotence}). 
The proof needs to state the $\sigma$-analogue of the \emph{Katz's canonical extension functor} (cf. 
\cite{Katz-local-to-global}, \cite{Katz-On-the-calculation}), that has 
been obtained in the $p$-adic context by S.Matsuda, cf. \cite{Matsuda-unipotent}.

\subsection{Analytic dependence on a parameter}
Theorem \ref{th.main intro .} is stated, in the paper, for a family $\Sigma$ of operators 
instead of a single operator $\sigma$, since no additional difficulties are required. 
In particular we take into account the possibility to have an analytic variety 
$\Sigma=\mathcal{G}$ acting continuously by infinitesimal automorphisms on $X$. In this 
case we obtain a functor associating to a $\mathcal{G}$-compatible differential module the same module together with 
a semi-linear action of $\mathcal{G}$. We prove in this context that the action of $\mathcal{G}$ obtained by 
deformation is analytic (cf. Section \ref{analyticity of the action of G et Sigma}). 
In other words the family of $\sigma$-difference equations, 
$\sigma\in\mathcal{G}$ obtained by deformation, varies continuously with $\sigma$. 
As an example we consider $\mathcal{G}$ to be a product of discs 
$G=\Q_{\tau,\nu} := \mathrm{D}_K^-(1,\tau) \times
\mathrm{D}^-_K(0,\nu) \subseteq K^\times$, 
of a convenient radius $(\tau,\nu)\neq(0,0)$, acting on $X$ by  $\sigma_{q,h}(x)= q\cdot x+h$, where 
$q\in\mathrm{D}^-(1,\tau)$, $h\in\mathrm{D}^-(0,\nu)$, $x\in X$. 
Then we recover the theory of $(q,h)$-difference equations. 
The analyticity of the action of $\Q_{\tau,\nu}$ generalizes the result of 
\cite{Pu-q-Diff} about the 
analyticity of the action of the $q$-difference operator $\sigma_q:=\sigma_{q,0}$ with respect to $q$. 
Notice that a stratified $\Q_{\tau,\nu}$-module $(\M,\sigma^{\M})$ 
carries an action of the Lie algebra of $\Q_{\tau,\nu}$.
In section \ref{Germs of stratified Q-actions} we precise 
the relations between the action of the Lie algebra on $\M$ and the differential equation 
associated to $(\M,\sigma^{\M})$ by deformation.

\subsection{Roots of unity}
The deformation functor does not require the family $\Sigma$ of operators to be an analytic variety. An interesting 
application is the case of $p^n$-th 
roots of unity $\Sigma:=\bs{\mu}_{p^{\infty}}=\cup_{n\geq 1}\bs{\mu}_{p^n}$. Let $K_\infty:=K(\bs{\mu}_{p^\infty})$, 
and let $K^{\mathrm{alg}}$ be the algebraic closure of $K$. Then $\Sigma=\bs{\mu}_{p^{\infty}}$ acts on 
the Robba ring $\mathcal{R}_{K_{\infty}}$, by dilatations $\sigma_{\xi}(f(x)):=f(\xi x)$ (the same as a 
$q$-difference operator with $q=\xi$). 
In this case the category of $\sigma_{\xi}$-equations is $R$-linear for a non trivial ring of functions 
$R/K_{\infty}$ so it can not be equivalent to a full sub-category of that of differential equations because the 
latter is $K_\infty$-linear. 
The same is true if we consider the action of the family $\Sigma=\bs{\mu}_{p^n}$ instead of a single operator. 
But actually the $\Sigma$-deformation functor exists at every level $n$. Unfortunately this faithful (but not full) 
functor is somehow  trivial since it sends a differential equation into a $\bs{\mu}_{p^n}$-semilinear-module that is 
always isomorphic to a sum of the unit object (cf. \cite[Prop.8.6]{Pu-q-Diff}). 
Nevertheless if one consider the action of 
the whole group $\bs{\mu}_{p^{\infty}}$ we have the following:
\begin{theorem}[(cf. Cor.\ref{action of mu p infty infty})]
Let $\Sigma:=\bs{\mu}_{p^{\infty}}$. The category of stratified $\Sigma$-modules over $\R_{K_{\infty}}$
(i.e. stratified $\mu_{p^{\infty}}$-semi-linear representations with coefficients in $\R_{K_\infty}$) with an 
unspecified Frobenius structure is equivalent to that of differential equations over $\R_{K_\infty}$ with an 
unspecified Frobenius structure.
\end{theorem}

The latter category is that appearing in the paper of Y.André (cf. \cite{An}). André proves that this particular 
category of differential equations is actually equivalent to that of $\mathcal{I}_{k^{\mathrm{alg}}((t))}$-linear 
representation, where 
$\mathcal{I}_{k^{\mathrm{alg}}((t))}$ is the inertia of the absolute Galois group of the field of power series 
$k^{\mathrm{alg}}((t))$ 
with coefficients in the residual field $k^{\mathrm{alg}}$ of $K^{\mathrm{alg}}$. 
So as a corollary we have the following
\begin{corollary}[(cf. Cor.\ref{action of mu p infty infty})]
The category of $\mathcal{I}_{k^{\mathrm{alg}}((t))}$-continuous linear representations is equivalent 
to that of stratified 
$\bs{\mu}_{p^{\infty}}$-modules over $\R_{K^{\mathrm{alg}}}$ 
(i.e. finite free $\R_{K^{\mathrm{alg}}}$-modules together with a 
semi-linear action of $\bs{\mu}_{p^{\infty}}$ that is stratified). 
\end{corollary}
This approach for the study of the $q$-difference equations, with $q$ equal to a root of unity $\xi_{p^n}$, 
is complementary to that of \cite{Pu-q-Diff}. Indeed in \cite{Pu-q-Diff} the idea consists in considering, 
somehow artificially, an additional data on the $\sigma_{\xi_{p^n}}$-equation: 
the action of a $\sigma_{\xi_{p^n}}$-tangent operator. In this way one constructs, for $q=\xi_{p^n}$, a category which 
is equivalent to that of differential equation as wanted. The objects of this category are 
$q$-difference modules together with an action of the $q$-tangent operator. But if the problem is to understand 
the category of $q$-difference equations with $q$ being a root of unity, one finds no useful statements in 
\cite{Pu-q-Diff}, while the above result gives informations in this direction. Namely the following question arises 
naturally. The above fully-faithful functor identifies the category of differential equations 
to a subcategory of that of $\bs{\mu}_{p^{\infty}}$-modules. As an example it may happen that the image 
of an irreducible differential equation splits into sub-quotients that are not in the essential 
image of the functor. So these sub-quotients are invariants of the starting differential equation. Hence a 
classification of the extension classes of $\bs{\mu}_{p^{\infty}}$-modules would be helpfull in the study of 
differential equations.

\subsection{Applications}
\subsubsection{Finite difference equations.} 
In section \ref{Finite Difference equations} we specialize the above theory to the case of finite 
difference equations in which $\sigma$ is given by $\sigma_{q,h}(T):=qT+h$ for conveniently small values 
of $q\in K^\times$ and $h\in K$. This part contains, clarifies, simplifies, and generalizes the results 
of \cite{Pu-q-Diff} about $q$-difference equations (i.e. where $h=0$). 
In this context we have a richer situation because we dispose of a so called \emph{twisted Taylor formula} for 
$\sigma_{q,h}$-difference modules (cf. formula \eqref{cocycle of q-h equation}). 
This permits to express explicitly the infinitesimality and the non degeneracy of $\sigma_{q,h}$ in terms of $q$ and 
$h$. It also permits to give a criterion for the existence of a stratification of $\Sigma$-modules which is given 
in terms of the powers of the $(q,h)$-twisted derivation $\Delta_{q,h}:=\frac{\sigma_{q,h}-1}{(q-1)T+h}$ (cf. 
Proposition \ref{Convergence of formal cocycle}). In other words we are able to describe the essential image of the 
deformation functor. The existence of such a twisted Taylor formula seems to be a peculiarity of the automorphism 
$\sigma_{q,h}$. In fact for a general automorphism $\sigma$ the existence of such a twisted Taylor formula is very 
improbable because the twisted $\sigma$-derivation $\Delta_\sigma:=\frac{\sigma-1}{\sigma(T)-T}$ presents denominators 
and hence singularities.

\subsubsection{The Morita's $p$-adic Gamma function $\Gamma_p$.} As an application we prove that 
the Morita's $p$-adic Gamma function $\Gamma_p$ verifies a first order linear differential equation 
$\Gamma_p'(T)=g_0(T)\cdot\Gamma_p(T)$, where $g_0(T)=\sum_{i\geq 0}a_iT^i$ is an analytic function 
converging on the open unit disc $\mathrm{D}^-(0,1)$ (cf. Thm. \ref{Gammmammmma}). 
We obtain this result by ``\emph{approaching}'' the differential equation with the family of finite difference 
equations (i.e. functional equations) satisfied by $\Gamma_p$: 
\begin{equation}\label{familyyfhgzi hujih jhgd lkjd }
\Gamma_p(T+p^n)=A(p^n,T)\Gamma_p(T)\;,
\end{equation} for increasing 
values of $n$. 

We then compute the radius of convergence of all these difference equations and of the limit differential equation. 
This radius is in direct relation with the Newton polygon of the function $g_0(x)$ (cf. Thm.\ref{Gammmammmma}). 
The explicit computation is possible because the radius of convergence is actually ``\emph{small}'' (cf. 
Lemma \ref{Small radius}). 

During the accomplishment of this work we benefited of discussions with D.Barsky who suggested to apply the above 
computations about the $p$-adic Gamma function to the Morita's formula (cf. \cite{Morita}). Let $\Gamma_p^{0}$ 
be the analytic development of $\Gamma_p(T)$ at $T=0$. Let $\omega_p$ be the Teichm\"uller character. Then  
\begin{equation}
\log(\Gamma_p^0(T))\;\;=\;\;\lambda_0T+\sum_{m\geq 1}\frac{L_p(2m+1,\omega_p^{2m})}{2m+1}\;T^{2m+1}\;,\qquad |T|\leq 
|p|\;, 
\end{equation}
where $L_p(1+2m,\omega_p^{2m})$ is the value at $1+2m$ of the Kubota-Leopoldt's $p$-adic $L$-function corresponding to 
the Dirichlet character $\omega_p^{2m}$. Indeed, by considering the derivative of this expression 
we find 
\begin{equation}
g_0(T) \;\; = \;\; \lambda_0+\sum_{m\geq 1}L_p(1+2m,\omega_p^{2m})T^{2m}\;.
\end{equation}
The knowledge of the Newton polygon of $g_0(T)$ gives an upper bound on the absolute value of 
$L_p(1+2m,\omega_p^{2m})$, and an exact value of them in correspondence of the breaks  of the Newton polygon (cf. 
Corollaries \ref{Newton Polygon} and \ref{corollary abs value of zeta at 1+e}). 
By this these values together with the Newton polygon of $g_0(T)$ are related to the radius of convergence 
of the differential equation $\Gamma_p^{0}(T)'=g_0(T)\cdot\Gamma_p^{0}(T)$. 

Finally, from the fact that $\Gamma_p^0(T)$ is simultaneously a solution of every equation of the family 
\eqref{familyyfhgzi hujih jhgd lkjd } and of the ``limit'' differential equation, we deduce a family of congruences 
satisfied by the 
values $L_p(1+2m,\omega_p^{2m})$. Similar congruences were already known by \cite{Washington} and \cite{Barsky-cong}. 

The main point here is the fact that these congruences can be 
understood in terms of the above theory. In particular this last fact gives a new proof of the existence of a 
pole of the $p$-adic Zeta function, that finally appears as a consequence of the behavior of the Radius of 
convergence of the differential equation satisfied by the $p$-adic Gamma function. 

\subsubsection*{Note.} An idea of D.Barsky suggests the possibility to extend the above computation in order 
to define a new type of $p$-adic Gamma function(s) for any (abelian) number field extension, that 
should satisfy a difference equation together with the fact that their Taylor coefficients compute the 
$L$-functions of these number fields. D. Barsky is working on this direction.

\subsection{Berkovich spaces, $x$-radius and $y$-radius of convergence functions of a function 
around the diagonal}\label{intro : ber sp - x-y-rad}
The first sections (cf. sections \ref{section - notations}, \ref{norms}, \ref{Radius of convergence functions and analytic functions around the diagonal}, \ref{Linear Differential Equations}) 
are intended to be an introduction to the basic notions about Berkovich spaces \emph{adapted to the (local) 
theory of $p$-adic differential equation in one variable} and in particular to the study of the radius of convergence 
function of a differential equation.  
Most part of the notions can be extrapolated from the literature, these sections are intended to 
put them together as exhaustively as possible. 
It has been an occasion for us to list some not always elementary facts that are scattered and 
not always present in the literature. We have included all the proofs in order to make the paper self-contained. 
One finds the notion of \emph{maximal skeleton}  (cf. Section \ref{Skeleton}) 
together with its elementary properties. Roughly this is a finite tree inside 
the Berkovich space on which the radius of convergence function will be minimal. We introduce 
the notion of \emph{critical points} (cf. Def. \ref{crit pojnts defi -nition}), 
that are a finite family of semi-norms belonging to 
the maximal skeleton generalizing the \emph{Shilow boundary}. 
These points will be necessary to test the ``minimal values'' of the radius of 
convergence function of a differential module. 
We are able to provide criteria for the \emph{infinitesimality}, for the \emph{non degeneracy}, and for 
\emph{$\sigma$-compatibility} of families of 
automorphisms (cf. Lemmas \ref{Criterion of non degeneracy -fififi}, \ref{lemma: non so dare nome**}, 
\ref{compatibility condition lemma}) that can be 
tested on the finite set of critical points. In this particular context it results to be convenient to not make any distinction between points of type 1,2,3,4 as in \cite[section 1.4.4]{Ber}.
Our main aim in these first sections has been to introduce the ring of \emph{analytic functions $f(x,y)$ 
converging in a tube around the diagonal} (cf. section 
\ref{Convergent functions on neighborhoods of the diagonal}, and Def. 
\eqref{A delta dag}). Our intention has been to isolate the results that hold for all such functions 
from those that hold only for the Taylor solutions of differential equations.  
We hence define the \emph{radius of convergence function} of such a $f(x,y)$ without assuming it to be 
a solution of a differential equation (cf. Def. \ref{definition of x- and y- radius}). 
There are actually two notions of radius of $f(x,y)$: 
the \emph{$x$-radius} and the \emph{$y$-radius}, corresponding to the $x$-section and the $y$-section 
of its convergence locus as in section \ref{Elementary Stratifications and differential modules-yhn,ko} (cf. Def. 
\ref{definition of x- and y- radius}). 
The radius of convergence function of a differential equation 
that is universally used in the literature is the $x$-radius of its matrix Taylor solutions (i.e. the minimum 
$x$-radius of the solutions). It happens (cf. Lemma \ref{x-radius=y-radius for diff eq}) 
that the $x$-radius function coincides in this case with the $y$-radius 
function i.e. the minimum of the 
$x$-radius functions of the solutions coincides with the minimum of the $y$-radius functions of 
the solutions. Unlikely for a general analytic function $f(x,y)$ around the diagonal, these 
two radii functions do not coincide and do not satisfy any reasonable properties at the $K$-rational 
points (cf. examples of Ramark 
\ref{explicit examples of radius x-y-rad}), nor at the Shilow boundary, 
nor at the critical points (cf. 
the example of Remark \ref{example of non continuity of radius ofofoornhnhgj}). 
Conjecturally also the $x$-radius and the $y$-radius of a single solution of a differential equation 
must be equal (cf. Remark \ref{conjecture x-rad=y-rad for each solty}), 
but \emph{a priori} they satisfy different continuity properties 
(cf. Corollary \ref{corollary the y-rad of a solution is loc c const}). 
For this reason we treat them together along the text, since no additional difficulties arises. 
We then are able to prove the \emph{lower semi continuity} of the $x$-radius and $y$-radius of convergence functions 
of each such $f(x,y)$ (cf. Proposition \ref{lower semi-cont of R(f,-)}) 
outside the points of the Berkovich space $\mathscr{M}(X)$ 
corresponding to the points $|.|$ defined by an element of the completion 
$\widehat{K^{\mathrm{alg}}}$ of the algebraic closure of 
$K$.\footnote{i.e. the semi-norms $|.|$ defined by $|g|:=|g(t)|$, with $t\in \widehat{K^{\mathrm{alg}}}$.} 
In general the $x$-radius and the $y$-radius functions of a general $f(x,y)$ are not continuous.
We provide along the text some explicit examples describing the pathologies of the 
$x$-radius and $y$-radius of convergence functions 
of these functions of two variables (cf. examples of Remarks \ref{explicit examples of radius x-y-rad} and 
\ref{example of non continuity of radius ofofoornhnhgj}).  
Moreover we provide a \emph{criteria of continuity} for their $x$-radius and $y$-radius of convergence functions 
(cf. Proposition \ref{criteria of continuity of the radius ,n}). 
The proof of the continuity of the radius of convergence function of a differential equation
is essentially based on the original idea of Christol-Dwork (cf. 
\cite{Ch-Dw}), ideas that have been taken up also by \cite{DV-Balda}. 
In our restricted context our arguments are analogous, but quite different 
in the presentation with respect to that of \cite{DV-Balda}, because we are in a local framework, 
in one variable, and we strongly make use of the notion of the 
maximal skeleton, together with an accurate local description of the Berkovich space. Our proof is supposed to distinguish the obstructions to the continuity that are of topological nature from those that are of differential nature.

\subsubsection*{Note 1:}
Along the text we often interpret each semi-norm $|.|$ in the Berkovich 
space as the valuation at an $\Omega$-valued point $t\in X(\Omega)$: $|g|:=|g(t)|_{\Omega}$, where $\Omega$ is a 
convenient large complete valued field extension of $K$ (cf. Section \ref{Dwork generic points}). 
This permits to relate the classical language of 
Dwork (Christol,Robba,\ldots), and in particular the notion of generic point, to the language of Berkovich.     
Notice that there are two possible choices for the radius of the ``\emph{generic disc}'' centered at $t$. 
The first choice is $\rho_{t}^{\mathrm{gen}}$ which measures the distance of $t$ to $K^{\mathrm{alg}}$ 
(cf. Section \ref{Continuity of the derivation with respect to a Berkovich seminorm.}). 
The second choice is $\rho_{t,X}\geq \rho_{t}^{\mathrm{gen}}$ 
that measures the distance of $t$ to the complement of $X$ (cf. Section \ref{the number rho_cX}).
These two choices coincide on the points of the maximal skeleton, in particular they 
coincide in the classical cases investigated by Dwork (Christol, Robba,\ldots). 
As for the radius of convergence function these quantities depend only to the semi-norm attached to $t$, and not on 
the particular choice of $t$. So they define functions on the Berkovich space. 
The function $|.|\mapsto\rho_{|.|,X}$ is a continuous function on the Berkovich space (cf. Prop. 
\ref{continuity of rho--}), while 
$|.|\mapsto\rho_{|.|}^{\mathrm{gen}}$ is only upper semi-continuous (cf. Cor. 
\ref{rho gen is upper semi cont ertyui}). 
Both these functions are important since the radius of convergence function acquires more properties inside the disc 
$\mathrm{D}^-(t,\rho_{t}^{\mathrm{gen}})$ (cf. Prop.\ref{x-rad=y-rad on gengen}), but it has to be defined on the 
larger disc $\mathrm{D}^-(t,\rho_{t,X})$ in order to be continuous. 
See section \ref{Comments about the definition of Radius-ertyunedy9} for a more extensive discussion.

\subsubsection{Note 2:} In section \ref{Bell's polynomial} we give the link, 
proved in \cite{Bell}, between the non commutative Bell's polynomials and the 
iterated matrices of the connection.  

\subsection{Appendix \ref{Appendix : Local Frobenius Structure}: Analytic continuation along Frobenius}\label{appendix : intro}
The formalism mentioned in section \ref{intro generalized sigma modules sss} 
includes also the Frobenius structure of a differential module 
and permits to generalize it into a slight more general structure called \emph{local Frobenius 
structure}. Let $q=p^r$, $r\geq 1$. Let $J\subseteq ]\varepsilon,1[$ be a  sub-interval satisfying 
$J^q\cap J\neq \emptyset$. Then $\cup_{n\geq 0}J^{1/q^n}=J^*=]\varepsilon',1[$ for some $\varepsilon'>0$. 
A \emph{local Frobenius structure} on a differential module $(\M,\nabla)$ over $J$ is nothing but a 
gluing data between $(\M,\nabla)$ and its pull-back by Frobenius $\varphi^*\M$ over $J\cap J^{1/q}$. 
This structure can be understood as a generalized difference equation. We define the category of $(\varphi,d)$-modules 
over $J$ as the category of finite free $\a_K(J)$ modules together with a compatible action of a connection and a \emph{local} Frobenius structure. 
We are able to prove the following (cf. Th. \ref{theo. global frob = local frob} for a more precise statement):
\begin{theorem}[(cf. Theorem \ref{theo. global frob = local frob})]\label{Local Frobenius structure = Frob Structure}
If $K$ is spherically complete the category of $(\varphi,d)$-modules over $J$ is equivalent to the category of $(\varphi,d)$-modules 
over $J^*$. The same holds for the categories of $d$-modules with an unspecified Frobenius structure (i.e. the morphisms are not assumed to commute with $\varphi$).  
\end{theorem}
The quasi inverse of the restriction functor from $J^*$ to $J$ is a sort of ``\emph{analytic continuation functor}'' that permits to 
glue all the iterates pull-backs of a $d$-module over $J$ with a local Frobenius structure in order to define a $d$-module over $J^*$.
This result is important because, if $J$ is closed, the ring $\a_K(J)$ of analytic functions over the anulus $\{|T|\in J\}$ is principal, 
so some technical problems are much easier over $\a_K(J)$ than over $\a_K(J^*)$. As a corollary we are able to give another proof 
of the following theorem (cf. \cite[3.2-5]{Ch-Me-3}):
\begin{corollary}[(cf. Cor. \ref{MLFR is abelian})]\label{Intro- cor abelian over R_K sph comple}
If $K$ is spherically complete, then the category $d-\Mod(\R_K)^{(\varphi)}$ of differential modules over the Robba ring with an unspecified 
Frobenius structure is abelian.
\end{corollary}

\subsection*{\textsc{Acknowledgments}} 
Special thank go to D.Barsky for the very suggesting and helpful 
discussions we had that pushed us to perform the computations of 
section \ref{Gamma}. 
We thank Gilles Christol for 
useful discussions and constant encouragements. 
We also thank Bertrand Toen for helpful discussions.

\begin{center}
\textsc{Berkovich spaces and radii of convergence functions}
\addcontentsline{toc}{section}{\textsc{BERKOVICH SPACES AND RADII OF CONVERGENCE FUNCTIONS}}
\end{center}
\vspace{-0.5cm}
}\fi

\section{Notations}\label{Notation}
All rings are commutative with unit element. $\mathbb{R}$ is the field 
of real numbers, and 
$\mathbb{R}_{\geq 0}:=\{r\in\mathbb{R}\;|\; r\geq 0\}$.
For all field $L$ we denote its algebraic closure by 
$L^{\mathrm{alg}}$,
by $L[T]$ the ring of polynomial with coefficients in $L$, 
and  by $L(T)$ the fraction field of $L[T]$. If $L$ is valued, $\widehat{L}$ will be its completion.

In all the paper  $(K,|.|)$ will be a complete field of characteristic $0$ 
with respect to an ultrametric absolute value 
$|.|:K\to\mathbb{R}_{\geq 0}$ 
i.e. verifying $|1|=1$, $|a\cdot b|=|a||b|$,  
$|a+b|\leq\max(|a|,|b|)$ for all $a,b\in K$, and $|a|=0$ if and only if 
$a=0$. 
We denote by $|K|:=\{r\in\mathbb{R}_{\geq 0} 
\textrm{ such that }r=|t|,\;\exists t\in K\}$.
The ring of integers of $K$ will be 
$K^{\circ}:=\{x\in K\;,\;\textrm{such that }|x|\leq 1\}$, its maximal 
ideal $K^{\circ\circ}:=\{x\in K\;,\;\textrm{such that }|x|< 1\}$, 
and its residue field $\widetilde{K}:=K^{\circ}/K^{\circ\circ}$. 
We refer to \cite{Ber} for the definition of Berkovich spaces. 
But we chose to use \cite{Duc} as a uniform 
reference for Berkovich curves. 

Let $I\subset\mathbb{R}_{\geq 0}$ be an interval, and let $t\in K$. 
The annulus (resp. disc) centered at $t$ relative to $I$ will be 
denoted by 
\begin{equation}\label{eq :disk}
C(t,I)\;:=\;\{x\in\mathbb{A}^{1,\mathrm{an}}_K\;|\;|T-t|(x)\in I\}\;\qquad
\textrm{(resp. $D^-(t,\rho):=C(t,[0,\rho[)$, 
$D^+(t,\rho):=C(t,[0,\rho])$)}\;.
\end{equation} 
We will use the word \emph{annulus} if and only if 
$0<\inf I$ and $\sup I<+\infty$. Analogously a disc 
will always have a finite radius by definition.
When $\inf I=0$ and $0\notin I$, 
we say that $C(t,I)$ is either a \emph{punctured disc} if 
$\sup I<\infty$, and multiplicative group 
$\mathbb{G}_{m,K}^{\mathrm{an}}$ if $\sup 
I=+\infty$. 

The ring of \emph{analytic functions} $\O(C(t,I))$  over $C(t,I)$  is formed 
by power series 
\begin{equation}
\label{eq : Oct}
\O(C(t,I))\;:=\;
\Bigl\{
\sum_{i\in\mathbb{Z}}a_i (T-t)^i,\;a_i\in K, 
\textrm{ such that }\forall \rho\in I \textrm{ one has } \;\lim_{i\to\pm\infty}|a_i|\rho^i=0 
\Bigr\}\;.
\end{equation}
It is understood that if $0\in I$, 
then $a_i=0$ for all $i\leq -1$.
The ring $\mathcal{B}(C(t,I))$ of \emph{Bounded analytic 
functions} on $C(t,I)$ is formed by analytic functions in $\O(C(t,I))$ 
satisfying $\sup_{i\in\mathbb{Z}}|a_i|\rho^i \leq c < +\infty$ for all 
$\rho\in I$, 
where $c$ is a convenient constant depending on the power series.

A \emph{virtual} open (resp. closed) disc (resp. annulus) is a connected 
subset of $\mathbb{A}^{1,\mathrm{an}}_{K}$
which becomes a finite disjoint union of 
open (resp. closed) discs (resp. annuli whose orientations are 
preserved by $\mathrm{Gal}(\widehat{K^{\mathrm{alg}}}/K)$) 
after base change to $\widehat{K^{\mathrm{alg}}}$. 
The \emph{skeleton} of a virtual annulus  $C$ is the set
$\Gamma_{C}$ of points without neighborhoods in $C$ isomorphic to a virtual disc. 
These points form an interval.

For $c\in K$, and $\rho\geq 0$ we set 
$\xi_{c,\rho}(f):=
\sup_{n\geq 0}|\frac{f^{(n)}(c)}{n!}|_K\cdot \rho^n$, for 
all $f\in K[T]$. 
For all $x\in\mathbb{A}^{1,\mathrm{an}}_K$ there exists 
a complete valued field extension $\Omega/K$ such that 
$x=\pi_{\Omega/K}(x_{t,\rho})$, where 
$\pi_{\Omega/K}:\mathbb{A}^{1,\mathrm{an}}_\Omega\to
\mathbb{A}^{1,\mathrm{an}}_K$ is the canonical projection, and 
$t\in\Omega$. By an abuse, we write $x=x_{t,\rho}$ if no confusion 
is possible. The choice of $t\in\Omega$ and $\rho\geq 0$ is 
not unique (cf. \cite{NP-I} for more details).

\subsection{Quasi-smooth Berkovich curves}
\label{Quasi-smooth Berkovich curves}
Let $X$ be a quasi-smooth Berkovich curve in the sense 
of \cite{Duc}. This means that  $\Omega_X^1$ is  locally 
free of rank 
one, and it corresponds to a rig-smooth $K$-analytic 
curve in the rigid analytic terminology. 
A \emph{weak triangulation} of $X$ is a locally finite 
subset $S\subset X$ such that each connected 
component of $X-S$ is a virtual open annulus, or a virtual 
open disc.
We denote by $\Gamma_S$ the union of $S$ with all the 
skeletons of the virtual open annuli that are connected 
components of $X-S$. We call $\Gamma_S$ the 
\emph{skeleton} of $S$.

The \emph{analytic skeleton} $\Gamma_X$ of $X$ is the set of points $x$ without neighborhoods in $X$ isomorphic to a virtual open 
disc. It is not always the skeleton of a weak triangulation (e.g. if $X=\mathbb{P}^{1,\mathrm{an}}_K$, then 
$\Gamma_X= \emptyset$). The analytic skeleton $\Gamma_X$ is contained in the skeletons of all weak triangulations. 

Assume that $X$ is connected, and that $S$ is a weak triangulation of 
$X$. Then
\begin{enumerate}
\item If $S$ is the empty set, then $X$ is either a virtual open disc or 
annulus;
\item If $\Gamma_S$ is the empty set, then $X$ is a virtual open 
disc;
\item The curve $X-\Gamma_S$ is a disjoint union 
of virtual open discs. 
If for all connected component $Y$ of $X$, 
$\Gamma_S\cap Y\neq\emptyset$, those discs are 
all relatively compact in $X$. In this case the map 
\begin{equation}\label{eq : retraction}
\tau_{S}\;:\; X\;\xrightarrow{\quad}\;\Gamma_S
\end{equation}
which is the identity on $\Gamma_S$, and which sends the connected 
components of $X-\Gamma_S$ into their boundary in $\Gamma_S$ is 
a continuous open retraction.
\end{enumerate} 

\begin{theorem}[\protect{\cite{Duc}}]
\label{Thm: existence of tri}
Each quasi-smooth $K$-analytic curve admits a weak triangulation. 
\end{theorem}

\subsubsection{Log-linearity, directions, slopes.}
From the existence of weak triangulations, one deduces that every 
point of $X$ has a neighborhood that is uniquely arcwise connected. 
On such a subset, it makes sense to speak of the segment $[x,y]$ 
joining two given points $x$ and $y$, hence of convex subsets 
(see also \cite[Section~2.5]{BR}).
 
A subset $\Gamma$ of $X$ is said to be a \emph{finite (resp. locally finite) 
subgraph} of $X$ if there exists a finite (resp. locally finite)  family 
$\mathscr{V}$ of affinoid domains of $X$ that covers $\Gamma$ and 
such that, for every element $V$ of $\mathscr{V}$, $V$ is uniquely arcwise connected, 
and $\Gamma \cap V$ is the convex hull of a finite number of points.
By Theorem \ref{Thm: existence of tri}, 
$\Gamma_S$ is a locally finite graph.

We now want to define a notion of $\log$-linearity. 
To do so, we first need to explain how to measure 
distances. 
Let~$C$ be a closed virtual annulus over~$K$. Its 
preimage over~$\widehat{K^{\mathrm{alg}}}$ is a 
finite union of closed annuli. 
If $C(c,[R_{1},R_{2}])$ is one of them, we set 
$\textrm{Mod}(C) = R_{2}/R_{1}$.
This is well defined up to the choice of an 
orientation of $C$ (e.g. the inversion of the variable 
$T\mapsto T^{-1}$ changes the sign of the modulus), 
and it is invariant by isomorphisms of 
annuli preserving the orientation (cf. \cite{Duc}).

Let $I=[x,y]$ be a closed segment in $\mathbb{A}^{1,\mathrm{an}}_K$ 
containing only points of type $2$ or $3$. 
Then $I$ is the skeleton of a virtual closed annulus $C\subseteq 
\mathbb{A}^{1,\mathrm{an}}_K$, and we set 
$\ell(I) = \log(\textrm{Mod}(C))$.
Pushing these ideas further, one can show that it is possible to define 
a \emph{canonical length} $\ell$ for any closed segment inside a curve that contains 
only points of type $2$ or $3$ (see \cite[Proposition 4.5.7]{Duc}). 
The definition may actually be extended to any curve, 
see \cite[Corollaire 4.5.8]{Duc}).

\begin{definition}\label{defi:loglinear}
Let $X_{[2,3]}$ be the set of points of $X$ that are of type $2$ or 
$3$. 
Let $J$ be an open segment inside $X_{[2,3]}$ and identify it with a 
real interval. A map $f \colon J \to 
\mathbb{R}_{\geq 0}$ is said to be 
$\log$-linear if there exists $\gamma\in \mathbb{R}$ 
such that, for every $a < b \in J$, we have
\begin{equation}\label{eq : defsl}
\log(f(b)) - \log(f(a)) \;=\; \gamma\cdot \ell([a,b])\;.
\end{equation}
If $J$ is oriented as from $a$ to $b$ (resp. from $b$ to $a$), 
we set $\partial_Jf:=\gamma$ 
(resp. $\partial_Jf:=-\gamma$), 
and we call this number the slope of $f$ along the 
oriented segment $J$.
\end{definition}

We define an equivalence class on the segments out of a point $x$. 
We say that the open segment $]x,y[$ is equivalent to $]x,z[$ 
if there exists a third non empty open segment $]x,t[$ contained in 
$]x,y[\cap]x,z[$. 
We say that a class of germ of segments out of $x$ is a 
\emph{direction} out of $x$, or equivalently a 
\emph{germ of segment} out of $x$, or again a \emph{branch} out 
of $x$.  
A \emph{section of a branch} $b$ out of $x$ is any open 
connected subset $U$ of $X$ having 
$x$ at its boundary, such that $U\cup\{x\}$ is 
topologically a tree (no loops), and $b\subset U$. 
By Theorem \ref{Thm: existence of tri}, every branch out of $x$ admits a 
section  isomorphic to an open annulus.
There are well-defined notions of direct and inverse images of 
branches that correspond to the intuitive ones. 
Let $\varphi\colon X \to Y$ be a morphism of curves and let~$x\in X$ 
be a point such that~$\varphi^{-1}(\varphi(x))$ is finite. Then the 
image of a branch out of~$x$ is a branch out of~$\varphi(x)$ and the 
inverse image of a branch out of~$\varphi(x)$ is a union of branches 
out of some point $y\in \varphi^{-1}(\varphi(x))$. 
In particular for all germs of segments $b$ out of $x\in X$ 
we denote by $\mathrm{deg}(b)$ the number of germs of segments 
in $X_{\widehat{K^{\mathrm{alg}}}}$ over $b$.

If $f:X\to \mathbb{R}$ is a function which is 
$\log$-linear along a direction $b$ out of $x$, 
we denote its slope by
\begin{equation}
\partial_bf(x)\;,
\end{equation}
where by convention $b$ is always oriented as out of 
the point $x$.
\begin{definition}
Assume that $f$ is $\log$-linear along all directions out of $x$, 
and that $\partial_bf(x)=0$ for all, but a finite number of directions.
The Laplacian of $f$ is by definition the sum
\begin{equation}
dd^cf(x)\;:=\;\sum_{b}\mathrm{deg}(b)\cdot \partial_bf(x)
\end{equation}
of all the slopes for all germ of segment $b$ out of $x$. We say that 
$f$ is harmonic (resp. super-harmonic, sub-harmonic) at $x$ 
if $dd^cf(x)=0$ (resp. $dd^cf(x)\leq 0$, $dd^cf(x)\geq 0$).
\end{definition}

\subsubsection{Open boundary.}
\label{Open boundary.}
Assume that $K$ is algebraically closed. 
Let $O$ be a connected component of $X-S$. 
Then $O$ is either an open disc or annulus.

Assume that $O=D^-(0,R)$ is an open disc.
If $O$ is not relatively compact in $X$, then it is a connected 
component of it.

Assume that $O=C^-(0;]R_1,R[)$ is an annulus. If $O$ is not 
relatively compact in $X$, then it is either a connected component 
of it, or its closure $\overline{O}$ in $X$ is of the form 
$\overline{O}=O\cup\{x\}$, where $x\in \overline{O}-O$. So 
$x=\lim_{\rho\to R_1^+}x_{0,\rho}$ or
$x=\lim_{\rho\to R^-}x_{0,\rho}$.
Assume that $x$ lies on the $R_1$'s side.

In both cases (disc or annulus), for all $0<\varepsilon<R-R_1$,  
$S$ is again a weak triangulation of $X-C'$, where
$C':=C^-(0;[R-\varepsilon,R[)$. 
Moreover, if $Y$ is the connected component of $X$ containing $O$, 
then $Y-C'$ is connected too.
We call \emph{germ of segment at the open boundary of $X$} any 
unspecified germ of segment $]x_{0,R-\varepsilon},x_{0,R}[$ 
which is the skeleton of a non relatively compact annulus 
$C^-(0;]R-\varepsilon,R[)$ placed as above inside the curve $X$.

If $K$ is general we define the open boundary of $X$ as the image of the open boundary of $X_{\Ka}$.

As an example if $X$ is a virtual open disc, then its open boundary 
counts one element, if it is a virtual open annulus, then its open 
boundary counts two elements.

\if{
\subsubsection{Basic neighborhoods.}
\label{Basic neighborhoods.}
Let $x\in X$. 
\if{
Let $V(x)$ be the union of $\{x\}$ with all virtual open 
discs with boundary $x$. By Theorem \ref{Thm: existence of tri}, $V(x)$ is an analytic domain of $X$. 
For all direction $b$ out of $x$ not belonging to $V(x)$, 
there exists a virtual open annulus $C_b$ whose skeleton lies in the 
class $b$ (i.e. $C_b$ is a section of $b$). 
We call \emph{basic (open) neighborhood} of $x$ 
the following open subset 
$V(x)\bigcup_{b\notin V(x)}\left(\cup_b C_b\right)$.
More generally 
we say that an analytic neighborhood $U$ of $x$ is a 
\emph{basic neighborhood} if the set of points of $U$ that are not 
contained in some virtual open disc inside $U$ form a finite graph 
which is a finite union of segments of the form $[x,y[$ or $[x,y]$ 
(i.e. they form a star centered at $x$). 
}\fi
We say that a connected analytic neighborhood $U$ of $x$ is a 
\emph{basic neighborhood} if its analytic skeleton $\Gamma_U$ 
is a finite union of segments of the form $[x,y[$ or $[x,y]$ 
(i.e. they form a star centered at $x$) such that $]x,y[$ is the skeleton of a virtual open annulus $C$. 
The union is possibly empty, so the cases $\Gamma_U=\{x\}$ and $\Gamma_U=\emptyset$ are allowed, 
in these cases $U$ is either a virtual open disc, or it is a connected component of $X$.
}\fi

\subsubsection{$\Gamma_S$-coverings.}
Without loss of generality, we can assume that $X$ is connected.

If $\Gamma_S$ is empty, then $X$ is a virtual open disc with empty 
weak triangulation. In this case the unique 
$\Gamma_S$-covering of $X$ is by definition the trivial one given by the 
whole disc $\{X\}$.

We now assume $\Gamma_S\neq\emptyset$. 
In this case, since $X$ is connected, 
we have a retraction \eqref{eq : retraction} of $X$ onto 
$\Gamma_S$. 
For all $x\in\Gamma_S$ we consider a star-shaped 
open neighborhood $\Lambda_x$ of $x$ in $\Gamma_S$.\footnote{We mean that 
$\Lambda_x\subseteq\Gamma_S$ is a simply connected neighborhood of $x$ in $\Gamma_S$ (no loops), and moreover  
that $\Lambda_x-\{x\}$ is a finite disjoint union of segments $[x,y[$ out 
of $x$, all incident upon $x$, such that $]x,y[$ is the skeleton of a 
virtual open annulus in $X$ (this is possible since $\Gamma_S$ is locally finite).} 
Its inverse image 
\begin{equation}
Y_x\;:=\;\tau_S^{-1}(\Lambda_x)
\end{equation}
by the retraction \eqref{eq : retraction} is open analytic domain in $X$ 
such that $\Gamma_S\cap Y_x=\Lambda_x$. 

If, for all $x\in\Gamma_S$, we consider an analytic neighborhood 
$Y_x$ as above, then $X=\cup_{x\in\Gamma_S}Y_x$.

\if{
We now choose more carefully the family $Y_x$. 
If $x\in\Gamma_S-S$, then $x$ lies in a connected component $C$ of 
$X-S$ which is necessarily either a virtual open annulus or disc. 
In this case we set $Y_x:=C$.

For the points in $x\in S$ we choose $\Lambda_x$ small enough in 
order that for all distinct $x,x'\in S$ we have 
$\Lambda_x\cap\Lambda_{x'}=\emptyset$, this amounts to ask 
$Y_x\cap Y_{x'}=\emptyset$. 

With this choice of the family $\{Y_x\}_{x\in\Gamma_S}$, 
if $x_1,x_2\in\Gamma_S$ are distinct points, then 
$Y_{x_1}\cap Y_{x_2}$ is either empty or equal to an annulus 
whose skeleton is contained in $\Gamma_S-S$. 

We summarize all this in the following
}\fi
\begin{definition}[$\Gamma_S$-covering]
\label{Def : Gamma_S-cov}
Assume $X$ connected, and $\Gamma_S\neq\emptyset$.
A $\Gamma_S$-covering of $X$ is a covering formed by 
the family of all connected components $C$ of $X-S$, together with 
an open neighborhood of each point $x\in S$ of the form 
$Y_x=\tau_S^{-1}(\Lambda_x)$. 
We assume moreover that the intersection of three distinct elements of the 
covering is always empty.
In the sequel $C$ (resp. $Y_x$) will be endowed with the empty weak 
triangulation $S_C=\emptyset$ (resp. $S_{Y_x}=\{x\}$), we then have
$\Gamma_{S_C}=\Gamma_C$ 
(resp. $\Gamma_{S_{Y_x}}=\Lambda_x$). 
\end{definition}

If $x\in S$ then $Y_x$ is not necessarily a quasi-Stein\footnote{An example is given by an elliptic curve $X$ with good 
reduction. In this case a weak triangulation of $X$ is given by an individual 
point $S=\{x\}$ which is the unique point of $X$ without neighborhoods isomorphic to an analytic domain of the 
affine line. In this case $\Gamma_S=S=\{x\}=\Lambda_x$, 
and the unique open of the $\Gamma_S$-covering is $Y_x=X$. 
The same happens for $\mathbb{P}^{1,\mathrm{an}}_K$ with a triangulation $S=\{x\}$, with $x$ of type $2$ or $3$.} 
in the sense of 
Kiehl \cite{AB}. 
However, each point $x\in S$ admits a quasi-Stein  
open neighborhood which is 
obtained from some $Y_x=\tau_S^{-1}(\Lambda_x)$
by removing a finite number of virtual discs $D_1,\ldots,D_n$ with 
boundary $\{x\}$, and replacing them by some virtual open 
annuli at the open boundaries of 
$D_1,\ldots, D_n$ (cf. Section \ref{Open boundary.}).

From a $\Gamma_S$-covering we can always obtain a 
quasi-Stein covering of $X$ by shrinking the neighborhoods 
$Y_x$ in this way, and adding the remaining discs 
$D_1,\ldots, D_n$ to the covering.
\begin{definition}[quasi $\Gamma_S$-covering]
\label{Def : quasi Gamma_S-cov}
We call \emph{quasi $\Gamma_S$-covering} a covering of $X$ 
formed by quasi-Stein opens, which is 
obtained from a $\Gamma_S$-covering as above. 
\end{definition}
\if{
\begin{remark}\label{Rk : affinoid domains}
Replacing $D_i$ by one of its 
virtual closed sub-disc, replacing each $Y_x$ by an 
affinoid neighborhood of $x$ such that 
$Y_x=\tau^{-1}_S(\Lambda_x)$ with $\Lambda_x$ closed in 
$\Gamma_S$, and 
replacing each annulus of $X-S$ by a virtual closed sub-annulus, 
we obtain a covering of $X$ by affinoid domains.
\end{remark}
}\fi
\subsubsection{Extension of scalars.}
\label{Extension of scalars.}
Following \cite{NP-II}, we now quickly recall how to extend canonically weak 
triangulations by base change of $K$. More details can be found in \cite{NP-II}.

Let $S$ be a weak triangulation 
of $X$. The inverse image $S_{\widehat{K^{\mathrm{alg}}}}$ of $S$ 
in $X_{\widehat{K^{\mathrm{alg}}}}$ is easily seen to be a weak 
triangulation of $X_{\widehat{K^{\mathrm{alg}}}}$.

Assume now that $K$ is algebraically closed, and let $L/K$ be a 
complete valued field extension. 
Denote the canonical projection by $\pi_{L/K}:X_L\to X$.
By \cite{NP-II}, the fiber $\pi_{L/K}^{-1}(x)$ 
has a canonical point $\sigma_{L}(x)$ of type $2$ such that 
$\pi_{L/K}^{-1}(x)-\{\sigma_{L}(x)\}$ is a disjoint union of open 
discs in $X_L$ all having $\sigma_{L}(x)$ at their boundary.
Moreover the set $S_L:=\{\sigma_{L}(x),\;x\in S\}$ is a weak 
triangulation of $X_L$, and the projection $\pi_{L/K}$ induces an 
isomorphism between $\Gamma_{S_L}$ and $\Gamma_S$.

If $L$ is spherically complete and algebraically closed, the group 
$\mathrm{Gal}^{\mathrm{cont}}(L/K)$, 
of continuous automorphisms of $L$ over $K$, fixes each point of 
$\Gamma_{S_L}$, and permutes transitively the discs of 
$\pi_{L/K}^{-1}(x)-\{\sigma_L(x)\}$, and also the set of $L$-rational points of 
$\pi^{-1}_{L/K}(x)$. In particular these discs are all isomorphic.

\begin{lemma}[\protect{\cite[Prop. 2.1.7]{NP-III}}] 
There exists a complete valued field $\Omega/K$ 
containing isometrically all fields $\mathscr{H}(x)$, $x\in X$. 
In other words all the 
points of $X$ are $\Omega$-rational.\hfill$\Box$
\end{lemma}
\begin{notation}\label{Notation : Omega}
We now fix, once for all, a field $\Omega$ containing isometrically 
$\H(x)$, for all $x\in X$. Moreover we assume that $\Omega$ is 
algebraically closed, spherically complete, and that 
$|\Omega|=\mathbb{R}_{\geq 0}$. 
\end{notation}
Notation \ref{Notation : Omega} 
is not strictly necessary, but it simplifies the exposition. 
We notice that a quasi-smooth curve over $\Omega$ has 
no point of type $3$ nor $4$.

\begin{remark}\label{Rk : LIU --> QS or Proj}
We recall that by a result of Q.Liu \cite{Liu} every curve 
with finite 
genus (see \cite{NP-IV} for the definition of genus) 
over $\Omega$ is either projective, 
or quasi-Stein (in the sense of \cite{AB}).  
In particular, all analytic domains $X$ of 
$\mathbb{P}^{1,\mathrm{an}}_\Omega$ 
distincts from it are quasi-Stein of genus zero.
\end{remark}

\begin{definition}[Maximal and generic discs]
For all $x\in X$ we fix an $\mathscr{H}(x)$-rational point 
$t_x\in X_{\mathscr{H}(x)}$, lifting $x\in X$.
We call the 
connected component of 
$X_{\Omega}-\Gamma_{S_{\Omega}}$ 
containing $t_x$ the \emph{maximal disc} 
centered at $x$. We denote it by $D(x,S)$.
We call the disc $D(x)$ 
which is the connected component of 
$\pi_{\Omega/K}^{-1}(x)-
\{\sigma_{\Omega}(x)\}$ containing $t_x$ the \emph{generic disc} of $x$. 
By definition one has $D(x)\subseteq D(x,S)$.
\end{definition}
Concretely, 
if $x\notin\Gamma_S$ and if $D$ is the connected 
component of $X-\Gamma_S$ containing $x$ 
(which is necessarily a virtual open disk), then 
$D(x,S):=D_\Omega$, and if $x\in\Gamma_S$ then
$D(x,S):=D(x)$.

\if{
\subsubsection{Again on affinoid domains of the affine line.}\label{dependence on t,rho}
Chose a coordinate $T$ of the affine line. 
For all $x\in\mathbb{A}^{1,\mathrm{an}}_K$, 
we denote by $r(x)$ the radius of the generic disc $D(x)$ with respect 
to $T$. This is the radius of the point $x$ in the terminology of 
\cite{Ber}.

For all complete valued field extension $L/\mathscr{H}(x)$, and all  
$L$-rational point $t\in\pi_{L/K}^{-1}(x)$, we have a map 
$[0,+\infty[ \to\mathbb{A}^{1,\mathrm{an}}_K$ 
associating to $\rho$ the point $\pi_{L/K}(x_{t,\rho})$. 
The structure of the fiber $\pi^{-1}_{L/K}(x)$ implies that this map is 
constant of value $x$ for all $\rho\in[0,r(x)]$, and that it induces an 
homeomorphism of $[r(x),+\infty [$ with its image in 
$\mathbb{A}^{1,\mathrm{an}}_K$. 

This permits to obtain the following description. 
Assume that $K$ is algebraically closed.
Let $t,t'\in \mathbb{A}^{1,\mathrm{an}}_L$ be two $L$-rational 
points, and set 
$x:=\pi_{L/K}(x_{t,\rho})$ and $x':=\pi_{L/K}(x_{t',\rho'})$ 
for some $\rho\geq r(x)$ and $\rho'\geq r(x')$. 
Then $x=x'$ if and only if $\rho=\rho'$ and there exists $L'/L$ 
together with an automorphism  
$\sigma\in\mathrm{Gal}^{\mathrm{cont}}(L'/K)$ such that 
$|\sigma(t)-t|_{L'}\leq r(x)$.
}\fi

\subsection{Controlling graphs}
\label{Controlling graphs}
Let $\mathcal{T}$ be a set, and let $f:X\to \mathcal{T}$ be any 
function. Following \cite{NP-I} and \cite{NP-III} we now 
introduce a graph inside $X$ that controls 
the locus outside $\Gamma_S$ where $f$ is locally constant. 
\begin{definition}\label{Def : Gamma_S(sigma)}
Let $\Gamma_S(f)$ be the set of points $x\in X$ such that 
there is no open discs $D$ satisfying: 
(i) $x\in D$; 
(ii) $D\cap\Gamma_S=\emptyset$;
(iii) $f$ is constant on $D$.
We call $\Gamma_S(f)$ the controlling graph of $f$.
\end{definition}
By definition we have $\Gamma_S\subseteq\Gamma_S(f)$.
It is also easily seen that if $x\in\Gamma_S(f)-\Gamma_S$, then 
the segment connecting $x$ to $\Gamma_S$ is contained in 
$\Gamma_S(f)$. Indeed if a disc $D$ verifying $(i)$, $(ii)$, $(iii)$ contains 
one of the points of the segment, then it also contains $x$. 
This shows that if $X$ is connected, $\Gamma_S(f)$ is a 
connected sub-graph of $X$, containing $\Gamma_S$, and that 
$X-\Gamma_S(f)$ is a disjoint union of virtual open discs (on which $f$ is constant).
In fact $\Gamma_S(f)$ is also characterized by the fact that it is 
the smallest connected graph containing $\Gamma_S$ such that 
$f$ is locally constant outside 
$\Gamma_S(f)$.

\section{Linear differential equations}
\label{Linear differential equations}

Classical complex differential equations over Riemann 
surfaces have the nice property that their restriction to 
any disk is trivial. 
In other words if $G(z)$ is a $n\times n$ matrix 
whose entries are analytic functions over a 
disk $D=
\{z\in\mathbb{C}\;,\textrm{ such that }|z|<r\}$, 
then the differential system $\frac{d}{dz}(Y(z))=G(z)Y(z)$ admits a full 
basis of analytic solutions converging on $D$. 
Equivalently the radius of convergence of the
Taylor expansion of its solutions at $0$ 
is $r$ (i.e. as large as possible). 

This property is not verified over an ultrametric field $K$ 
as showed by the example of the equation $y'=y$ 
(cf. introduction) which is defined on the whole 
affine line and whose solution $\exp(T)$ does not 
converges on the whole line. 

In the ultrametric context 
one of the major invariants associated to 
$\Fs$ is the radius of convergence of its solutions, 
which is a function on $X$ constructed from the 
default of convergence of Taylor solutions of $\Fs$. 
It will play an important role on this article. 
We now recall its definition from 
\cite{Balda-Inventiones},  
\cite{Potentiel}, \cite{NP-II}. We consider those papers 
as general references.

\subsection{Radius of convergence function}
Let $X$ be a quasi-smooth $K$-analytic Berkovich curve. 
Recall that this is a locally ringed topological space with a 
structural sheaf $\O_X$ of analytic functions on $X$ (cf. 
\cite[Section 3.1]{Ber}) the topology of such a curve is 
described in detail in the book \cite{Duc}.\footnote{In 
this article we do not consider the G-topology (cf. 
\cite[Section 3.3]{Ber}). Opens of $X$ are subsets of 
$X$ that are open with respect to the Berkovich  
topology, coverings are collections of opens of $X$ 
whose union is $X$, and sheafs on $X$ 
are genuine sheafs over this topological space $X$.}

By \emph{differential equation} we mean 
a locally free $\O_X$-module of finite rank $\Fs$, 
together with a connection 
$\nabla:\Fs\to\Omega^1_X\widehat{\otimes}\Fs$ 
i.e. a map satisfying the 
Leibniz rule $\nabla(f\cdot m)=d(f)\otimes m + 
f\nabla(m)$ for all 
$f\in \O_X(U)$, $m\in \Fs(U)$, and for all open 
$U\subset X$ of the Berkovich topology.

Morphisms of differential equations are $\O_X$-linear 
maps commuting with the connections. 

We say that $\Fs$ is trivial if it is isomorphic to a direct 
sum of copies of the equation $d:\O_X\to\Omega^1$.

Consider now a point $x\in X$. 
By Remark \ref{Rk : LIU --> QS or Proj} the 
restriction of $\Fs_\Omega=\Fs\otimes_K\Omega$ 
to the disk $D(x,S)$ is a quasi-Stein space, 
so the locally free 
$\O_{D(x,S)}$-module $(\Fs_\Omega)_{|D(x,S)}$ 
corresponds to a finitely generated projective module 
over the ring 
$\O(D(x,S))=\{f:=\sum_{n\geq 0}a_n(T-t_x)^n\;,\;
\textrm{$f$ converges on $D(x,S)$}\}$ 
(cf. \eqref{eq : Oct}). 
Moreover by a result of Lazard (cf. \cite{Lazard}), the restriction of $(\Fs_\Omega)_{|D(x,S)}$ to 
$D(x,S):=D(t_x,S_\Omega)$ 
is free, since $\Omega$ is spherically complete 
(see also \cite[Théorème 4.40]{Christol-Book}). 

Now, we denote by 
\begin{equation}\label{eq : def of D(x,F)}
D(x,\Fs)
\end{equation} 
the largest disc centered at $t_x$, contained 
in $D(x,S)$, on which $\Fs$ is trivial. 
Such a disc is not empty by the Cauchy existence theorem 
\cite[Appendix III]{DGS}.
Let $T$ be a coordinate on $D(x,S)$. 
We denote by $\R^{\Fs}(x)>0$
the radius of $D(x,\Fs)$ in the coordinate $T$.
If $\rho_{S,T}(x)$ is the radius of $D(x,S)$, the ratio 
\begin{equation}
\R_{S,1}(x,\Fs)\;:=\;\R^{\Fs}(x)/\rho_{S,T}(x)
\end{equation}
 is independent of $T$ by the following 
\begin{lemma}\label{lem:isometry}
Let $R_{1},R_{2} >0$. Up to a translation, any 
$K$-isomorphism $\alpha : D^-(0,R_{1}) 
\xrightarrow[]{\sim} D^-(0,R_{2})$ 
is given by a power series of the form 
$f(T) = \sum_{i\ge 1} a_{i} T^i \in K[[T]]$,
with 
\begin{equation}
|a_{1}| \;=\;  R_{2}/R_{1}\;,\qquad |a_{i}| 
\leq R_{2}/R_{1}\;,\quad\forall\;i\geq 2\;.
\end{equation} 
In particular, it multiplies distances by the constant factor 
$R_{2}/R_{1}$: for any complete valued extension~$L$ of~$K$, and for all $t_1,t_2 \in D^-(0,R_{1})(L)$ we 
have $|\alpha(t_1)-\alpha(t_2)| = \frac{R_{2}}{R_{1}}\cdot |t_1-t_2|$.

As a consequence, such an isomorphism may only exist when 
$R_{2}/R_{1} \in |K^*|$.\hfill$\Box$
\end{lemma}
\begin{definition}[Radius of convergence]
We call $\R_{S,1}(x,\Fs)$ the radius of convergence of $(\Fs,\nabla)$ 
at $x\in X$. Following \cite{NP-I} and \cite{NP-III} its controlling 
graph (cf. Section \ref{Controlling graphs}) 
will be denoted by 
\begin{equation}
\Gamma_{S,1}(\Fs)\;:=\;\Gamma_S(\R_{S,1}(-,\Fs))\;.
\end{equation}
\end{definition}

\begin{theorem}[\protect{\cite{NP-I},\cite{NP-II}}]
\label{Thm : continuity}
The function $x\mapsto\R_{S,1}(x,\Fs)$ enjoys the following 
properties:
\begin{enumerate}
\item The controlling graph $\Gamma_{S,1}(\Fs)$ 
of $\R_S(-,\Fs)$ is locally finite;
\item $\R_{S,1}(-,\Fs)$ is a continuous function on $X$, 
which is independent of the choice of $t_x$, and of $\Omega/K$. 
It is moreover piecewise $\log$-linear along each segment in $X$, and 
its slopes belong to $\frac{1}{r!}\mathbb{Z}$, where $r$ is the local 
rank of $\Fs$;
\item Let $D$ be a virtual open disc which is a connected component of 
$X-\Gamma_S$. Let $C$ be any open annulus in $D$, and let 
$I:=\Gamma_C$ be its skeleton.
If $I$ is oriented as out of $D$, then the function 
$y\mapsto\R_{S,1}(y,\Fs)$ is $\log$-decreasing and $\log$-concave 
along $I$;
\item Let $C$ be a virtual open annulus which is a connected 
component of $X-S$. Let $I$ be its skeleton. 
Then $y\mapsto\R_{S,1}(y,\Fs)$ is $\log$-concave along $I$. 
\hfill$\Box$
\end{enumerate}
\end{theorem}
\begin{remark}\label{Remark : trivializing Radius over A^1}
If $X$ is an analytic domain of $\mathbb{A}^{1,\mathrm{an}}_K$, 
and let $T$ be a global coordinate on $X$. 
Let $\R^{\Fs}(x)$ and $\rho_{S,T}(x)$ be the radii of $D(x,\Fs)$ 
and $D(x,S)$ respectively in that coordinate.
Then $\R_{S,1}(x,\Fs)=\R^{\Fs}(x)/\rho_{S,T}(x)$.
The function $x\mapsto\rho_{S,T}(x)$ can be easily 
described: it is continuous, constant on each disc of 
$X-\Gamma_S$, and if $I\subseteq\Gamma_S$ is a 
segment, its slope along $I$ is $1$, if $I$ is oriented 
toward the point 
$\infty=\mathbb{P}^{1,\mathrm{an}}_K-\mathbb{A}^{1,\mathrm{an}}_K$.
In this case the function $\R^{\Fs}$ also enjoys the 
properties of Theorem \ref{Thm : continuity}.
\end{remark}
A morphism between differential equations is an 
$\O_X$-linear map 
commuting with the connections. 
We recall that its Kernel and  
Cokernel are locally free $\O_X$-modules on $X$ 
(cf. \cite[1.0.2]{NP-III}).  The category of 
differential equations is hence abelian. We denote by 
$\mathrm{Hom}^{\nabla}(\Fs,\Fs')$ the group of morphisms.

\begin{lemma}
\label{Lemma : alpha commutes with nabla loc then glob}
Let $\Fs,\Fs'$ be differential equations over $X$.
Let $\alpha:\Fs\to\Fs'$ be an $\O_X$-linear morphism. 
The following conditions are equivalent
\begin{enumerate}
\item $\alpha$ commutes with the connection;
\item For all connected component $X'$ of $X$, 
there exists a point $x\in X'$ of type $2$, $3$, or $4$, such that 
$\alpha(x):\Fs(x)\to\Fs'(x)$ commutes with the connections over 
$\H(x)$.
\end{enumerate}
\end{lemma}
\begin{proof}
We can assume $X$ connected. i) $\Rightarrow$ ii) is evident. 
Assume that ii) holds.
Consider a quasi $\Gamma_S$-covering $\{U_i\}_i$ of $X$ formed 
by quasi-Stein domains on which 
$\Fs$, $\Fs'$, and $\Omega^1_{X/K}$ are all free. 
So $\alpha$ commutes with the 
connections if and only if so does each $\alpha_{|U_i}$. 
Assume that the claim is proved for quasi-Stein curves, then it holds for the opens 
$U_i$ containing $x$. Now if $U_j$ is another open of the covering such that $U_i\cap U_j$ is not empty, 
the intersection always contains a point of type $2$, 
so $\alpha_{|U_j}$ commutes with the connection. Since $X$ is connected this proves that ii) $\Rightarrow$ i).

Hence we can assume $X$ quasi-Stein, and that $\Fs$, $\Fs'$, and 
$\Omega^1_X$ are all free. Let $d:\O(X)\to\O(X)$ be a derivation 
generating $\Omega^1_X$. 
In some bases, $\alpha$ is given by a matrix $H$, and we have to 
prove that it is solution of the differential equation $d(H)=GH$,  
associated to the differential module 
$\mathrm{Hom}_{\O_X}(\Fs,\Fs')$. 
Here $H$ is seen as a vector 
with entries in $\O(X)$ and $G$ as a square matrix with entries in 
$\O(X)$. We know that its specialization 
$H(x)$ over $\mathscr{H}(x)$ is solution of $d(H(x))=G(x)H(x)$.
Since 
$\O_{X,x}\subset\mathscr{H}(x)$ is injective, the equality 
$d(H)=GH$ holds over some affinoid neighborhood of $x$. Hence it 
also holds over $X$, by analytic continuation 
(cf. \cite[3.3.21]{Ber}).
\end{proof}

\subsection{Stratifications.}
\label{Stratifications.}
It follows from \cite{Grothendieck-10-exposes}, \cite{Berthelot-these}, 
\cite{Illusie-Cotangent-II}, \cite{Berthelot-preprint}, 
(and others), that the 
category of differential equations over $X$ is equivalent to that of 
so called \emph{stratifications}. 
We here quickly recall the definitions.

\subsubsection{}\label{eq : tubular n}
Let $\Delta:X\to X\times X$ be the diagonal closed immersion.
Let $\mathscr{I}\subset\O_{X\times X}$ be the ideal  corresponding to 
$\Delta$. Set $\mathscr{P}^n_{X/K}:=\O_{X\times X}/\mathscr{I}^{n+1}$, and 
$\mathscr{P}^\infty_{X/K}:=\varprojlim_{n}\mathscr{P}^n_{X/K}$.
We will say that the elements of $\mathscr{P}^n_{X/K}$ are 
convergent functions on the \emph{$n$-th 
infinitesimal neighborhood} of the diagonal, while those in
$\mathscr{P}^\infty_{X/K}$ corresponds to the  
\emph{formal neighborhood} of the diagonal. 

We now trivialize these notions locally on $X$.
Let $x\in X$ and let $U$ be an analytic 
neighborhood of $x$. 
Up to shrinking $U$ we may assume that $U$ is quasi-Stein, 
that $\Omega^{1}_U$ is free, and that there is an 
étale map 
$T:U\to \mathbb{A}^{1,\mathrm{an}}_K$. 
Let $p_1,p_2 : X \times X \to X$ be the canonical 
projections. Denote by 
$T_i:=p_i^*(T)=T\circ p_i\in\O_{U\times U}$.
The image of 
$T_1-T_2\in \mathscr{I}$ in 
$\Omega^1_X=\mathscr{I}/\mathscr{I}^2$ is the generator 
$dT$ of $\Omega^1_X$. 
Consider now $\O_{U\times U}=\O_U\widehat{\otimes}_K\O_U$ as 
an $\O_U$-ring via  
$p_2^*:\O_U\to\O_U\widehat{\otimes}_K\O_U$, 
$p_2^*(g)=1\otimes g$. 
Since $X$ is quasi-smooth, a classical 
computation shows that we have a (non canonical) isomorphism
$\mathscr{P}^n_{U/K}\simto\O_U[T_1-T_2]/(T_1-T_2)^{n+1}$
associating to $f\otimes g$ the Taylor expansion 
$(1\otimes g)\cdot\sum_{k=0}^n f^{(k)}(T_2)\frac{(T_1-T_2)^k}{k!}$,
where $f^{(k)}(T_2)$ means 
$\Delta^*\Bigl(\bigl((\frac{d}{dT})^kf\bigr)\otimes 1\Bigr)\in\O_U$. 
It follows that $\mathscr{P}^\infty_{U/K}\simto\O_U[[T_1-T_2]]$.

In this situation, we call \emph{tubular neighborhood of 
the diagonal of $U$} a Weierstrass domain of 
$U\times U$ of the form
$\mathcal{T}(U,T,R):=\{|T_1-T_2|\leq R\}$.

The ring $\mathscr{P}^\infty_{U/K}$, is a natural place 
where searching solutions of differential equations. In fact 
all differential solutions are trivialized by $\mathscr{P}^\infty_{U/K}$. 
We now recall quickly this fact. 

By the above local description  
of $\mathscr{P}_{U/K}^\infty$, the diagram
\begin{equation}
\xymatrix{
\O_U\ar[r]^-{p_1^*}&\mathscr{P}^\infty_{U/K}\\
K\ar[u]\ar[r]&\O_U\ar[u]_-{p_2^*}}
\end{equation}
provides a natural identification 
$(\Omega^1_{U/K}\widehat{\otimes}_{\O_U,p_1^*}
\mathscr{P}^\infty_{U/K})
\xrightarrow[]{\;\sim\;}
\Omega^1_{\mathscr{P}^\infty_{U/K}/\O_U}$,
where in the tensor product
$\mathscr{P}^\infty_{U/K}$ is considered as an $\O_U$-ring with 
respect to $p_1^*$, while the 
module of differentials 
$\Omega^1_{\mathscr{P}^\infty_{U/K}/\O_U}$ 
represents the $\O_U$-linear derivations of 
$\mathscr{P}^\infty_{U/K}$ with respect to the $\O_U$-ring structure 
given by $p_2^*$ (not $p_1^*$). By this identification the derivation 
$d/dT:\O_U\to\O_U$ corresponds to a $\O_U$-linear derivation of 
$\mathscr{P}^\infty_{U/K}\simto\O_U[[T_1-T_2]]$ which acts as $d/dT_1$. 

Now consider $\mathscr{P}^\infty_{U/K}$ as an $\O_U$-ring via 
$p_1^*$. The above identification of differentials  
permits to consider the \emph{scalar extension} to 
$\mathscr{P}^\infty_{U/K}$ of any differential equation 
$\Fs$ over $U$. 
In fact all differential equation  
become trivial over 
$\mathscr{P}^\infty_{U/K}$. More precisely there is an 
$\mathscr{P}^\infty_{U/K}$-linear isomorphism
\begin{equation}\label{eq : chi}
\chi\;:\;
\Fs\widehat{\otimes}_{\O_U, p_1^*}\mathscr{P}^\infty_{U/K}
\;\xrightarrow{\;\sim\;}\; 
\Fs\widehat{\otimes}_{\O_U, p_2^*}\mathscr{P}^\infty_{U/K}
\end{equation}
which commutes with structure of trivial differential equation over 
$\mathscr{P}^\infty_{U/K}$ of the right hand side.
Loosely speaking this means that $\Fs$ has a basis of solutions over
$\mathscr{P}^\infty_{U/K}$. Namely, up to shrinking $U$, $\Fs$ is free 
and we can consider an isomorphism $\Fs\simto\O_U^r$, i.e. a basis 
of $\Fs$. So the connection of $\Fs$ corresponds to 
a $K$-linear endomorphism $\nabla:\Fs\to\Fs$. 
We associate to it a matrix 
$G\in M_{r\times r}(\O_U)$ whose columns are the images of 
the chosen basis of $\Fs$. With these choices we have 
the following explicit expression of the matrix of $\chi$
\begin{equation}\label{eq : Y_chi}
Y_\chi\;:=\;\sum_{n\geq 0}G_n(T_2)\frac{(T_1-T_2)^n}{n!}\;\in\; 
GL_{r\times r}(\mathscr{P}^\infty_{U/K})\;,
\end{equation}
where $G_n(T_2)\in M_{r\times r}(\O_U)$ is inductively defined by 
the relations $G_0=\mathrm{Id}$, $G_1=G$, 
$G_{n+1}=\frac{d}{dT}(G_n)+G_n\cdot G$.
One verifies easily that $\frac{d}{dT_1}
(Y_\chi)=G(T_1)\cdot Y_\chi$,
because $G(T_1)\;=\; \sum_{k\geq 0}G^{(k)}(T_2)
(T_1-T_2)^k/k!$, 
and $G_{n+1}=
\sum_{k=0}^n\tbinom{n}{k}\cdot G_1^{(k)}\cdot 
G_{n-k}$.

\subsubsection{}\label{eq : rough estimation}
If $U$ is an affinoid domain, an induction gives 
$\|G_n\|_U\leq\max(\|d/dT\|_U,\|G\|_U)^n$,
where $\|\cdot\|_U$ is the sup-norm on $U$ 
(the norm of a matrix is by definition the sup of the norms of its entries), 
and $\|d/dT\|_U=\sup\{\frac{\|d/dT(f)\|_U}{\|f\|_U}, f\neq 0\}$ is 
the norm operator of $d/dT:\O_U\to\O_U$.
From this we deduce that $Y_\chi$ belongs to $\mathcal{T}(U,T,R)$ 
for all $R<\frac{\omega}{\max(\|d/dT\|_U,\|G\|_U)}$,
where $\omega:=\lim_n|n!|^{1/n}$.
If the valuation of $K$ is trivial on $\mathbb{Z}$, then 
$\omega=1$, otherwise $\omega=|p|^{\frac{1}{p-1}}$ 
where $p$ is the characteristic of $\widetilde{K}$.

\subsubsection{}\label{eq : cocycle condition stratification}
We now come back to the global curve $X$. We say that an open 
neighborhood $\mathcal{T}$ of the diagonal of $X\times X$ is 
\emph{admissible} if, for all $x\in X$, there exists an 
affinoid neighborhood 
$U$ of $x$ in $X$ as above, and a neighborhood 
$\mathcal{T}(U,T,R)$ of the diagonal of $U\times U$ 
such that 
$\mathcal{T}(U,T,R)\subseteq\mathcal{T}$.
A \emph{stratification} over $X$ is a locally free $\O_X$-module of finite 
rank $\Fs$ together with an $\O_{\mathcal{T}}$-linear isomorphism
\begin{equation}\label{eq : stratif :associated to F}
\chi\;:\;(p_2^*\Fs)_{|\mathcal{T}}\;
\xrightarrow[]{\;\sim\;}\;
(p_1^*\Fs)_{|\mathcal{T}}
\end{equation}
for some unspecified admissible neighborhood of the diagonal 
$\mathcal{T}$. 
The isomorphism $\chi$ is moreover 
subjected to the cocycle conditions:
\begin{enumerate}
\item If $p_{i,j}:X\times X\times X\longrightarrow X\times X$ is the 
projection on the $(i,j)$-factor, then over $p_{1,2}^{-1}(\mathcal{T})\cap p_{2,3}^{-1}(\mathcal{T})
\cap p_{1,3}^{-1}(\mathcal{T})$ one has 
$p_{1,2}^*(\chi)\circ p_{2,3}^*(\chi)=p_{1,3}^*(\chi)$.
\item $\Delta^*(\chi)=\mathrm{Id}_\Fs$ (here we canonically identify 
$\Delta^*p_{i}^*\Fs$ with $\Fs$).
\end{enumerate}

A morphism $\alpha:(\Fs_1,\chi_1)\to(\Fs_2,\chi_2)$ between 
stratifications is an $\O_X$-linear morphism $\alpha:\Fs_1\to\Fs_2$ 
such that $p_1^*\alpha\circ\chi_1=\chi_2\circ p_2^*\alpha$.
The following result is classical (e.g. see
\cite[Section 4.1.3]{Le-Stum-Book}).
\begin{theorem}[\protect{\cite{Berthelot-preprint}}]
The category of 
differential equations over $X$ is equivalent to 
the category of stratifications over $X$.
\end{theorem}
The above equivalence roughly goes as follows. If $(\Fs,\nabla)$ is a 
differential equation, the corresponding stratification consists in the 
same $\O_X$-module $\Fs$ together with the stratification 
whose local expression is given by \eqref{eq : chi}. 

If $\alpha: (\Fs_1,\nabla_1)\to(\Fs_2,\nabla_2)$ is a 
morphism of differential equations, then  $\alpha$ commutes also with 
the corresponding stratifications, so that the equivalence 
is the identity on the morphisms.

We now want to recover the connection $\nabla$ 
from the stratification $\chi$. 
This can be done by showing that the matrix 
$G:=\frac{d}{dT_1}(Y_\chi)\cdot Y_\chi^{-1}$ 
actually lies in $\O_U$. Or 
we can consider consider the reduction of $\chi$ in 
$\mathscr{P}_{U/K}^2\cong\O_U\oplus 
\mathscr{I}_U/\mathscr{I}_U^2$, 
and consider its retraction onto 
$\mathscr{I}_U/\mathscr{I}^2_U$.

\section{S-infinitesimal automorphisms}
\label{S-infinitesimal automorphisms}
Let $X$ be a quasi-smooth curve, and let $S$ be a weak triangulation 
of $X$. 

\begin{definition}
Let $\sigma:X\xrightarrow[]{\;\sim\;} X$ be a $K$-isomorphism. 
We say that $\sigma$ is an $S$-infinitesimal automorphism of $X$ 
if $\sigma_\Omega:X_\Omega\xrightarrow[]{\;\sim\;}X_\Omega$ 
induces an automorphism of each maximal disc 
$D(x,S)\subseteq X_\Omega$, for all $x\in X$. We often say infinitesimal instead of 
$S$-infinitesimal if no confusion is possible.
\end{definition}
Here and below $\sigma_\Omega$ means 
$\sigma\widehat{\otimes} \mathrm{Id}_\Omega$.
By definition, an $S$-infinitesimal automorphism fixes all the 
points of $\Gamma_S$. 
We denote by $\mathfrak{S}(X,S)$
the group of $S$-infinitesimal automorphisms of $X$.

An $S$-infinitesimal automorphism is often not $S'$-infinitesimal with 
respect to another weak triangulation $S'$.
However if $\Gamma_{S}=\Gamma_{S'}$, then $\sigma$ is 
$S$-infinitesimal if and only if it is $S'$-infinitesimal. 
This is because $D(x,S)=D(x,S')$ for all $x\in X$. A similar consideration gives the following

\begin{lemma}\label{Lemma : localization}
Let $\sigma$ be an $S$-infinitesimal of $X$. Then:
\begin{enumerate}
\item If $Y\subseteq X$ is a connected 
analytic domain admitting a \emph{non empty} weak 
triangulation $S_Y$ such that $\Gamma_S\cap Y=\Gamma_{S_Y}$, 
then $\sigma$ induces an 
$S_Y$-infinitesimal automorphism of $Y$.  
\item
If $Y$ is a connected component of $X-S$ or $X-\Gamma_S$ 
(necessarily a virtual open disc or annulus) together 
with the empty weak-triangulation 
$S_Y=\emptyset$, 
then $\sigma$ induces an $S_Y$-infinitesimal 
automorphism of $Y$.\hfill$\Box$
\end{enumerate}
\end{lemma}
\begin{remark}
Lemma \ref{Lemma : localization} 
applies to all opens of a 
$\Gamma_S$-covering of $X$  (cf. Def. \ref{Def : Gamma_S-cov}).
\end{remark}
\begin{remark}\label{Rk : reduction to Kalg}
An automorphism $\sigma$ is $S$-infinitesimal if and only if 
$\sigma_{\widehat{K^{\mathrm{alg}}}}$ is 
$S_{\widehat{K^{\mathrm{alg}}}}$-infinitesimal.
\end{remark}
\subsection{The function $\R_S(-,\sigma)$ and its controlling 
graph.}
We now define a function 
$\R_S(-,\sigma):X\to\mathbb{R}_{\geq 0}$
that controls how an infinitesimal automorphism $\sigma$ is close to 
the identity. For this we need the following straightforward 
consequence of Lemma \ref{lem:isometry}:
\begin{lemma}\label{Lemma: D(t,sigma)}
Let $\sigma$ be any automorphism of an open disc $D$. Let 
$t\in D$ be a rational point. Then
\begin{enumerate}
\item There exists a smallest closed disc\footnote{Notice that 
$D^+(t,\sigma)$ is allowed to be equal to the individual point $\{t\}$.} 
$D^+(t,\sigma)\subset D$ %
centered at $t$ which is globally stable by $\sigma$. 
Moreover $D^+(t,\sigma)$ is the smallest closed disc containing $t$ 
and $\sigma(t)$;
\item Each (open or closed) disc $D'$ satisfying $D^+(t,\sigma)\subseteq D'\subseteq D$ 
is globally stable by $\sigma$.
\item For all disc $D'$ as in ii), the 
annulus $C:=D-D'$ is globally stable under $\sigma$.
\hfill$\Box$
\end{enumerate}
\end{lemma}
\begin{remark}
\label{Rk : restriction of sigma to the wedge of D}
Let $C$ be an annulus as in Lemma \ref{Lemma: D(t,sigma)}. 
Then $\sigma$ does not necessarily 
induce a $S$-infinitesimal automorphism on $C$ with respect to the empty weak triangulation. 
As an example, if $\sigma$ is the multiplication by $q\in K$ with $|q-1|=1$ acting on the open unit disc $D$. 
Then $\sigma$ is infinitesimal with respect to the empty weak triangulation of 
$D$ because $D(x,S)=D$ for all $x\in D$. In this case 
$D^+(0,\sigma)$ is reduced to $\{0\}$ and no discs in 
$C\subset D-\{0\}$ are stable by $\sigma$.
\end{remark}
Let $X$ be a quasi smooth curve. For all $x\in X$, we denote by 
\begin{equation}\label{eq : def of D^+(x,sigma)}
D_S^+(x,\sigma)
\end{equation}
the smallest closed disc in $D(x,S)$ 
containing $t_x$ and $\sigma_\Omega(t_x)$.
Let $T$ be a coordinate on $D(x,S)\subseteq X_\Omega$. 
Denote by  $\R^\sigma(x)\geq 0$ 
the radius of $D_S^+(x,\sigma)$ in the coordinate $T$.
If $\rho_{S,T}(x)$ is the radius of $D(x,S)$ with respect to the same coordinate, the ratio 
\begin{equation}\label{eq : R_S(x,sigma)}
\R_S(x,\sigma)\;:=\;\R^{\sigma}(x)/\rho_{S,T}(x)
\end{equation}
is independent of $T$ by Lemma \ref{lem:isometry}. 
Moreover, by section \ref{Extension of scalars.}, the action of 
$\mathrm{Gal}^{\mathrm{cont}}(\Omega/K)$ on 
$\pi_{\Omega/K}^{-1}(x)$ is transitive on the $\Omega$-rational 
points, and it preserves the modulus of the discs.
Since it commutes with $\sigma_\Omega$, 
the function $\R_S(-,\sigma)$ is independent of the choice of $t_x$ and $\Omega$.

\begin{definition}
We call $\R_S(-,\sigma)$ the radius of $\sigma$. Its 
controlling graph will be denoted (with an abuse)
by 
\begin{equation}
\Gamma_S(\sigma)\;:=\;\Gamma_S(\R_S(-,\sigma))\;.
\end{equation}
\end{definition}
\begin{remark}
If $x\in\Gamma_S$, then $\sigma(x)=x$. Nevertheless 
$\sigma_\Omega$ is not necessarily the identity on $D(x,S)$, so the 
function $\R_S(-,\sigma)$ is not necessarily equal to $0$ on the 
points of $\Gamma_S$.
\end{remark}
\subsection{Compatibility with the restriction to 
an analytic domain.}
\label{Section : Localization works well}
Let $Y$ be an analytic domain of $X$, together with a 
weak triangulation $S_Y$. 
Assume that $\sigma$ induces by restriction 
an automorphism $\sigma_{|Y}$ of $Y$. 
Remark \ref{Rk : restriction of sigma to the wedge of D} 
shows that $\sigma_{|Y}$ is not necessarily 
$S_Y$-infinitesimal.
It may also arises that $\sigma_{|Y}$ is 
$S_Y$-infinitesimal, but the restriction of the function 
$\R_S(-,\sigma)$ to the set $Y$ does not coincide 
with the function $\R_{S_Y}(-,\sigma_{|Y})$. 
The following remark is an example of this phenomenon.
\begin{remark}
Let $X=D^-(0,1)$ be the open unit disc 
together with the empty weak-triangulation $S_Y$, and 
let $\sigma_q$ be the multiplication by $q\in K$ with 
$|q-1|<1$. Then $\R_S(x,\sigma_q)=|q-1||T|(x)$ for all 
$x\in X$. If $Y=D^-(0,\rho)$ with $\rho<1$, together 
with the empty weak-triangulation $S_Y=\emptyset$, 
then $\sigma_q$ restricts to $Y$ 
and it induces an $S_Y$-infinitesimal automorphism 
$(\sigma_{q})_{|Y}$ of $Y$. 
But for all $x\in Y$ we have 
$\R_{S_Y}(x,(\sigma_{q})_{|Y})=\R_S(x,\sigma_q)/\rho$.
\end{remark}
\begin{proposition}
\label{Proposition : Localization works well}
Let $(Y,S_Y)$ be a pair as in 
point i) or ii) of 
Lemma \ref{Lemma : localization}, then $\sigma$ 
induces an $S_Y$-infinitesimal automorphism of $Y$, 
and
\begin{equation}\label{eq : restriction of R_S(-,sigma)}
\R_S(y,\sigma)_{|Y}\;=\;
\R_{S_Y}(y,\sigma_{|Y})\;.
\end{equation}
Moreover
\begin{equation}\label{eq : ytt}
\Gamma_S(\sigma)\cap Y\;=\;\Gamma_{S_Y}(\sigma_{|Y})\;.
\end{equation}
\end{proposition}
\begin{proof}
For all $y\in Y$ we have $D(y,S_Y)=D(y,S)$. 
This together with Def. \ref{Def : Gamma_S(sigma)} imply \eqref{eq : ytt}.
\end{proof}

\begin{remark}
Proposition \ref{Proposition : Localization works well} 
applies in particular to all opens of a 
$\Gamma_S$-covering of $X$ 
(cf. Def. \ref{Def : Gamma_S-cov}), but not to quasi 
$\Gamma_S$-coverings.
Indeed point iii) 
of Lemma \ref{Lemma: D(t,sigma)} ensures the 
existence of a quasi 
$\Gamma_S$-covering of $X$ stable by $\sigma$. 
Notice however that the action of $\sigma$ on the opens of such a 
covering is not necessarily $S$-infinitesimal as showed in the 
Remark \ref{Rk : restriction of sigma to the wedge of D}.
\end{remark}

%
%

\subsection{A finiteness result}
The aim of this section is to prove 
the following analogue 
of Theorem \ref{Thm : continuity} :
\begin{theorem}\label{Thm : finiteness sigma}
The function $x\mapsto\R_S(x,\sigma)$ 
enjoys the following properties: 
\begin{enumerate}
\item $\R_{S}(x,\sigma)$ is a continuous function 
on $X$. It is moreover piecewise $\log$-linear along 
each segment in $X$, and its slopes belong to 
$\mathbb{Z}$;
\item If $C$ is a connected component of $X-S$ 
(either a virtual open disc or annulus) there exist
an analytic 
function $f_\sigma\in\O(C)$, and a real number 
$\alpha\in\mathbb{R}_{\geq 0}$, such that 
\begin{equation}
\R_S(x,\sigma)\;=\;\alpha\cdot|f_\sigma|(x)\;.
\end{equation}
for all $x\in C$. 
In particular $\R_S(x,\sigma)$ is harmonic outside 
$S$.
\item Its controlling graph $\Gamma_S(\sigma)$ is 
locally finite. Moreover, if $\sigma\neq\mathrm{Id}_X$, 
the end points of 
$\Gamma_S(\sigma)$ that do not belong to 
$\Gamma_S$ are exactly the rigid points of $X$ that are 
fixed by $\sigma$. 
\end{enumerate}
\end{theorem}
\begin{remark}\label{Rk : convexity of R_S(-,sigma)}
In analogy with Theorem \ref{Thm : continuity} 
one has the following immediate 
consequences:
\begin{enumerate}
\item[v)] 
Let $D$ be a virtual open disc which is a connected component of 
$X-\Gamma_S$. Let $C$ be any open annulus in $D$, and let 
$I:=\Gamma_C$ be its skeleton.
If $I$ is oriented as out of $D$, then the function 
$y\mapsto\R_{S}(y,\sigma)$ is $\log$-increasing and $\log$-convex 
along $I$;

\item[vi)] Let $C$ be a virtual open annulus which is a connected 
component of $X-S$. Let $I:=\Gamma_C$ be its skeleton. 
Then $y\mapsto\R_{S}(y,\sigma)$ is $\log$-convex along $I$. 
\end{enumerate}
\end{remark}
\begin{proof}[\protect{Proof of Theorem 
\ref{Thm : finiteness sigma}}]
The claims hold for $X$ if and only if they hold for 
$X_{\widehat{K^{\mathrm{alg}}}}$ (cf. 
Remark \ref{Rk : reduction to Kalg}). So, without loss of generality, we 
may assume that $K=\widehat{K^{\mathrm{alg}}}$, 
and that $X$ is connected.
\begin{lemma}
\label{Thm : Finiteness : the case of the line}
Theorem \ref{Thm : finiteness sigma} 
holds if $X$ is an analytic domain of 
$\mathbb{A}^{1,\mathrm{an}}_K$.
\end{lemma}
\begin{proof}
Let $T:X\hookrightarrow\mathbb{A}_K^{1,\mathrm{an}}$ 
be a global coordinate on $X$. 
Set $\delta_{\sigma,T}:=T\circ\sigma-T\in\O(X)$.
The value of $(\delta_{\sigma,T})_\Omega$ at $t_x$ is 
$T(\sigma_\Omega(t_x))-T(t_x)$. For all $f\in \O_X$ we have 
$|f|(x)=|f_\Omega|(t_x)$, so the norm $|\delta_{\sigma,T}|(x)$ 
equals the distance $|\sigma(t_x)-t_x|_\Omega$ in the coordinate $T$. 
By Lemma \ref{Lemma: D(t,sigma)}, we then have 
$\R^\sigma(x)=|\delta_{\sigma,T}|(x)$. 
If $\rho_{S,T}(x)$ denotes the radius of $D(x,S)$ in the coordinate 
$T$, the claim follows from the properties of $x\mapsto\rho_{S,T}(x)$
(cf. Remark \ref{Remark : trivializing Radius over A^1}) since
$\R_S(x,\sigma)=|\delta_{\sigma,T}|(x)/\rho_{S,T}(x)$.
\end{proof}

Let now $X$ be a general quasi-smooth curve.
We consider a $\Gamma_S$-covering of $X$. If $O$ is an open of the 
covering we call $S_O$ a weak triangulation of $O$ as in 
Definition \ref{Def : Gamma_S-cov}. 
By Proposition \ref{Proposition : Localization works well} 
we are reduced to proving the claim for an individual 
open $O$ of the covering. 
\if{
By Proposition \ref{Proposition : Localization works well}
one has $\R_S(-,\sigma)_{|O}=\R_{S_{O}}(-,\sigma_{|O})$, and 
$\Gamma_S(\sigma)\cap O=\Gamma_{S_O}(\sigma_{|O})$.
Now the intersection of three distinct opens of the covering is empty, 
and the intersection of two of them is an open annulus $C$ 
such that $\Gamma_S\cap C=\Gamma_C$. 
The localization to $C$ behave again well, 
by Proposition \ref{Proposition : Localization works well}.
This implies that we are reduced to prove the result for an individual 
open of the covering.
}\fi

If $O$ is a connected component of $X-S$ 
(which is necessarily either an open annulus or disc) 
Theorem \ref{Thm : finiteness sigma} is a consequence of 
Lemma \ref{Thm : Finiteness : the case of the line}.
Claims ii) and iii) are then clear. 
Since a germ of segment always belongs to a connected component of $X-S$, the claims about 
the $\log$-linearity and the slopes are also consequence of Lemma \ref{Thm : Finiteness : the case of the line}.

It remains to prove the continuity and local finiteness of $\Gamma_S(\sigma)$ at a point $x\in S$.
We have to find a neighborhood of $x$ in $X$ of the form 
$Y_x=\tau^{-1}_S(\Lambda_x)$ (cf. Def. \ref{Def : Gamma_S-cov}) on which the claims hold. 
We can exclude points of type $3$ since, by \cite[Thm.4.3.5]{Duc}, 
$Y_x$ can be chosen either as a closed disc containing $x$, or as a closed annulus containing $x$ in its skeleton 
$\Gamma_C$. 

\if{\begin{lemma}[\protect{\cite[Thm.4.3.5]{Duc}}]
\label{Lemma: neighb of a point of type 3}
Let $x\in S$ be a point of type $3$. 
Then we can choose the neighborhood $Y_x$ of $x$ 
either as a closed disc containing $x$ in its interior, or a closed 
annulus containing $x$ in its skeleton $\Gamma_C$. \hfill$\Box$
\end{lemma}
\begin{proof}
By \cite[Théorème 4.3.5]{Duc}, $x$ admits a neighborhood $C$ in $X$ 
which is either a closed disc containing $x$ in its interior, or a closed 
annulus containing $x$ in its skeleton $\Gamma_C$. 
More precisely if $x$ is an end point of $\Gamma_S$, and if 
$x\notin\partial X$, then $x$ admits a neighborhood which is 
isomorphic to a closed disc $D$ containing $x$ in its interior. 
Hence $\Gamma_S\cap D$ is the skeleton of a weak 
triangulation of $D$, so we may choose $Y_x=D$.

If $x\in\partial X$, then $C$ is a closed annulus such that 
$\Gamma_C\subseteq\Gamma_S$ and $x$ lies in its 
boundary. In this case $\Gamma_S\cap C$ is the skeleton of a weak 
triangulation of $C$. So we can assume $Y_x=C$.

In the other cases we can choose $C$ as an open annulus such that 
$\Gamma_C\subseteq\Gamma_S$ and such that $x$ is an internal 
point of $\Gamma_C$. As above $\Gamma_S\cap C$ is the skeleton of 
a weak triangulation on $C$, so we can assume $Y_x=C$.
\end{proof}

\begin{lemma}\label{Lemma : finiteness around type 3}
Let $x\in S$ be a point of type $3$. Let $Y_x$ be a disc or annulus
as in Lemma \ref{Lemma: neighb of a point of type 3}. 
Then the restriction 
$\R_{S_{Y_x}}(-,\sigma_{|Y_x})=\R_S(-,\sigma)_{|Y_x}$ 
to $Y_x$ verifies the claims of Theorem \ref{Thm : finiteness sigma}. 
\end{lemma}
\begin{proof}
By Lemma \ref{Lemma: neighb of a point of type 3} 
there exists a weak triangulation of $Y_x$ with 
skeleton $\Gamma_S\cap Y_x$, so 
by Proposition 
\ref{Proposition : Localization works well} we can 
restrict to $Y_x$. The claim the follows directly from Lemma 
\ref{Thm : Finiteness : the case of the line}.
\end{proof}
}\fi

We now study the behavior of $\R_S(-,\sigma)$ in the neighborhood of 
a point of type $2$ of $S$. To continue our proof we need the 
following results:

\begin{theorem}[\protect{\cite{NP-II}}]
\label{thm:bonvois}
Assume $K=\Ka$.
Let $x$ be a point of~$X$ of type $2$. 
Let $b_{1},\dotsc,b_{n},c$ be distinct directions out of~$x$. 
Let $N$ be a positive integer.  There exists an affinoid 
neighborhood~$Z$ of~$x$ in~$X$, a quasi-smooth affinoid 
curve~$Y$, an affinoid domain~$W$ of 
$\mathbb{P}^{1,\textrm{an}}_{K}$ 
and a finite \'etale map $\psi \colon Y \to W$ such that
\begin{enumerate}
\item $Z$ is isomorphic to an affinoid domain of~$Y$ and~$x$ lies in 
the interior of $Y$;
\item\label{i:deg} the degree of~$\psi$ is prime to~$N$;
\item\label{i:pointx} $\psi^{-1}(\psi(x))=\{x\}$;
\item\label{i:compx} almost every connected component of 
$Y\setminus\{x\}$ is an open unit disc with boundary~$\{x\}$;
\item\label{i:compfx} almost every connected component of 
$W-\{\psi(x)\}$ is an open unit disc with 
boundary~$\{\psi(x)\}$;
\item\label{i:compiso} for almost every connected component~$C$ of 
$Y-\{x\}$, the induced morphism $C \to \psi(C)$ is an 
isomorphism;
\item\label{i:isobi} for every $i \in \{1,\ldots,n\}$, 
the morphism~$\psi$ induces an isomorphism between a section 
of~$b_{i}$ and a section of~$\psi(b_{i})$ and we have 
$\psi^{-1}(\psi(b_{i})) \subseteq Z$;
\item\label{i:branchec} $\psi^{-1}(\psi(c)) = \{c\}$. 
\hfill$\Box$
\end{enumerate}
\end{theorem}

\begin{lemma}[\protect{\cite{NP-II}}]
\label{Lemma : trivial covering}
Let $K$ be an algebraically closed field.
Let $Z$ be a quasi-smooth $K$-analytic curve. 
Let $\psi: X \to Z$ be a finite morphism. 
Let $x\in X$ be a point of type $2$ or $3$. 
Assume that $d = [\mathscr{H}(x) : \mathscr{H}(\psi(x))]$ 
is prime to $p$.\footnote{If 
the residual field $\widetilde{K}$ has characteristic $0$, 
then $p=1$ and this condition is always satisfied.} 
Then every connected component of 
$\pi_{\Omega/K}^{-1}(\psi(x))\setminus\{\sigma_\Omega(\psi(x))\}$ 
is a disc and the morphism $\psi_\Omega$ induces a trivial cover of 
degree $d$ over it.
\hfill$\Box$
\end{lemma}

Let $x$ be a point of $S$ of type $2$. 
By Theorem \ref{thm:bonvois} there is an affinoid domain $V$ of $X$, 
containing $x$, and an affinoid domain $W'$ of 
$\mathbb{A}^{1,\mathrm{an}}_K$, together with a finite étale map 
$\psi:V\to W'$ such that 
\begin{enumerate}
\item[(a)] $V-\{x\}$ and $W'-\{\psi(x)\}$ are both  
disjoint union of open discs;
\item[(b)] $V\cap \Gamma_S=\{x\}$;
\item[(c)] $\psi$ induces an isomorphism $D\simto\psi(D)$ 
on each disc $D$ in $V-\{x\}$. 
\end{enumerate} 

We endow $V$ with the weak triangulation $S_V:=\{x\}$. We 
can localize on $V$ as in Proposition \ref{Proposition : Localization works well}.

We may also assume that the degree 
$d:=[\mathscr{H}(x) : \mathscr{H}(\psi(x))]$ is prime to $p$, so that 
the map $\psi$ induces an isomorphism on the generic 
discs $\psi_\Omega:D(x)\simto D(\psi(x))$, 
by Lemma \ref{Lemma : trivial covering}. 

Since $x\in S$ we have $D(x)=D(x,S)$. Moreover by choosing on $W'$ 
the weak triangulation given by $S':=\{\psi(x)\}$ we also have 
$D(\psi(x))=D(\psi(x),S')$. So $\psi$ induces an isomorphism on the 
maximal discs 
$\psi_\Omega:D(x,S)\xrightarrow[]{\sim} D(\psi(x),S')$.
In the other cases the maximal disc is just the 
connected component of $V-\{x\}$ 
(resp. $W'-\{\psi(x)\}$) containing the point. 
So, by point $(c)$,  
we also have an isomorphism  for all point $y\in V$: 
\begin{equation}\label{eq : psi: D(x,S)-->D(psi(x),S')}
\psi_\Omega\;:\;D(y,S)\xrightarrow{\;\sim\;}D(\psi(y),S')\;.
\end{equation} 

Now we consider $\psi_\Omega$ as a simultaneous coordinate 
on each $D(y,S)$, for all $y\in V$, and we define
\begin{equation}
\delta_{\sigma,\psi}\;:=\;\psi\circ\sigma-\psi\;\in\; \O(V)\;.
\end{equation}
The map $y \mapsto |\delta_{\sigma,\psi}|(y)$ is continuous at $x$, it 
is locally constant outside a finite sub-graph of $V$,
and it controls the distance $|\sigma(t_y)-t_y|_\Omega$ measured 
with the coordinate $\psi_\Omega$. Let $\rho_{S,\psi}(y)$ be the 
radius of $D(y,S)$ measured with the same coordinate on it. 
Then by Lemma \ref{lem:isometry} we have
\begin{equation}\label{eq : R_S(y,sigma)=delta/rho_psi}
\R_S(y,\sigma)\;=\;|\delta_{\sigma,\psi}|(y)/\rho_{S,\psi}(y)\;.
\end{equation}
The function $y\mapsto\rho_{S,\psi}(y)$ is constant on $V$ since the 
radius of $D(y,S)=D(y,S_V)$ 
with respect to $\psi$ coincides by definition 
with the radius of $D(\psi(y),S')$. 
This proves the continuity of $\R_S(-,\sigma)_{|V}=
\R_{S_V}(-,\sigma_{|V})$ on $V$, and the finiteness of 
$\Gamma_S(\sigma)\cap V=\Gamma_{S_V}(\sigma_{|V})$.\\

There is a finite number of branches
$b_1,\ldots,b_n$ out of $x\in S$ that do not belong to 
$V$. 
Let $b$ be one of them, 
$O_b$ be the connected component of $X-S$ 
(either an open disc or annulus) containing $b$. 
We already know that  
$\R_{S}(-,\sigma)_{|O_{b}}=
\R_{\emptyset}(-,\sigma_{|O_b})$ is continuous on $O_b$, 
and that $\Gamma_S(\sigma)\cap O_b=
\Gamma_{\emptyset}(\sigma_{|O_b})$ is locally finite. 
Theorem \ref{Thm : finiteness sigma} 
then follows from Proposition 
\ref{Prop : finiteness on a branch} below.
\end{proof}

\begin{proposition}\label{Prop : finiteness on a branch}
$\R_S(-,\sigma)$ is continuous on the closure 
$\overline{O}_b$ of $O_b$ in $X$, 
and $\Gamma_S(\sigma)\cap O_b$ is 
\emph{finite} around $x$, i.e.  there exists a neighborhood $U$ of $x$ 
in $\overline{O}_b$ 
such that $\Gamma_S(\sigma)\cap U$ is finite.
\end{proposition}
\begin{proof}
\if{It remains to prove that $\Gamma_S(\sigma)$ 
is \emph{finite} around $x$, and the continuity of $\R_{S}(-,\sigma)$ 
around $x$. 
In particular if $b=]x,y[$, 
$\R_S(-,\sigma)$ is also continuous as a function on the segment 
$[x,y[$. 
}\fi
By Theorem \ref{thm:bonvois} 
we may find an affinoid neighborhood $Z$ of $x$ and an 
étale map $\psi:Z\to \mathbb{A}^{1,\mathrm{an}}_K$ of degree 
prime to $p$, verifying point vii) of Theorem \ref{thm:bonvois} at 
the branch $b$. 
More precisely we may find sections  
$C_{b}$ and $C_{\psi(b)}$ of $b$ and $\psi(b)$ respectively 
that are open annuli, and such that $\psi$ induces an isomorphism 
between annuli:
\begin{equation}
\psi\;:\;C_{b}\;\xrightarrow{\;\sim\;}\;C_{\psi(b)}\;.
\end{equation}
Up to shrinking $Z$ we may assume that its closure verifies 
$\overline{C}_{b}\cap S=\{x\}$, 
and that  $\Gamma_S\cap C_b$ is either empty or 
equal to $\Gamma_{C_{b}}$. Now, define as above 
$\delta_{\sigma,\psi}:=\psi\circ\sigma-\psi\in\O(Z)$.

We now distinguish two situations 
$C_{b}\cap\Gamma_S=\Gamma_{C_b}$ and 
$C_{b}\cap\Gamma_S=\emptyset$.

If $C_{b}\cap\Gamma_S=\Gamma_{C_b}$, we can 
localize to $C_b$ 
as in 
Proposition \ref{Proposition : Localization works well}. 
We proceed then similarly as in 
\eqref{eq : R_S(y,sigma)=delta/rho_psi} to 
show, by Lemma \ref{lem:isometry}, 
that for all $y\in C_b\cup\{x\}$ one 
has 
$\R_S(y,\sigma):=
|\delta_{\sigma,\psi}|(y)/\rho_{S,\psi}(y)$.
So we are done in this case.
\if{
In this case one sees that for all $y\in C_b\cup\{x\}$, 
$\psi$ 
induces an isomorphism between the maximal discs 
$\psi:D(y,S)\xrightarrow{\sim}
D(\psi(y),S')$,\footnote{Indeed if 
$y\in\Gamma_{C_b}\cup\{x\}$, then 
$D(y)=D(y,S)$, and if $y\not\in\Gamma_{C_b}$, then 
$D(y,S)$ is the connected component of 
$C_{b}-\Gamma_{C_b}$ containing it.} 
where $S'$ is a weak triangulation of $\psi(Z)$ such that 
$\Gamma_{S'}\cap C_{\psi(b)}=\Gamma_{C_{\psi(b)}}$. 

As in \eqref{eq : R_S(y,sigma)=delta/rho_psi}, 
Lemma \ref{lem:isometry} shows that for all $y\in C_b$ one 
has $\R_S(y,\sigma):=|\delta_{\sigma,\psi}|(y)/\rho_{S,\psi}(y)$.
So we are done in this case.
}\fi

Assume now that $C_b\cap\Gamma_S=\emptyset$. 
Since $X-\Gamma_S$ is a disjoint union of discs, 
$O_b$ is one of them, so the localization to $C_b$ 
affects $\R_{S}(-,\sigma)$ 
(cf. Remark 
\ref{Rk : restriction of sigma to the wedge of D}). 
In order to describe the link between
$\R_S(-,\sigma)$ and the norm of 
$\delta_{\sigma,\psi}$ 
we need the following Lemma which is deduced from 
\cite[9.7.1/2]{BGR}:

\begin{lemma}\label{Lemma : psi isometric}
Let $T$ be a coordinate on $\mathbb{A}_K^{1,\mathrm{an}}$.
Let $D^-(0,1)$ be the open unit disc, and let 
\begin{equation}
C\;:=\;C^-(0;]R,1[)\;=\;\{x\textrm{ such that }R<|T(x)|<1\}\;, \quad0<R<1\;.
\end{equation}

Let $\psi:C^-(0;]R,1[)\xrightarrow{\;\sim\;}C^-(0;]R,1[)$ be an 
isomorphism. Then
\begin{enumerate}
\item $\psi$ permutes the set of maximal discs in $C$ (i.e. if $D$ is a 
connected component of $C-\Gamma_C$, then 
$\psi$ induces an isomorphism of $D$ with another 
connected component $D'$).
\item $\psi$ is either the identity on the skeleton $\Gamma_C$ or it is 
the map sending $x_{0,\rho}$ into $x_{0,\rho^{-1}R}$.
\end{enumerate}
Moreover if $\psi$ induces the identity on $\Gamma_C$, 
then it is also isometric : for all $L/K$ and 
all $L$-rational points $t_1,t_2\in C(0;]R,1[)$ one has 
$|\psi(t_1)-\psi(t_2)|_L=|t_1-t_2|_L$.\hfill$\Box$
\end{lemma}
Let us come back to our situation: $x$ is a point of type 
$2$, so $O_b\cong D^-(0,1)$, 
and $C_b\cong C^-(0;]R,1[)$ for some $0<R<1$, 
as in Lemma \ref{Lemma : psi isometric}. 
\if{
Notice that in this case $\sigma$ stabilizes globally $C_b$, and it fixes 
the points of $\Gamma_{C_b}$, but it does 
not necessarily stabilizes its maximal discs (i.e. the connected 
components of $C_b-\Gamma_{C_b}$).
In other words $\sigma$ does not necessarily induces on $C_b$ an 
infinitesimal automorphism with respect to any weak triangulations of 
it.
Nevertheless to prove our statement we only need an explicit 
expression of $\R_S(-,\sigma)$ on $O_b\cup\{x\}$ expressing it as a 
function which is visibly constant outside some unspecified finite graph.
}\fi
We can assume that $\psi(x)=x_{0,1}$, and that 
$C_{\psi(b)}$ is an annulus with $x_{0,1}$ in its boundary. 
By translating, rescaling, and possibly considering the inversion 
$x\mapsto x^{-1}$ of $\mathbb{G}^{\mathrm{an}}_{m,K}$ 
we can assume that 
$\psi(C_b)=C_{\psi(b)}=C^-(0;]R,1[)$ with the same 
$R$ of  $C_b$.

So by Lemma \ref{Lemma : psi isometric}, $\psi$ is isometric. 
With these normalizations, by the above arguments, 
for all $y\in C_b\cup\{x\}$ we have 
$\R_S(y,\sigma)=|\delta_{\sigma,\psi}|(y)$.
This concludes the proof of Proposition 
\ref{Prop : finiteness on a branch}.
\end{proof}
\if{
Now, just after Lemma \ref{Lemma : trivial covering}
we have found an affinoid domain $V$ 
containing our point $x\in S$ of type $2$, 
such that $V\cap\Gamma_S=\{x\}$, on which 
$\R_S(-,\sigma)$ verifies Theorem \ref{Thm : finiteness sigma}.

Proposition \ref{Prop : finiteness on a branch} ensures that if $V'$ is 
the union of $V$ together with each connected component of 
$X-\Gamma_S$ (necessarily an open disc) having $x$ at its 
boundary, then $\R_S(-,\sigma)$ verifies again 
Theorem \ref{Thm : finiteness sigma} on $V'$. 

Now the remaining branches out of $x$ corresponds to the directions 
$b:=]x,y[$ belonging to $\Gamma_S$. Again 
Proposition \ref{Prop : finiteness on a branch} (together with 
Proposition \ref{Proposition : Localization works well}) 
shows that $\R_S(-,\sigma)$ verifies the claims of Theorem 
\ref{Thm : finiteness sigma} 
on the union of $V'$ with the annulus whose skeleton is $]x,y[$. 
If necessary we are allowed to shrinking $]x,y[$.

Doing so for all directions out of $x$ that belong to $\Gamma_S$ we 
obtain an open neighborhood $Y_x$ of $x$ of the type 
$Y_x=\tau_S^{-1}(\Lambda_x)$ as in Definition 
\ref{Def : Gamma_S-cov} on which Theorem 
\ref{Thm : finiteness sigma} holds.

Theorem \ref{Thm : finiteness sigma} is now proved.
\end{proof}
}\fi

\if{\subsection{Generalized $\Sigma$-algebras and generalized $\Sigma$-modules}

\comment{Credo che questa sezione sia meglio sopprimerla ?}
\subsubsection{Generalized $\Sigma$-algebras.} A \emph{generalized $\sigma$-algebra} is a ring $\B$ together with a collection of pairs of ring morphisms 
\begin{equation}
\{\;\B\begin{picture}(23,0)
\put(9,-5.5){\begin{tiny}$i_{\sigma}$\end{tiny}}
\put(3,-2.8){$\longrightarrow$}
\put(9,8.5){\begin{tiny}$\sigma$\end{tiny}}
\put(3,3.5){$\longrightarrow$}\end{picture}
\B_{\sigma}\;\}_{\sigma\in\Sigma}\;,
\end{equation}
where $\{\B_{\sigma}\}_{\sigma\in\Sigma}$ is a family of rings.
\subsubsection{$\Sigma$-constants.} We denote by $\B^{\Sigma}$, 
or $\B^{\Sigma=1}$, 
the sub-ring of $\B$ formed by elements $b\in\B$ 
verifying $\sigma(b)=i_{\sigma}(b)$, for all $\sigma\in\Sigma$. We call its elements \emph{$\Sigma$-constants} of $\B$ 
or equivalently \emph{$\Sigma$-invariant elements}. 
\subsubsection{Generalized $\Sigma$-modules.} A \emph{generalized $\Sigma$-module} is a finite free $\B$-module $\M$ 
together with a collection of isomorphisms
\begin{equation}
\{\;\sigma^{\M}:i_{\sigma}^*\M\simto \sigma^*\M\;\}_{\sigma\in\Sigma}\;,
\end{equation}
where $\sigma^*\M$ and $i^*_{\sigma}\M$ denote the scalar extension of $\M$ to $\B_{\sigma}$ via $\sigma$ and 
$i_\sigma$ respectively. 
A morphism between two generalized $\Sigma$-modules is a $\B$-linear map $\alpha\in\Hom_{\B}(\M,\N)$ satisfying 
$\sigma^*(\alpha)\circ\sigma^{\M}=\sigma^{\N}\circ i_{\sigma}^*(\alpha)$, for all $\sigma\in\Sigma$:
\begin{equation}
\xymatrix{
\ar@{}[dr]|{\odot}\sigma^*\M\ar[d]_{\sigma^*(\alpha)}&\ar[l]^-{\sim}_-{\sigma^{\M}} 
i_{\sigma}^*\M\ar[d]^{i_\sigma^*(\alpha)}\\
\sigma^*\N&\ar[l]_-{\sim}^-{\sigma^{\N}}i_{\sigma}^*\N&}\quad.
\end{equation}

Let $\Hom_{\B}^{\Sigma}(\M,\N)$ denote the set of \emph{morphisms}, it is canonically a $\B^{\Sigma}$-module. The 
category of generalized $\Sigma$-modules will be denoted by $\Sigma-\Mod(\B)$. 
This is a tensor category. 
The \emph{unit object} is $\mathbb{I}_{\B}:=(\B,\sigma^{\B})$, with $\sigma^{\B}=\mathrm{Id}_{\B_{\sigma}}$, where we 
identify $\B\otimes_{\B,\sigma}\B_\sigma$ and $\B\otimes_{\B,i_\sigma}\B_\sigma$ with $\B_\sigma$. The ring of 
endomorphisms of the unit object is then canonically identified with $\B^{\Sigma}$. The internal tensor product $\otimes$ is given by 
$(\M\otimes_{\B}\N,\sigma^{\M\otimes\N})$, with $\sigma^{\M\otimes\N}$ defined by
\begin{eqnarray}
\sigma^{\M\otimes\N} &=& \sigma^{\M}\otimes\sigma^{\N}\;,
\end{eqnarray}
where we identify $(\M\otimes_{\B}\N)\otimes_{\B,*}\B_{\sigma}$ with 
$(\M\otimes_{\B,*}\B_{\sigma})\otimes_{\B_{\sigma}}(\N\otimes_{\B,*}\B_{\sigma})$, where $*$ denotes $\sigma$ or 
$i_\sigma$. We say that $\M$ is \emph{trivial} if it is direct sum of the unit object $\mathbb{I}_{\B}$.

\begin{remark}
If $\Sigma$ is reduced to a single element $\sigma$, if $\B=\B_\sigma$, and $i_{\sigma}=\mathrm{Id}$, then we 
obtain the classical definition of $\sigma$-modules. 
\end{remark}

\subsubsection{Matrix of $\sigma^{\M}$.}\label{matrix of sigma} If $\e:=\{e_1,\ldots,e_n\}\subset\M$ is a basis of 
$\M$, the
operator $\sigma^{\M}$ is given by its values on $\e\otimes 1$. If
$\sigma^{\M}(e_j\otimes 1)=\sum_{i=1}^na_{\sigma,i,j}(e_i\otimes 1)$ we call 
\begin{equation}
A_\sigma\;:=\;(a_{\sigma,i,j})_{i,j=1,\ldots,n}\;\in\; GL_n(\B_\sigma)	
\end{equation}
 the matrix of $\sigma^{\M}$ in the basis $\e$.

\subsection{Solutions (formal definition)}
By analogy with the differential setting, we now define the notion of 
solution of a generalized $\Sigma$-module. For this we firstly need  to 
define the notion of \emph{extension of generalized $\Sigma$-algebras}.

\subsubsection{Extension of generalized $\Sigma$-algebras.} 
Let $\Sigma'$ be a family of morphisms, and let $\{\B'\begin{picture}(23,0)
\put(9,-5.5){\begin{tiny}$i_{\sigma'}$\end{tiny}}
\put(3,-2.8){$\longrightarrow$}
\put(9,8.5){\begin{tiny}$\sigma'$\end{tiny}}
\put(3,3.5){$\longrightarrow$}\end{picture}
\B_{\sigma'}'\}_{\sigma'\in\Sigma'}$ be a generalized $\Sigma'$-algebra. Assume given a surjective map 
$\P:\Sigma'\to\Sigma$. Assume moreover that $\B'$ is a $\B$-algebra with structural morphism $j:\B\to\B'$, and assume 
given, for all $\sigma'\in\Sigma'$, a ring morphism $\Delta_{\sigma'}:\B_{\P(\sigma')}\to\B_{\sigma'}$ making 
$\B_{\sigma'}$ a $\B_{\P(\sigma')}$-algebra. We say that the data of \begin{equation}
(\;\B'\begin{picture}(23,0)
\put(9,-5.5){\begin{tiny}$i_{\sigma'}$\end{tiny}}
\put(3,-2.8){$\longrightarrow$}
\put(9,8.5){\begin{tiny}$\sigma'$\end{tiny}}
\put(3,3.5){$\longrightarrow$}\end{picture}
\B_{\sigma'}'\;,\;j\;,\;\Delta_{\sigma'}\;)_{\sigma'\in\Sigma'}
\end{equation} 
is an \emph{extension of generalized $\Sigma$-algebras} if the following diagrams commute for all $\sigma'\in\Sigma'$:
\begin{equation}
\xymatrix{
\ar@{}[dr]|{\odot}\B'\ar[r]^{\sigma'}&\B'_{\sigma'}\ar@{}[rrd]|{;}&&\B'\ar[r]^{i_{\sigma'}}\ar@{}[dr]|{\odot}
&\B'_{\sigma'}&\\
\B\ar[u]^j\ar[r]_{\sigma}&\ar[u]_{\Delta_{\sigma'}}\B_{\sigma}&&\B\ar[u]^j\ar[r]_{i_{\sigma}}&\ar[u]_{\Delta_{\sigma'}}
\B_{\sigma}&,
}
\end{equation}
where $\sigma=\P(\sigma')$.

\subsubsection{Pull back.}\label{PULL-BACJK} With the above notations let $\M\in \Sigma-\Mod(\B)$. The $\B'$-module 
$j^*\M:=\M\otimes_{\B,j}\B'$ is  a $\Sigma'$-module in a canonical way by setting 
$(\sigma')^{j^*\M}:=\Delta_{\sigma'}^*(\sigma^{\M})=\sigma^{\M}\otimes_{\B_{\sigma},\Delta_{\sigma'}}\mathrm{Id}
_{\B_{\sigma'}'}$, with $\sigma=\P(\sigma')$, where $(\sigma')^*j^*\M$ (resp. $i_{\sigma'}^*j^*\M$) is canonically 
identified with $\Delta_{\sigma'}^*\sigma^*\M$ (resp. $\Delta_{\sigma'}^*i_{\sigma}^*\M$). If 
$A_\sigma=(a_{\sigma,i,j})_{i,j}\in GL_n(\B_{\sigma})$ is the matrix of $\sigma^{\M}$ in a fixed basis $\e\subset\M$ 
(cf. Section \ref{matrix of sigma}), then the matrix of $(\sigma')^{j^*\M}$ in the basis $\e\otimes 1$ is given by 
\begin{equation}
A_{\sigma'}\;\;:=\;\;\Delta_{\sigma'}(A_\sigma)\;\;=\;\;(\Delta_{\sigma'}(a_{\sigma,i,j}))_{i,j}\in 
GL_n(\B_{\sigma'}')\;.
\end{equation}
If $\alpha:\M\to\N$ is a morphism in $\Sigma-\Mod(\B)$, then $j^*\alpha:j^*\M\to j^*\N$ commutes with $\Sigma'$ and 
lies hence in $\Sigma'-\Mod(\B')$.

The definition of the scalar extension functor depends highly on the chosen family 
$\Delta:=\{\Delta_{\sigma'}\}_{\sigma'\in\Sigma'}$, so we denote it by 
\begin{equation}
j^*_{\Delta}\;\;:\;\;\Sigma-\Mod(\B)\;\;\xrightarrow{\qquad}\;\;\Sigma'-\Mod(\B')\;.
\end{equation}

\subsubsection{Solutions.} Let $\M\in \Sigma-\Mod(\B)$. A solution of $\M$ with values in $\B$ is an element of $\M$ 
satisfying $\sigma^{\M}(y\otimes 1)=y\otimes 1$, where $\sigma^{\M}:i_\sigma^*\M\simto \sigma^*\M$ for all 
$\sigma\in\Sigma$. Let now $\B'/\B$ be an extension of generalized $\Sigma$-algebras. 
A \emph{solution of $\M$ with values in $\B'$} is an element in 
\begin{equation}
(j^*\M)^{\Sigma'=1}\;\;:=\;\;\{\; y\in\M\otimes\B' \;|\; (\sigma')^{j^*\M}(y\otimes 1)=y\otimes 1,\;\textrm{ for all 
}\sigma'\in\Sigma' \;\}	\;.
\end{equation}
One sees that $(j^*\M)^{\Sigma'=1}$ is canonically a $(\B')^{\Sigma'=1}$-module.
\if{
\subsubsection{Solutions.} Let $\M\in \Sigma-\Mod(\B)$. A \emph{solution of $\M$ with
values in $\B'$} is a $\B$-linear map $\alpha:\M\to\B'$ satisfying $(\sigma')^*\alpha\otimes 1 
=i_{\sigma'}^*(\alpha\otimes 1)\circ \Delta_{\sigma'}^*\sigma^{\M}$, for all $\sigma'\in\Sigma'$, as expressed by the 
following commutative diagram:
\begin{equation}
\xymatrix{
\Delta_{\sigma'}^*\sigma^*\M
\ar[r]^-{\sim}\ar[d]_-{\sigma^{j^*\M}:=\Delta^*_{\sigma'}\sigma^{\M}}&
(\sigma')^*j^*\M\ar[rr]^{(\sigma')^{*}j^*\alpha}\ar@{}[drr]|{\odot}&&(\sigma')^*(\B')\ar[r]^-{\sim}&\B_{\sigma'}
'\ar@{=}[d]\\
\Delta_{\sigma'}^*i_{\sigma}^*\M
\ar[r]_-{\sim}&
i_{\sigma'}^*j^*\M\ar[rr]_-{i_{\sigma'}^{*}j^*\alpha}&&i_{\sigma'}^*(\B')\ar[r]_-{\sim}&\B_{\sigma'}'\\
}
\end{equation}
where $\sigma=\P(\sigma')$. If $\mathbb{I}_{\B}$ denotes the unit object, then the set of solutions is canonically 
identified with $\Hom_{\B'}^{\Sigma'}(j^*_{\Delta}\M,\mathbb{I}_{\B'})$. The solutions of $\M$ with values in $\B'$ 
are hence naturally a $(\B')^{\Sigma'}$-module. 
}\fi
\if{
\begin{lemma}
If $\mathrm{rank_{\B}}\M=n$, then $\Hom_{\B'}^{\Sigma'}(j^*_{\Delta}\M,\mathbb{I}_{\B'})$ is a finite free 
$(\B')^{\Sigma'}$-module of rank $\leq n$.
\end{lemma}
\begin{proof}
Clearly one can assume $\B=\B'$, $\Sigma=\Sigma'$, $j=\mathrm{Id}$,\ldots. Then 
$\Hom_{\B}^{\Sigma}(\M,\mathbb{I}_{\B})\subset \Hom_{\B}(\M,\mathbb{I}_{\B})=\M^*$. The $\B$-module $\N^{*}$ generated 
by $\Hom_{\B}^{\Sigma}(\M,\mathbb{I}_{\B})$ is stable under the action of $\Sigma$ on $\M^{*}$, its $\B$-rank is then 
less than the $\B$-rank of $\M$.
\end{proof}

\begin{lemma}
Let $n:=\mathrm{rank_{\B}}\M=n$. Then 
$\mathrm{rank}_{(\B')^{\Sigma'}}\Hom_{\B'}^{\Sigma'}(j^*_{\Delta}\M,\mathbb{I}_{\B'})=n$ if and only if 
$j^{*}_{\Delta}\M$ is trivial over $\B'$ (i.e. direct sum of copies of $\mathbb{I}_{\B'}$).
 \end{lemma}
\begin{proof}
Clearly one can assume $\B=\B'$, $\Sigma=\Sigma'$, $j=\mathrm{Id}$,\ldots. We consider the sub-$\B$-module $\N$ of 
$\M$ generated by the solutions ???, it is stable under $\Sigma$ ???, trivial over $\B$ ???. By the above lemma one 
has $\mathrm{rank}_{\B}(\N)=n$. Since the base change matrix is invertible ??? then $\M=\N$. 
\end{proof}
\begin{remark}
A generalized $\Sigma$-module $\M$ is trivial (i.e. direct sum of the unit object) if and only if 
$\Hom_{\B}^{\Sigma}(\M,\mathbb{I}_{\B})$ is a finite free $\B^\Sigma$-module having the same rank of $\M$. In other 
words $\M$ is trivial over $\B$ if and only if it has a complete basis of solutions in $\B$. Indeed if $Y\in GL_n(\B)$
\end{remark}
}\fi

\subsubsection{Matrix notation.} With the above notations, fix a $\sigma'\in\Sigma'$, and let $\sigma=\P(\sigma')$. 
Let $\e=\{e_1,\ldots,e_n\}\subset\M$ be a basis of $\M$, and let $A_{\sigma}$ be the matrix of $\sigma^{\M}$ in the 
basis $\e$. We identify $j^*\M$ with $(\B')^n$ by sending the canonical basis of $(\B')^n$ into $\e$. With this 
identification the vector $\Bigl(
\sm{y_1\\
\begin{picture}(0,10)
\put(-2,0){\vdots}
\end{picture}\\
y_n} \Bigr)\in (\B')^n$ is a solution of $\M$ with values in $\B'$, if and only if it verifies
$\Bigl(
\sm{\sigma'(y_1)\\
\begin{picture}(0,10)
\put(-2,0){\vdots}
\end{picture}\\
\sigma'(y_n)} \Bigr)=A_{\sigma'}\cdot\Bigl(
\sm{i_{\sigma'}(y_1)\\
\begin{picture}(0,10)
\put(-2,0){\vdots}
\end{picture}\\
i_{\sigma'}(y_n)} \Bigr)$, where $A_{\sigma'}=\Delta_{\sigma'}(A_\sigma)\in GL_n(\B_{\sigma'})$. By abuse of notation 
we will say that \emph{$\M$ is defined in the basis $\e$ by the family of
$\Sigma$-difference equations}
\begin{equation}\label{sigma(Y)=AY}
\sigma(Y)\;\;=\;\;A_\sigma\cdot i_\sigma(Y)\;,\quad\sigma\in\Sigma \;.
\end{equation}

}\fi
\section{Deformation}
\label{Deformation}
In this section we show how to deform 
a differential equation into a so called $\sigma$-difference 
equation.
\subsection{$\sigma$-difference equations.}
Let $\sigma:X\simto X$ be an automorphism of $X$.
A $\sigma$-difference equation over $X$ is a locally free  
$\O_X$-module $\Fs$ of finite rank, together with an $\O_X$-linear 
isomorphism 
\begin{equation}\label{eq : sigma:F-->F}
\sigma\;:\;\Fs\;\xrightarrow{\;\sim\;}\;\sigma^*\Fs\;.
\end{equation}
A morphism between $\sigma$-difference equations 
$\alpha:(\Fs,\sigma)\to(\Fs',\sigma')$ 
is an $\O_X$-linear map $\alpha:\Fs\to\Fs'$ such that 
$\sigma'\circ\alpha=\sigma^*(\alpha)\circ\sigma$. 
Usual operations of linear algebra exist. The category 
admits an internal tensor product, and a unit object 
which is 
$(\O_X,\mathrm{Id}_{\O_X}:\O_X\simto 
\sigma^*(\O_X))$. 
It not possible to localize such a structure to an analytic 
domain $Y$ in $X$, simply because $Y$ is possibly not 
stable under $\sigma$.

If $\Fs$ is free and $X$ is quasi-Stein, then 
\eqref{eq : sigma:F-->F} corresponds to a 
$\sigma$-semi-linear automorphism 
$\sigma:\Fs(X)\simto\Fs(X)$ (i.e. satisfying 
$\sigma(af)=\sigma(a)\sigma(f)$, for all 
$a\in \O(X)$, $f\in\Fs(X)$). 
If a basis of $\Fs$ is chosen, this 
corresponds to a system of the form:
\begin{equation}\label{eq : sigma(f)=Af}
\sigma(f_1,\ldots,f_r)^t=A_\sigma\cdot(\sigma(f_1),\ldots,\sigma(f_r))^t\;,
\qquad A_{\sigma}\in GL_r(\O(X))\;.
\end{equation}
\begin{definition}
Let $\Sigma\subseteq\mathfrak{S}(X,S)$ be a family of 
$S$-infinitesimal automorphisms of $X$. A $\Sigma$-module is a 
locally free $\O_X$-module $\Fs$ of finite type, together with a 
structure of $\sigma$-difference equation for all $\sigma\in\Sigma$.
A morphism of $\Sigma$-modules is an $\O_X$-linear morphism 
commuting with the action of all $\sigma\in\Sigma$. We denote by 
$\mathrm{Hom}^{\Sigma}(\Fs,\Fs')$ the group of morphisms.
\end{definition}

\subsection{$\sigma$-compatibility}
Let $\sigma\subseteq\mathfrak{S}(X,S)$ be an 
$S$-infinitesimal automorphism of $X$. 
Let $\Fs$ be a differential equation over $X$. 
We say that 
$\Fs$ is \emph{$\sigma$-compatible} if for all $x\in X$ 
we have $D^+(x,\sigma)\subset D(x,\Fs)$
(cf. \eqref{eq : def of D^+(x,sigma)} and 
\eqref{eq : def of D(x,F)}).
This is equivalent to the condition that for all $x\in X$ the condition
\begin{equation}\label{eq : R_S(x,sigma)<R_S(x,F)}
\R_S(x,\sigma)\;<\;\R_{S,1}(x,\Fs)\;.
\end{equation}
\begin{lemma}\label{Lemma : infinitesimality on S}
Assume that $X$ is connected.
The following conditions are equivalent:
\begin{enumerate}
\item \eqref{eq : R_S(x,sigma)<R_S(x,F)} holds for all 
$x\in X$;
\item \eqref{eq : R_S(x,sigma)<R_S(x,F)} holds for all 
$x\in S$, and 
over all germs of segments at the open 
boundary of $X$.\footnote{As an example, 
if $X=C^-(0;]R_1,R_2[)$ is an open 
annulus with empty weak triangulation, then condition ii) 
asks that there 
exist unspecified $\varepsilon_1,\varepsilon_2>0$ such 
that \eqref{eq : R_S(x,sigma)<R_S(x,F)} holds for all 
$x\in]x_{0,R_2-\varepsilon_2},x_{0,R_2}[$ 
and all $x\in]x_{0,R_1},x_{0,R_1+\varepsilon_1}[$.}
\end{enumerate}
\end{lemma}
\begin{proof}
This follows by the concavity properties of $\R_{S,1}(-,\Fs)$ (cf. 
properties iii) and iv) of Theorem \ref{Thm : continuity}), and by the 
convexity properties of $\R_S(-,\sigma)$ 
(cf. properties v) and vi) of Remark 
\ref{Rk : convexity of R_S(-,sigma)}).
\end{proof}

\subsection{Deformation}
We recall that all finite locally free 
$\O_{D(x,\Fs)}$-module 
is free since $\Omega$ is spherically complete.

Our main result is the following:
\begin{theorem}\label{Thm : deformation fredft}
Let $\sigma$ be an $S$-infinitesimal automorphism of 
$X$, and let 
$\Fs$ be a $\sigma$-compatible differential equation. 
Then there exists on $\Fs$ a canonical structure of 
$\sigma$-difference equation characterized by the fact 
that, for all $x\in X$, all solutions $(f_1,\ldots,f_r)^t$ 
of $\Fs_{|D(x,\Fs)}$ in a given basis, 
is also a solution of \eqref{eq : sigma(f)=Af} 
with respect to the same basis. 

If $\alpha:\Fs\to\Fs'$ is an $\O_X$-linear map between 
$\sigma$-compatible differential equations commuting with 
the connections, then $\alpha$ also commutes with the corresponding 
action of $\sigma$ on $\Fs$ and $\Fs'$.
\end{theorem}
\begin{proof}
We consider the map $\Delta_\sigma:X\to X\times X$ 
defined by 
$\Delta_\sigma=(\sigma,\mathrm{Id})$.
By Lemma 
\ref{Lemma : existence of adm neighb of diag} 
below there is an admissible open 
neighborhood $\mathcal{T}$ of the diagonal 
such that $\Delta_\sigma$ factorizes as
\begin{equation}
\Delta_\sigma\;:\;X\;\to\;\mathcal{T}
\;\subseteq\; X\times X
\end{equation}
and that the stratification $\chi$ of 
\eqref{eq : stratif :associated to F},
associated to $\Fs$, is defined over 
$\mathcal{T}$. 
We can consider the pull back of $\chi$ by $\Delta_\sigma$. Since 
$\Delta_\sigma^*\circ p_1^*(\Fs)=\sigma^*\Fs$ and 
$\Delta_\sigma^*\circ p_2^*(\Fs)=\Fs$ we find an isomorphism
\begin{equation}
\sigma^{\Fs}:=
\Delta_\sigma^*(\chi)\;:\;\Fs\xrightarrow{\;\sim\;}\sigma^*\Fs
\end{equation}
that satisfies 
the required properties.
\end{proof}
\begin{remark}\label{Rk : Same Taylor solutions}
Assume that 
$X$ is quasi-Stein, and that $\Omega^1_X$ 
and $\Fs$ are  both free over $X$. 
Let $Y'=G\cdot Y$, $G\in M_r(\O(X))$, be the differential 
equation attached to $\Fs$ in some basis. Then the 
action of $\sigma$ on $\Fs$ is given, in the same basis, 
by the equation $\sigma(Y)=A_\sigma\cdot Y$, where 
\begin{equation}
A_\sigma\;:=\;\Delta_\sigma(Y_\chi)\in GL_r(\O(X))\;,
\end{equation}
and $Y_\chi\in GL_r(\O(\mathcal{T}))$ 
is the matrix of $\chi$ in the same basis of $\Fs$. 
\end{remark}
\begin{lemma}\label{Lemma : existence of adm neighb of diag}
If $\Fs$ is $\sigma$-compatible, 
there exists an admissible open subset 
$\mathcal{T}\subset X\times X$, containing the image of 
$\Delta_\sigma$ and the diagonal, 
on which the stratification $\chi$ associated to $\Fs$ 
converges.
\end{lemma}
\begin{proof}
Let $x\in X-\Gamma_S$ and let $D_x$ be the connected 
component  (necessarily a virtual open disc) containing 
$x$. 
Let $T_x$ be a coordinate on $D_x$, 
and let $R_{D_x}$ be the radius of $D_x$ in that coordinate.  
If $\delta_{\sigma,T_x}:=T_x\circ\sigma-T_x$, 
then for all $y\in D_x$ we have 
$\R_S(y,\sigma)_{|_{D_x}}=
\R_{\emptyset}(y,\sigma_{|D_x})=
|\delta_{\sigma,T_x}|(y)/R_{D_x}$.

The radius of convergence of $\Fs$ enjoys the same 
properties and for all $y\in D_x$ we have 
$\R_S(y,\Fs)_{|_{D_x}}=
\R_{\emptyset}(y,\Fs_{|D_x})=
\R^{\Fs}(y)/R_{D_x}$,
where $\R^{\Fs}(y)$ is the radius of the largest open 
disc in 
$(D_x)_\Omega$ containing $t_y$ on which 
$\Fs_\Omega$ is trivial.
The function $y\mapsto|\delta_{\sigma,T_x}|(y)$ is 
increasing on 
the segments in $D_x$ (oriented as towards $+\infty$), 
and $y\mapsto\R^{\Fs}(y)$ is decreasing 
(cf. Theorem \ref{Thm : continuity}, and 
Remark \ref{Rk : convexity of R_S(-,sigma)}).
So, by \eqref{eq : R_S(x,sigma)<R_S(x,F)}, there exists $R_{x}\geq 0$ 
such that for all $y\in D_x$ we have
\begin{equation}\label{eq : inequalitycle over disck}
|\delta_{\sigma,T_x}|(y)< R_{x}<\R^{\Fs}(y)\;. 
\end{equation}
With the notations of Section \ref{eq : tubular n} we consider the following 
admissible open of $D_x\times D_x$
\begin{equation}
\mathcal{T}(D_x,T_x,R_x^-)\;:=\;\{|T_1-T_2|<R_x\}\;
\subseteq\; D_x\times D_x\;.
\end{equation}
It is then easy to check that the stratification $\chi$ 
defined by $\Fs_{|D_x}$ converges 
on $\mathcal{T}(D_x,T_x,R_x^-)$ and 
\begin{equation}\label{eq : Delta(D)subset T(D)}
\Delta_\sigma(D_x)\;\subseteq\;
\mathcal{T}(D_x,T_x,R_x^-)\;.
\end{equation}
Indeed the algebra of functions over 
$\mathcal{T}(D_x,T_x,R_x^-)$ 
is formed by the formal power series 
$f(T_1,T_2)=\sum_{n\geq 0}f_n(T_2)(T_1-T_2)^n$, 
with $f_n\in\O(D_x)$, such that for all virtual 
closed sub-disc $D\subseteq D_x$  one has
$\lim_n\|f_n\|_D\cdot\rho^n=0$, for all 
$\rho<R_{D_x}$.
We see that 
\begin{equation}\label{eq : grdftehsr}
\Delta_\sigma(f(T_1,T_2))\;=\;f(\sigma\circ T,T)\;=\;
\sum_{n\geq 0}f_n(T)\cdot\delta_{\sigma,T}(T)\;.
\end{equation}
Condition \eqref{eq : inequalitycle over disck} implies that 
$\delta_{\sigma,T}$ is bounded on $D_x$ and 
$\|\delta_{\sigma,T}\|_{D_x}\leq R_x\leq R_{D_x}$, 
so \eqref{eq : grdftehsr} 
converges as a series of functions in $\O(D_x)$. 

On the other hand, with the notations of \eqref{eq : Y_chi}, 
the matrix of $\chi$ is given in some basis by 
\begin{equation}
Y_\chi(T_1,T_2)\;:=\;\sum_{n\geq 0}G_n(T_2)(T_1-T_2)^n/n!\;.
\end{equation}
Its convergence locus is related to $\R^\Fs(-)$ by the relation
\begin{equation}
\R^{\Fs}(y)\;:=\;\min\Bigl(\;R_{D_x}\;,\;
\liminf_n\frac{1}{\sqrt[n]{|G_n|(y)/n!}}\;\Bigr)\;,\qquad
\textrm{ for all }y\in D_x\;.
\end{equation}
So \eqref{eq : inequalitycle over disck} implies that $Y_\chi$ lies in 
$M_{r\times r}(\O(\mathcal{T}(D_x,T_x,R_x^-)))$.

To conclude the proof we proceed as follows. We consider the 
open covering $\{U_x\}_{x\in X}$ of $X$ formed by 
\begin{enumerate}
\item for all $x\in X-\Gamma_S$ we consider the 
connected component $U_x:=D_x$ of $X-\Gamma_S$ containing 
$x$, together with the triplet $(D_x,T_x,R_x^-)$ that we have just 
obtained;
\item for all $x\in\Gamma_S$ we consider an arbitrary  
open neighborhood $U_x:=Y_x$ of $x$, 
of the form $Y_x=\tau_S^{-1}(\Lambda_x)$, 
with $\Lambda_x=\Gamma_S\cap Y_x$, together with 
an arbitrary triplet $(Y_x,T_x,R_x^-)$, 
where $R_x>0$ is such that 
$\chi$ converges on $\mathcal{T}(Y_x,T_x,R_x^-)$ 
(cf. Section \ref{eq : rough estimation}).
\end{enumerate}
Notice that, for all $x\in X$, $U_x\times U_x$ 
is open in $X\times X$, so $\mathcal{T}(U_x,T_x,R_x^-)$ is open in 
$X\times X$. 
We now consider the following open neighborhood of the diagonal of 
$X\times X$:
\begin{equation}\label{eq : big tubular union}
\mathcal{T}\;:=\;\bigcup_{x\in X}\mathcal{T}(U_x,T_x,R_x^-)\;=\;
\Bigl(\bigcup_{x\in X-\Gamma_S}
\mathcal{T}(D_x,T_x,R_x^-)\Bigr)\;
\bigcup\;
\Bigl(\bigcup_{x\in\Gamma_S}\mathcal{T}(Y_x,T_x,R_x^-)\Bigr)
\end{equation}
By construction $\chi$ converges on $\mathcal{T}$. 
On the other hand, for all $x\in X$, we have 
$\Delta_\sigma(x)\in\mathcal{T}$, indeed if 
$x\in X-\Gamma_S$ this 
follows from \eqref{eq : Delta(D)subset T(D)}, and if 
$x\in\Gamma_S$ this is evident since $\sigma(x)=x$, so 
$\Delta_\sigma(x)$ lies in the diagonal.
\end{proof}

\begin{definition}
Let $\Sigma\subseteq\mathfrak{S}(X,S)$ be a family of 
$S$-infinitesimal automorphisms. We say that a
differential equation $\Fs$ is $\Sigma$-compatible if it is 
$\sigma$-compatible, for all $\sigma\in\Sigma$.

The structure of $\Sigma$-module on $\Fs$ obtained by Theorem 
\ref{Thm : deformation fredft} is called the \emph{deformation} 
of the differential equation, and it will be indicated as 
$\Def_{S,\Sigma}(\Fs)$ or simply $\Def_{\Sigma}(\Fs)$.
\end{definition}

Theorem \ref{Thm : deformation fredft} gives a functor
called  \emph{$(S,\Sigma)$-deformation}
\begin{equation}\label{eq : Deformation Functor}
\Def_{S,\Sigma}\;:\;\{\Sigma-\textrm{compatible differential equations}/X\}\;
\xrightarrow{\quad}\;\{\Sigma-\textrm{Modules}/X\}\;.
\end{equation} 
\begin{definition}[\emph{Stratified $\Sigma$-modules}]
We call  the essential image of 
$\Def_{S,\Sigma}$ \emph{stratified $\Sigma$-modules}.
\end{definition}
The functor $\Def_{S,\Sigma}$ is additive, exact, it commutes with tensor 
products, and it is faithful since it is the identity on 
the morphisms. But it is not necessarily fully faithful as shown by the 
following example. 

\begin{example}
Let $q$ be a root of unity satisfying $|q-1|<1$.
Let $X=C^-(0;]R_1,R_2[)$ be an annulus with empty weak 
triangulation, and let $\sigma:=\sigma_q$ be the automorphisms 
sending $T\mapsto qT$. The deformation functor sends the unit 
object $\mathbb{I}=(d:\O(X)\to\O(X))$ 
into the unit object 
$\mathbb{I}=(\sigma_q:\O(X)\simto\O(X))$. 
The endomorphisms of $(\O(X),\sigma_q)$ 
are naturally in bijection with the elements of $\O(X)^{\sigma_q=1}=
\{f\in\O(X),\textrm{ such that }\sigma_q(f)=f\}$. 
Since $q$ is a root of unity there are several non constant function of 
this type. Hence the inclusion $\Def_\sigma:\mathrm{End}(\mathbb{I})\to\mathrm{End}(\Def_\sigma(\mathbb{I}))$ is strict.
\end{example}
\subsection{Fully faithfulness and non degeneracy}
\label{Fully faithfulness and non degeneracy}
In this section we provide a condition, called non degeneracy, on the 
family $\Sigma\subseteq\mathfrak{S}(X,S)$ that guarantee  the fully 
faithfulness of $\Def_\Sigma$.

It is clear that if a differential equation $\Fs$ is $\sigma$-compatible, 
it is also $\sigma^n$-compatible for all $n\in \mathbb{Z}$. 
More generally if $\Sigma$ is a family of $S$-infinitesimal 
automorphisms, and if $\Fs$ is $\sigma$-compatible for all 
$\sigma\in \Sigma$, then it is $\sigma$-compatible for all 
$\sigma$ lying in the subgroup $\lr{\Sigma}$ of 
$\mathfrak{S}(X,S)$ generated by $\Sigma$.

We are then naturally induced to work with groups of 
$S$-infinitesimal automorphisms.

\begin{definition}[Non degeneracy]
\label{Def : non degeneracy}
Let $\Sigma\subseteq\mathfrak{S}(X,S)$ be a family of 
$S$-infinitesimal automorphisms. We say that the action of 
$\Sigma$ is non degenerate if for all connected component $X'$ 
of $X$ there is a point $x\in X'$ such that for all open disc $D$ such 
that $\bigcup_{\sigma\in\Sigma}
D^+(x,\sigma)\subseteq 
D\subseteq D(x,S)$ %
one has\footnote{Here 
$\Sigma_\Omega:=\{\sigma_\Omega\}_{\sigma\in\Sigma}$, where 
as usual $\sigma_\Omega=\sigma\widehat{\otimes} 
\mathrm{Id}_\Omega$.}
\begin{equation}\label{eq : O(D)^Sigma=Omega}
\O(D)^{\Sigma_\Omega=1}\;=\;\Omega\;.
\end{equation}
\end{definition}
\begin{proposition}[Criterion of non degeneracy]
\label{Prop : criterion non degeneracy ,gf}
Let $x\in X$, and let $\Sigma\subseteq\mathfrak{S}(X,S)$. 
Assume that we have 
a sequence of elements $\{\sigma_n\}_n$ in the group $\lr{\Sigma}$ such that
\begin{enumerate}
\item $\bigcap_n D^+(x,\sigma_n) = \{t_x\}$;
\item For infinitely many $n$ the disc $D^+(x,\sigma_n)$ is not 
reduced to $\{t_x\}$.
\end{enumerate}
Then the action of $\Sigma$ is non degenerate.
\end{proposition}
\begin{proof}
Let $T$ be a coordinate on $D(x,S)$. 
Let $D$ be an open disc as in Definition
\ref{Def : non degeneracy}. Let 
$f\in\O(D)$ be a function stable under $\Sigma$, then 
$g(T):=f(T)-f(t_x)$ is also stable under $\Sigma$. 
Now we have a sequence $n\mapsto\sigma_n(t_x)$ of zeros of 
$g$ accumulating to $t_x$. So $g=0$ and $f=f(t_x)$ is constant.
\end{proof}
\begin{remark}
 Let 
$T_x:U\to\mathbb{A}^{1,\mathrm{an}}_K$ 
be a local coordinate on some neighborhood $U$ of $x$ 
verifying Lemma \ref{Lemma : trivial covering}. 
As usual, let $\delta_{\sigma,T_x}:=T_x\circ\sigma-
T_x$. 
Then conditions i) and ii) of Proposition 
\ref{Prop : criterion non degeneracy ,gf} are
equivalent to the condition that the closure in 
$\mathbb{R}$ of the set  
$\{|\delta_{\sigma_n,T_x}|(x)\;,\;n\geq 0\}-\{0\}$ contains $0$.
\end{remark}

\begin{proposition}\label{Prop : non degeneracy --> fully faith}
If the action of $\Sigma$ is non degenerate, then  
$\mathrm{Def}_{S,\Sigma}$ is fully faithful.
\end{proposition}
\begin{proof}
The functor is the identity on the morphisms, so it is 
enough to show 
that if a morphism $\alpha:\Fs\to\Fs'$ of 
$\O_X$-modules 
commutes with $\Sigma$, then it also commutes 
with $\nabla$. This is true if and only if it is true for 
$\alpha_\Omega:\Fs_\Omega\to\Fs'_\Omega$.
Now, by Lemma 
\ref{Lemma : alpha commutes with nabla loc then glob}, 
it is enough to show that $\alpha_\Omega$ 
commutes with $\nabla$ over the disc $D$ of Definition 
\ref{Def : non degeneracy}. The ring $\O(D)^{\Sigma}$ 
contains $\Omega$, and the space 
$\mathrm{Hom}^{\Sigma}(\Fs_{|D},\Fs_{|D})$ is an 
$\O(D)^{\Sigma}$-module containing 
$\mathrm{Hom}^{\nabla}(\Fs_{|D},\Fs_{|D})$ as a 
sub-$\Omega$-vector space. 
Since
$\O(D)^{\Sigma_\Omega}=\Omega$,
$\mathrm{Hom}^{\Sigma}(\Fs_{|D},\Fs_{|D})$ is 
an $\Omega$-vector space too.

Now choose $D$ so that the differential equations 
$\Fs_{|D}$ and $(\Fs')_{|D}$ are trivial on $D$ 
(this is possible since $\Fs$ and 
$\Fs'$ are both $\Sigma$-compatible). 
Their deformations over $D$ are hence trivial too, and 
the deformation commutes with the localization to $D$. 
Hence the space 
$\mathrm{Hom}^\Sigma(\Fs_{|D},\Fs'_{|D})$ is an 
$\Omega$-vector space of dimension 
$\mathrm{rank}(\Fs)\cdot\mathrm{rank}(\Fs')$. 
The dimension of
$\mathrm{Hom}^\nabla(\Fs_{|D},\Fs'_{|D})$ is the same, so they 
coincide. Hence $\alpha_\Omega$ commutes with $\nabla$ over $D$, 
and the claim is proved.
\end{proof}
\if{
\subsubsection{Categorical non degeneracy.}
In some cases it may be difficult to find a point $x$ satisfying 
\eqref{eq : O(D)^Sigma=Omega}, or even such a point does not 
exist. 
We then generalize the notion of non degeneracy as follows.
\begin{definition}[Categorical non degeneracy]
We say that the family $\Sigma$ is categorically non degenerate if for 
all connected component $Y$ of $X$ there exists a quasi-Stein 
analytic domain $Z\subseteq Y$, stable by $\Sigma$, 
together with a integral domain $\B$ such that
\begin{enumerate}
\item $\B$ is a $\Omega$-algebra;
\item $\Sigma$ acts by automorphisms on $\B$, and we have 
\begin{equation}
\B^{\Sigma}\;=\;\mathrm{Frac}(\B)^{\Sigma}\;=\;\Omega\;.
\end{equation}
\item There exists an injective $\Omega$-linear ring morphisms 
$\O(Y_\Omega)\to \B$ commuting with $\Sigma$.
\end{enumerate}
\end{definition}
\comment{CONTINUARE ...
BISOGNA ANCHE CHE TUTTE LE EQUAZIONI SIANO TRIVIALIZZATE DA $\B$}
}\fi
\subsection{Analyticity of the action of $\Sigma$.}
\label{Analyticity of the action}
Until now we have studied the action of a family of 
automorphisms on $X$. 
We now consider the action of a $K$-analytic group $G$ 
on it.  
In this section we prove that the deformation of a 
differential equation produces an 
\emph{analytic} (semi-linear) action of $G$ on  $\Fs$ 
lifting the action of $G$ on $X$. 
This kind of object is commonly know as 
\emph{$G$-equivariant sheaf on $X$} 
(cf. \cite[Section 12]{Mumford-AbVar} and 
\cite{Mumford-invariants}). 

The action of a $K$-analytic group 
$G$ on $X$ is a morphism 
\begin{equation}
\mu\;:\;G\times X\;\xrightarrow{\quad}\;X
\end{equation}
satisfying the natural conditions of 
\cite[Section 12]{Mumford-AbVar}, and 
\cite[Section 5.1]{Ber} for the details about the setting. 
In particular, for all complete valued field extension $L/K$ and all 
$L$-rational point $g:\mathscr{M}(L)\to G_L$ 
we have an automorphism 
$\sigma_g:=\mu_L\circ(g\times\mathrm{Id}_{X_L})$ of $X_L$
\begin{equation}
\sigma_g\;:\;X_L\;\xrightarrow{\;\sim\;}\;X_L\;.
\end{equation}
This family verifies 
$\sigma_g\circ\sigma_h=\sigma_{g\cdot h}$ 
for all $g,h\in G(L)$.
So for all  $L/K$ we have a group morphism 
\begin{equation}
G(L)\;\xrightarrow{\quad}\;\mathrm{Aut}(X_L)
\end{equation}
Let un call $\Sigma_L$ the image of $G(L)$ in 
$\mathrm{Aut}(X_L)$.
\begin{definition}
We say that the action of $G$ on $X$ is $S$-infinitesimal 
if for all $L/K$ and for all $g\in G(L)$ the map $\sigma_g$ is an 
$S_L$-infinitesimal automorphism of $X_L$.

We say that the action of $G$ is non degenerate if there exists 
$L/K$ such that the action of $G(L)$ on $X_L$ is non degenerate.

We say that a differential equation $\Fs$ on $X$ 
is $G$-compatible, if for all $L/K$ and all $g\in G(L)$ the 
equation $\Fs_L$ is $\sigma_g$-compatible.\footnote{This is equivalent to saying that for all point 
$g:\mathscr{H}(g)\to G$ the equation 
$\Fs_{\mathscr{H}(g)}$ is $\sigma_g$-compatible.}
\end{definition}

\subsubsection{Analytic semi-linear $G$-modules.} 
\label{Analytic semi-linear $G$-modules.} 
We shall now introduce (somehow informally)
the notion of analytic semi-linear $G$-module 
(cf. Definition \ref{Def : An-semilin G mod}). 
In order to do that, for all $L/K$ we interpret the family of automorphism 
$\{\sigma_g:X_L\simto X_L\}_{g\in G(L)}$ as a \emph{covering} of 
$X_L$. A \emph{semi-linear $G(L)$-module} is then a 
\emph{gluing datum} on a family $\{\Fs_g\}_g$ 
of locally free $\O_{X_L}$-modules on the covering,
 where $\Fs_g=\Fs$ for all $g\in G(L)$.
Concretely this amounts to give a locally $\O_{X_L}$-module $\Fs_L$ 
together with a family of 
$\O_{X_L}$-linear isomorphism
\begin{equation}\label{eq: sigma-mod trivial point of view ....}
\{\;\sigma_g^{\Fs_L}\;:\;\Fs_L\xrightarrow{\;\sim\;}
\sigma_g^*(\Fs_L)\;\}_{g\in G(L)}\;,
\end{equation}
subjected to the cocycle condition
\begin{equation}\label{eq : cocycle relation naif}
\sigma_{gh}^{\Fs_L}\;=\;
\sigma_h^*(\sigma_g^{\Fs_L})\circ\sigma_{h}^{\Fs_L}\;.
\end{equation}
We may give the following informal definition:
\begin{definition}
A semi-linear $G$-module is a locally free $\O_X$-module 
$\Fs$ of finite type together with a semi-linear 
$G(L)$-module structure on $\Fs_L$ for all $L/K$, 
which satisfies the evident compatibilities for all 
base change of the ground field $K$.
\end{definition}
More generally we can perform the same construction 
for all base change $S/K$, i.e we regard the objects as functors on 
the category of $K$-analytic spaces. 
So by Yoneda's Lemma we obtain the following definition.

Let $p_X:G\times X\to X$ be the second projection. 
Consider the simplicial object 
\begin{equation}
\label{eq : simplicial object}
G\times G\times X
\begin{picture}(25,20)
\put(5,15){\begin{small}$d_0$\end{small}}
\put(3,0){$\longrightarrow$}
\put(5,5){\begin{small}$d_1$\end{small}}
\put(3,10){$\longrightarrow$}
\put(5,-5){\begin{small}$d_2$\end{small}}
\put(3,-10){$\longrightarrow$}
\end{picture}
G\times X
\begin{picture}(25,0)
\put(5,10){\begin{small}$p_X$\end{small}}
\put(3,3){$\longrightarrow$}
\put(3,-3){$\longrightarrow$}
\put(7,-7){\begin{small}$\mu$\end{small}}
\end{picture}
X\;.
\end{equation}
where $d_0=p_2\times p_3$, $d_1=m_G\times\mathrm{Id}_X$, 
and $d_2=\mathrm{Id}_G\times\mu$, and $p_i$ is the $i$-th 
projection of $G\times G\times X$.  
\begin{definition}\label{Def : An-semilin G mod}
An analytic semi-linear $G$-module is a locally free $\O_X$-module 
$\Fs$ together with an isomorphism
\begin{equation}\label{eq : q}
\sigma_G\;:\;p_X^*(\Fs)\xrightarrow{\;\sim\;}\mu^*(\Fs)\;,
\end{equation}
satisfying the cocycle condition. This means 
that the following diagram commutes
\begin{equation}\label{eq : cocycle condition G-invariant}
\xymatrix{p_3^*(\Fs)
\ar[rd]_{d_1^*(\sigma_G)}
\ar[rr]^{d_0^*(\sigma_G)}&&
(\mu\circ d_0)^*(\Fs)\ar[ld]^{d_2^*(\sigma_G)}\\
&\eta^*(\Fs)&}
\end{equation}
where 
$\eta:=\mu\circ d_1=
\mu\circ d_2$.
\end{definition}
For all $L/K$ and all $g\in G(L)$ we can pull-back $\sigma_G$ by the 
map $g\times\mathrm{Id}_{X_L}:X_L\to G_L\times X_L$ and we 
obtain the family of maps 
\eqref{eq: sigma-mod trivial point of view ....}, and the diagram 
\eqref{eq : cocycle condition G-invariant} gives the cocycle relation 
\eqref{eq : cocycle relation naif}.

\begin{remark}
If $X$ is quasi-Stein and if $\Fs$ is free, a 
semi-linear $G(L)$-module corresponds to a 
family of difference equations of the form 
$\sigma_g(\vec{f})=A_{\sigma_g}\cdot \vec{f}$ over 
$X_L$ (cf. \eqref{eq : sigma(f)=Af}). 
The action of $G$ is analytic if and only if 
$A_{\sigma_g}$ is an analytic function on 
$G\times X$, i.e. also with respect to the variable $g$.
\end{remark}

\begin{theorem}\label{Thm : G-semilinear}
Assume that the action of $G$ on $X$ is 
$S$-infinitesimal. There exists a functor 
\begin{equation}
\{G\textrm{-compatible diff. eq.}\}/X
\;\xrightarrow{\quad}\;
\{\textrm{Analytic semi-linear }G\textrm{-modules}\}/X\;.
\end{equation}
The functor associates to a differential equation $\Fs$ 
over $X$ the 
same $\O_X$-module $\Fs$ together with a semi-linear 
action of $G$. If $\Sigma_L$  is the image of $G(L)$ in 
$\mathrm{Aut}(X_L)$, the action of $G$ 
is characterized by the fact that for 
all $L/K$ the action of $\Sigma_L$ on $\Fs_L$ so 
obtained coincides with that of Theorem 
\ref{Thm : deformation fredft}. 

The functor is the identity on the morphisms, in particular 
it is faithful. If the action of $G$ on $X$ is non 
degenerate, it is also fully faithful.
\end{theorem}
\begin{proof}
We consider the map $\Delta_G\;:=\;\mu\times p_X$
\begin{equation}\label{eq : map frdmplqym for strat}
\Delta_G\;:\;
G\times X\;\xrightarrow{\quad}\;X\times X\;.
\end{equation}
By the Lemma \ref{Lemma : existence of pull-back fo} below, 
the image of the map $\Delta_G$ is contained in some 
admissible neighborhood $\mathcal{T}$ of the diagonal over 
which the stratification 
$\chi:p_2^*(\Fs)_{\mathcal{T}}\xrightarrow{\;
\sim\;}p_1^*(\Fs)_{\mathcal{T}}$ associated to the differential 
equation $\Fs$ converges. So by pull-back we have an isomorphism 
$\sigma_G:=\Delta_G^*(\chi)$ as in \eqref{eq : q}. 
It is clear that it gives a structure of analytic semi-linear 
$G$-module on $\Fs$ with the required properties. In particular the 
cocycle condition \eqref{eq : cocycle condition G-invariant} follows 
from the cocycle condition i) of Section \ref{eq : cocycle condition stratification} 
verified by $\chi$.
\end{proof}

\begin{lemma}\label{Lemma : existence of pull-back fo}
There exists an admissible neighborhood of the diagonal 
$\mathcal{T}$ of $X\times X$ on which the stratification 
associated with $\Fs$ converges, such that the image of 
the morphism
\begin{equation}
\Delta_G\;:\;G\times X
\;\xrightarrow{\;\quad\;}\;X\times X
\end{equation}
is contained in $\mathcal{T}$.
\end{lemma}
\begin{proof}
We can assume $K=\Ka$. 
If $g\in G(L)$ for some $L/K$, we have a commutative 
diagram:
\begin{equation}\label{eq : diag : Delta_G and Delta_g}
\xymatrix{G_L\times X_L\ar[r]^{(\Delta_G)_L}&X_L\times X_L\\
X_L\ar[u]^{g\times\mathrm{Id}_{X_L}}\ar[ur]_{\Delta_{\sigma_g}}&}
\end{equation}
We consider a large field 
$\Omega/K$ such that the vertical maps of the following diagram are 
all surjective
\begin{equation}
\xymatrix{
G(\Omega)\times 
X(\Omega)\ar[r]^{\Delta_G(\Omega)}
\ar@{->>}[d]&X(\Omega)\times 
X(\Omega)\ar@{->>}[d]&
\ar@{}[l]|-{\supset}\mathcal{T}(\Omega)\ar@{->>}[d]\\
G\times X\ar[r]^{\Delta_G}&X\times 
X&\ar@{}[l]|-{\supset}\mathcal{T}}
\end{equation}
This is possible thanks to  \cite[Prop. 2.1.7]{NP-III}. 
It is then enough to prove that 
\begin{equation}\label{eq : deltaG subset}
\Delta_G(\Omega)(\,G(\Omega)\times X(\Omega)\,)
\;\subseteq\; 
\mathcal{T}(\Omega)\;.
\end{equation}
We now show that this is automatic. 
Let $\{U_x\}_{x\in X}$ be the open covering of $X$ 
that we have obtained in the proof of Lemma 
\ref{Lemma : existence of adm neighb of diag}. Then
$\{(U_x)_\Omega\}_{x\in X}$ is again an open covering 
of $X_\Omega$  since the projection 
$\pi_{\Omega/K}:X_\Omega\to X$ gives an 
isomorphism 
$\Gamma_{S_{\Omega}}\simto\Gamma_{S}$.
As in the proof of 
\ref{Lemma : existence of adm neighb of diag} 
define $\mathcal{T}$ as the union 
\eqref{eq : big tubular union} 
of local neighborhoods of the diagonal defined by some conditions of 
type $|(T_{x})_1-(T_x)_2|<R$, where $T_x:U_x\to\mathbb{A}^{1,
\mathrm{an}}_K$ is 
a local étale coordinate on some open $U_x$. 
This implies that $\mathcal{T}_\Omega$ is defined by the 
same conditions, with $(T_x)_\Omega:(U_x)_\Omega\to
\mathbb{A}^{1,\mathrm{an}}_\Omega$, and the same 
$R$. Now the radii $\R_{S}(-,\sigma)$ and 
$\R_{S,1}(-,\Fs)$ are stable by 
scalar extension to $\Omega$, so the proof of Lemma 
\ref{Lemma : existence of adm neighb of diag} shows 
that for all $g\in\Sigma_\Omega=G(\Omega)$ we have 
$\Delta_{\sigma_g}(X_\Omega)\subseteq 
\mathcal{T}_\Omega$.
This implies \eqref{eq : deltaG subset} by the diagram 
\eqref{eq : diag : Delta_G and Delta_g}. 
\end{proof}

\section{Deformation and quasi unipotence over the 
Robba ring}
\label{Deformation over the Robba ring}
Several situations requires an analysis along a germ of 
segment along a Berkovich curve. This corresponds to the study of differential equations over the 
Robba ring $\mathfrak{R}$.
So we now restrict our attention to this case, 
which is the case studied in \cite{An-DV} in the context of $q$-difference equations. 

In sections \ref{Section : Deformation Robba} 
we state the Deformation 
over $\mathfrak{R}$ 
(that needs some ad-hoc definitions). 
In sections \ref{Section : Frob str sol}, 
\ref{Section : Special ext}, 
\ref{Section  Katz-Matsuda can ext}, 
\ref{Section p-adic loc mon th},  we will assume the 
following
\begin{hypothesis}\label{Hyp : K discr}
$K$ is discretely 
valued, of mixed characteristic, and with perfect 
residual field. 
\end{hypothesis}
As a consequence every locally free 
$\mathfrak{R}$-module of finite 
type will be free by \cite[Thm. 4.40]{Christol-Book}.

\begin{definition}
We set $C_\varepsilon:=C^-(0;]1-\varepsilon,1[)$. 
We always consider 
on $C_\varepsilon$ the empty weak triangulation. The 
Robba ring is defined as  
\begin{equation}
\mathfrak{R}\;:=\;
\bigcup_{\varepsilon>0}\O(C_\varepsilon)\;=\;
\Bigl\{\sum_{n\in\mathbb{Z}}a_nT^n\;|\;\exists 
\varepsilon>0 
\textrm{ such that }\lim_{n\to\pm\infty}
|a_n|\rho^n=0\;,\;\forall 
\rho\in]1-\varepsilon,1[ \Bigr\}\;.
\end{equation}
\end{definition}

\subsection{Deformation}\label{Section : Deformation Robba}
Since we are working over a \emph{germ} of annulus, 
some definition have to be adapted. 

If $\sigma$ is an infinitesimal 
automorphism of $C_\varepsilon$, then it is so also for 
all $C_{\varepsilon'}$ with $\varepsilon'<\varepsilon$ 
(cf. Lemma \ref{Lemma : localization}). 
The function 
$x\mapsto\R_{\emptyset}(x,\sigma)$ over $C_\varepsilon$ 
commutes with the restriction to  $C_{\varepsilon'}$, 
$\varepsilon'<\varepsilon$.
An automorphism of $\mathfrak{R}$ is called \emph{infinitesimal} if 
it comes from an infinitesimal automorphism of $C_\varepsilon$ for 
some $\varepsilon>0$ (with respect to the empty weak triangulation). 
If $\Sigma$ is a family of infinitesimal automorphisms of 
$\mathfrak{R}$ we denote by 
$\Sigma_\varepsilon\subseteq\Sigma$
the subfamily of $\Sigma$ of those automorphisms that 
are defined over $C_\varepsilon$, and  are infinitesimal 
on it. 

A differential equation over $\mathfrak{R}$ is, 
by definition, a locally 
free $\O_{C_\varepsilon}$-module together with a 
connection, for some 
unspecified $\varepsilon>0$.

On the other hand a $\Sigma$-module $\Fs$ over $\mathfrak{R}$ 
does not necessarily come from a $\Sigma_\varepsilon$-module 
$\Fs_\varepsilon$ over $C_\varepsilon$. 
Indeed we lose the action of each  
$\sigma\in\Sigma-\Sigma_\varepsilon$. The 
definitions are then the following

\begin{definition}
Let $\Fs$ be a differential equation over $\mathfrak{R}$, defined 
over $C_{\varepsilon_\Fs}$ for some $\varepsilon_\Fs>0$.

For all $\varepsilon<\varepsilon_\Fs$ let 
$\Sigma_\varepsilon(\Fs)\subseteq\Sigma_\varepsilon$ 
be the subset formed by those $\sigma\in\Sigma_\varepsilon$ such 
that $\Fs_\varepsilon$ is $\sigma$-compatible over $C_\varepsilon$.\footnote{We recall 
that this happens if and only if 
$\R_{\emptyset}(x_{0,\rho},\sigma)<
\R_{\emptyset,1}(x_{0,\rho},\Fs)$,
for all $\rho\in]1-\varepsilon,1[$.}
We say that $\Fs$ is $\Sigma$-compatible if for all 
$\varepsilon>0$ we have
\begin{equation}\label{eq : Sigma-compatibility}
\bigcup_{0<\varepsilon'<\varepsilon}
\Sigma_{\varepsilon'}(\Fs)\;=\;\Sigma\;.
\end{equation}
\end{definition}

\if{\comment{Non è la definizione giusta ! Non funziona per 
$\bs{\mu}_{p^\infty}$ !!

$\bs{\mu}_{p^\infty}$ agisce tutto intero su ogni disco massimale, 
ma una equazione solubile non è tutta intera 
$\bs{\mu}_{p^\infty}$-compatibile, 
perché se la sua pendenza $p$-adica è non nulla per ogni $\rho<1$
il disco $D(x_{0,\rho},\Fs)$ è strettamente contenuto in $D(x)$, 
mentre invece $D(x_{0,\rho},\bs{\mu}_{p^{\infty}})=D(x)$ per tutti 
i $\rho<1$.\\

Proposta di definizione : 

Sia $\varepsilon_\Fs$ tale che $\Fs$ proviene per estensione degli 
scalari da una equazione $\Fs_\varepsilon$ su $C_\varepsilon$. 

Per 
ogni $0<\varepsilon<\varepsilon_\Fs$ sia 
$\Sigma_\varepsilon(\Fs)$ il piu' grande sottoinsieme di 
$\Sigma$ tale che $\Fs_\varepsilon/C_\varepsilon$ è 
$\Sigma_\varepsilon(\Fs)$-compatibile. 
E' l'insieme degli automorfismi $\sigma\in\Sigma$ tali che 
$\Fs_\varepsilon$ è $\sigma$-compatibile sopra $C_\varepsilon$.
Diciamo che $\Fs$ è $\Sigma$-compatibile se per ogni $\varepsilon$ 
vicino ad $1$ si ha
\begin{equation}
\bigcup_{0<\varepsilon'<\varepsilon}
\Sigma_{\varepsilon'}(\Fs)\;=\;\Sigma\;.
\end{equation}
(Questo funziona 
per $\bs{\mu}_{p^\infty}$).

Condizione necessaria per avere questa uguaglianza è che per ogni 
$\varepsilon>0$ si abbia 
\begin{equation}
\bigcup_{0<\varepsilon'<\varepsilon}\Sigma_{\varepsilon'}
\;=\;\Sigma
\end{equation}
ove $\Sigma_{\varepsilon'}$ indica l'insieme dei $\sigma$ che sono 
definiti su $C_{\varepsilon'}$. La condizione è necessaria perché 
$\Sigma_{\varepsilon'}(\Fs)\subseteq\Sigma_{\varepsilon'}$ \\

Per la stessa ragione bisogna generalizzare la dimostrazione della 
"fully faithfulness" che non segue in modo straightforward da 
quella che ho fatto per le curve. 
Qui infatti non c'è nessun punto in cui $D(x_{0,\rho},
\bs{\mu}_{p^{\infty}})\subseteq D(x,\Fs)$.
Pero' dovrebbe funzionare lo stesso !

Solo che ora non riesco a capire completamente perché ... l'idea è la 
seguente. Sia $\Sigma_\varepsilon(\Fs_\varepsilon)$ 
il sottoinsieme di $\Sigma_\varepsilon$ tale che $\Fs_\varepsilon$ è 
$\Sigma_\varepsilon(\Fs_\varepsilon)$-compatibile. 
Voglio far tendere $\rho$ verso $1$, e considerare il morfismo
di specializazione 
$\mathfrak{R}\to\cup_\rho D(x_{0,\rho},
\Sigma_\varepsilon(\Fs_\varepsilon))$
Vorrei dimostrare che se $\alpha:\Fs\to\Fs'$ commuta con $\Sigma$, 
allora $\alpha$ commuta con $\nabla$... vorrei usare un argomento 
di uguaglianza di dimensioni come prima, ma forse potrei anche 
procedere per assurdo e mostrare che se $\alpha$ non commutasse 
con $\nabla$ allora non commuta con nabla in nessun disco 
$D(x_{0,\rho},
\Sigma_\varepsilon(\Fs_\varepsilon))$ per il Lemma 
\ref{Lemma : alpha commutes with nabla loc then glob}. L'idea è di 
trovare un assurdo per $\rho\to 1$.

Quel che è certo è che funziona per l'equazione triviale :-)
}
}\fi

Theorem \ref{Thm : deformation fredft} furnishes an 
faithful functor associating to a $\Sigma$-compatible differential 
equation, a $\Sigma$-module over $\mathfrak{R}$. As 
usual we call the essential image of 
$\mathrm{Def}_{\Sigma}$ \emph{stratified} 
$\Sigma$-modules.

\begin{definition}
Let $\Sigma$ be a family of infinitesimal automorphisms of 
$\mathfrak{R}$. We say that $\Sigma$ is non degenerate if for all 
$\varepsilon>0$ there exists $0<\varepsilon'<\varepsilon$ such that 
$\Sigma_{\varepsilon'}$ is 
non degenerate over $C_{\varepsilon'}$.
\end{definition}
Since the family $\Sigma_{\varepsilon}$ is not necessarily 
contained in $\Sigma_\varepsilon(\Fs)$, we 
then proceed as follows:
\begin{proposition}[Fully faithfulness]
\label{Prop: fully faithfulness over Robba}
Assume that $\Sigma$ is a non degenerate family of automorphisms 
of $\mathfrak{R}$. 
Let $\{\varepsilon_n\}_{n\geq 0}$ 
be a strictly decreasing sequence of 
positive non zero real numbers. For all $n$ we consider a 
sub-family 
$\Sigma_{n}\subseteq\Sigma_{\varepsilon_n}$ which 
is non degenerate over $C_{\varepsilon_n}$.
 
We say that a differential equation is $(\Sigma_n)_n$-compatible if 
for all $n$ one has 
$\Sigma_n\subseteq\Sigma_{\varepsilon_n}(\Fs)$.

Let $\mathcal{C}$ be the category formed 
by differential equations $\Fs$ that are 
$(\Sigma_n)_n$-compatible.\footnote{It is a 
full subcategory of the category of all differential 
equations.} 
Then the restriction of the deformation functor to 
$\mathcal{C}$ 
is fully faithful.\hfill$\Box$
\end{proposition}
\begin{remark}
This equivalence does not require any additional assumption on the 
objects as solvability, Frobenius structure, or non Liouville exponents.
\end{remark}
\if{\subsubsection{A conjecture on non degeneracy.}
Assume that  
$\sup_{\sigma\in\Sigma}\R_{\emptyset}(x_{0,\rho},\sigma)<1$.
Then all solvable differential equations are $\Sigma$-compatible (cf. 
\eqref{eq : solvability over Robba} just below).
Moreover for a solvable differential equation we have 
$\Sigma_\varepsilon=\Sigma_\varepsilon(\Fs)$ for all 
$\varepsilon>0$. So the condition given in 
Proposition \ref{Prop: fully faithfulness over Robba}
only involves $\Sigma_\varepsilon$. We are induce to give the 
following definition
\begin{definition}
We say that $\Sigma$ is \emph{weakly non degenerate} 
if for all $\varepsilon>0$ there exists $0<\varepsilon'<\varepsilon$
such that the family $\Sigma_{\varepsilon'}$ is non degenerate over 
$C_{\varepsilon'}$.
\end{definition}
\begin{conjecture}
If the family $\Sigma$ is weakly non degenerate, then the 
deformation functor is fully faithful.
\end{conjecture}
We clarify the meaning of the conjecture, by the following example.

If $\bs{\mu}_{p^{\infty}}$ acts  on $\mathfrak{R}$ as 
$\sigma_q:f(T)\mapsto f(qT)$, $q\in\bs{\mu}_{p^{\infty}}$ ($q$-
difference with $q$ being a root of unity). 
Then the action of $\bs{\mu}_{p^{\infty}}$ is infinitesimal, and we 
have $\bs{\mu}_{p^{\infty}}=\Sigma=\Sigma_\varepsilon$ for all 
$\varepsilon >0$. A differential equation is $\Sigma$-compatible 
if and only if it is solvable (cf. \eqref{eq : solvability over Robba} 
below). Indeed it is known that 
$\R_{\emptyset,1}(x_{0,\rho},\Fs)=\rho^\beta$, for all $\rho$ close 
to $1$, while $\R_{\emptyset}(x_{0,\rho},\sigma_q)=|q-1|$. 

Now if the slope of the radius function $\R_{S,1}(-,\Fs)$ is not $0$, 
then $\Sigma_\varepsilon(\Fs)=\bs{\mu}_{p^n}$, for some 
$n\neq\infty$. 
Hence   
the action of $\Sigma_\varepsilon(\Fs)$ is never non degenerate. 
Indeed there is no point $x$ of $C_\varepsilon$ such that 
\eqref{eq : O(D)^Sigma=Omega} holds.

However, if $\Fs$ is defined over some larger annulus 
$A_{\varepsilon}=\{1-\varepsilon<|T|(x)<1+\varepsilon\}$, 
then $\Fs$ remains $\Sigma$-compatible also at $x_{0,1}$, and 
$\bs{\mu}_{p^\infty}$ is always non degenerate over 
$A_{\varepsilon}$ (we have \eqref{eq : O(D)^Sigma=Omega} 
at $x=x_{0,1}$).

The existence of a lattice over $A_\varepsilon$ will be realized for 
finite free differential equations with a Frobenius structure 
\smallcomment{REF}. In this case we will prove that the deformation 
is an equivalence \smallcomment{REF}. 
The conjecture is true for this class of equations.
}\fi

\subsection{Solvability and Frobenius structure}
\label{Section : Frob str sol}
We now assume hypothesis \ref{Hyp : K discr}.
For $\varepsilon>0$ we set
$A_\varepsilon:=\{1-\varepsilon<|T|<1+\varepsilon\}$, 
and  $A_0:=\{|T|=1\}$. 
Similarly as in the case of Robba ring we set 
$\O^\dag(A_0):=\cup_{\varepsilon>0}\O(A_\varepsilon)$. 
Its elements are power series 
$f=\sum_{i\in\mathbb{Z}}a_iT^i$, $a_i\in K$, such that 
there is $\varepsilon>0$ such that 
$\lim_{n\to\pm\infty}|a_i|(1+\varepsilon)^n=0$,
and $\lim_{n\to\pm\infty}|a_i|
(1-\varepsilon)^n=0$. 
A differential equation $\Fs$ over $\O^\dag(A_0)$ or 
$\mathfrak{R}$ is called solvable if 
\begin{equation}\label{eq : solvability over Robba}
\lim_{\rho\to 1^-}\R_{\emptyset,1}(x_{0,\rho},\Fs)\;=\;1\;.
\end{equation}

We now focus on the Frobenius structure. Let $\phi_K:K\to K$ be a 
lifting of the $p$-th power map $x\mapsto x^p$ of the residual field 
$\widetilde{K}$ of $K$. 
Let $A$ be one of the rings $\Od$ or $\mathfrak{R}$. 
Let $\phi(T)\in A$ be a function such that 
$x_{0,\rho}(\phi(T)-T^p)<\rho$ for all $\rho$ close to $1$. The 
setting $\sum a_i T^i\mapsto \sum \phi_K(a_i)\phi(T)^i$ is a ring 
endomorphism of $A$ called a \emph{Frobenius}.
We say that a differential equation $\Fs$ over $A$ has a 
Frobenius structure of 
order $n>0$ if there is an isomorphism of differential 
modules $(\phi^n)^*(\Fs)\xrightarrow{\sim}\Fs$.
A differential equation admitting an 
unspecified Frobenius structure is solvable, and enjoy 
several nice properties, one of them is the quasi unipotence that we will explain 
in the next sections.

\if{
\begin{remark}
We can transport this to the category of stratified 
$\Sigma$-modules. Indeed since the underling $\O_X$-module of 
$\Def_\Sigma(\Fs)$ is $\Fs$ we have
\begin{equation}
\phi^*\Def_\Sigma(\Fs)\;=\;\Def_\Sigma(\phi^*\Fs)\;.
\end{equation}
The Frobenius isomorphism 
$(\phi^n)^*(\Fs)\xrightarrow{\;\sim\;}\Fs$ 
being a morphism of differential modules, 
it is also a morphism of the $\Sigma$-modules obtained by 
deformation.
\end{remark}
}\fi
\subsection{Special extensions}
\label{Section : Special ext}
By a result of Katz (cf. \cite{Katz-local-to-global}), 
 finite separable Galois extensions of 
$\widetilde{K}(\!(t)\!)$ correspond to the so called 
\emph{special coverings} of 
$\mathbb{G}_{m,\widetilde{K}}=
\mathrm{Spec}(\widetilde{K}[t,t^{-1}])$. 
We now recall the definitions of 
\cite{Matsuda-unipotent}.

A finite étale coverings $\widetilde{V}\to\mathbb{G}_{m,
\widetilde{K}}$ 
is \emph{special } if it is tame at $0$ and if its geometric monodromy 
group has a unique $p$-Sylow subgroup 
(cf. \cite{Katz-local-to-global}). One proves that $\widetilde{V}$ 
is affine. If $\widetilde{\B}$ is its algebra, we say that 
$\widetilde{\B}/\widetilde{K}[t,t^{-1}]$ is a 
\emph{special extension}.
By the theory of Monsky-Washnitzer 
(cf. \cite{Monsky-Washnitzer-F1}), 
\emph{special extensions} of $\widetilde{K}[t,t^{-1}]$ can be 
lifted (preserving the Galois group) to the so called Special 
extensions of $K^{\circ}[T,T^{-1}]^{\dag}$, where 
\begin{equation}\label{eq : Od^circ}
K^{\circ}[T,T^{-1}]^{\dag}\;=\;\{
f=\sum_{i\in\mathbb{Z}}a_iT^i,\;a_i\in K^{\circ},\;
\exists\varepsilon>0\;\lim_{n\to\pm\infty}|a_i|
(1+\varepsilon)^n=0\;,\;\lim_{n\to\pm\infty}|a_i|
(1-\varepsilon)^n=0\}
\end{equation} 
is the Monsky-Washnitzer's weak completion of 
$K^{\circ}[T,T^{-1}]$. 
Special extensions of $K^{\circ}[T,T^{-1}]^{\dag}$ 
produce (by scalar extension) the so called 
\emph{Special extensions} of 
$\Od=
K^{\circ}[T,T^{-1}]^{\dag}\otimes_{K^\circ}K$. 

The $\mathfrak{R}$-algebras obtained by scalar extension from 
\emph{Special extensions} of $\Od$ 
will be called  
\emph{étale extensions} of $\mathfrak{R}$. 
We need to introduce the sub-ring of bounded functions in 
$\mathfrak{R}$:
\begin{equation}
\Ed\;\;:=\;\;\{\;f\in\mathfrak{R}\;|\; 
\lim_{\rho\to 1^{-}}x_{0,\rho}(f)<+\infty\;\}\;.
\end{equation}
The ring $\Ed$ has two topologies. 
The first one arises by restriction from that of $\mathfrak{R}$ (which 
is a $\mathcal{LF}$ space as inductive limit of the Frechet spaces 
$\O(C_\varepsilon)$). 
For this topology $\Ed$ is dense in $\mathfrak{R}$. 
The second topology on $\Ed$ is given by the Gauss norm $x_{0,1}$, 
for which $\Ed$ is \emph{not complete}. 
By the fact that the valuation of $K$ is discrete one  
proves that $(\Ed,x_{0,1})$ is a \emph{Henselian} 
field with residual field 
$\widetilde{K}(\!(t)\!)$. 
One has the following inclusions 
\begin{equation}
\Od\;\subset\; 
\Ed\;\subset\;\mathfrak{R}\;.
\end{equation}
We have introduced $\Ed$ because it is a field, and 
because it is an intermediate object between 
$\Od$ and $\mathfrak{R}$. 
Special extensions of $\Od$ and 
$\mathfrak{R}$ correspond bijectively to 
unramified extensions of $\Ed$. 
The situation is resumed in the following diagram 
(for more details we refer to \cite{An-DV}, \cite{Matsuda-unipotent}):
\begin{equation}\label{eq : diagram special extensions}
\xymatrix{ \protect{\left\{
\begin{array}{l}
\textrm{Special}\\
\textrm{extensions of }\Od
\end{array}
\right\}
\ar[r]_-{\sim}^-{-\otimes\Ed}\ar@{}[dr]|-{\odot} }&
\protect{\left\{\begin{array}{l}
\textrm{Finite unramified}\\
\textrm{extensions of }\Ed
\end{array}
\right\}}\ar[r]_-{\sim}^-{-\otimes\mathfrak{R}}&\protect{\left\{\begin{array}{l}
\textrm{Special}\\
\textrm{extensions of }\mathfrak{R}
\end{array}
\right\}}\\
\protect{\left\{
\begin{array}{l}
\textrm{Special extensions}\\
\textrm{of }K^\circ[T,T^{-1}]^\dag
\end{array}
\right\} \ar[r]_-{\sim}^-{-\otimes(\Ed)^\circ}\ar@{}[dr]|-{\odot}
\ar[u]_{\wr}^{-\otimes K}
\ar[d]_{-\otimes \widetilde{K}}^{\wr}  }&%
\protect{\left\{\begin{array}{l}
\textrm{Finite unramified}\\
\textrm{extensions of }(\Ed)^\circ
\end{array}
\right\}\ar[d]^{-\otimes \widetilde{K}}_{\wr}\ar[u]_{-\otimes K}^{\wr} } &\\
\protect{\left\{
\begin{array}{l}
\textrm{Special}\\
\textrm{coverings of }\widetilde{K}[t,t^{-1}]
\end{array}
\right\} \ar[r]_-{\sim}^-{\textrm{Pull-back}} }& \protect{\left\{
\begin{array}{l}
\textrm{Finite separable}\\
\textrm{extensions of }\widetilde{K}(\!(t)\!)
\end{array}
\right\}}.& }
\end{equation}
\begin{hypothesis}
From now on we fix an algebraic closure $\mathrm{Frac}(\mathfrak{R})^{\mathrm{alg}}$ 
of $\mathrm{Frac}(\mathfrak{R})$ and we consider only 
Special (resp. unramified, étale) extension of 
$\Od$, (resp. $\Ed$, $\mathfrak{R}$) inside it. 
\end{hypothesis}

\begin{remark}\label{Rk : power series as in N}
An unramified extension of $\Ed$ is again a ring of power series 
of the same type as those in 
$\Ed$ with respect to another variable, and another base field $L$ 
which is a finite unramified extension of $K$ 
(cf. \cite{Matsuda-unipotent}). 
The same holds for special extensions of $\mathfrak{R}$.
\end{remark}

The results we are going to use hold after replacing the base field $K$ 
by a finite unspecified extension $L/K$. For any $K$-algebra $\Ring$ 
and all finite field extension $L/K$ we set
\begin{equation}
\Ring_L\;:=\;\Ring\otimes_KL\;,\qquad 
\Ring_{K^{\mathrm{alg}}}\;:=\;\bigcup_{L/K \textrm{ finite}} 
\Ring_L\;.
\end{equation}
If $\Ring$ is one of the above differential rings  $\Od$, 
$\Ed$, or $\mathfrak{R}$, 
then a differential module $\Fs$ over 
$\Ring_{K^{\mathrm{alg}}}$ comes by scalar extension 
from a 
differential module over $\Ring_L$ for some finite 
extension $L/K$. 
So, by deformation, the same holds for \emph{stratified} 
$\Sigma$-modules.

\subsubsection{Deformation over $\Od$.}
Below we work with differential equations and $\Sigma$-modules 
over the rings $\mathfrak{R}$, $\Ed$, 
$\Od$. Since we need to interchange the 
base ring, moving along the first line of 
\eqref{eq : diagram special extensions}, we fix once for all a family of 
infinitesimal automorphisms of $\Od$. 
The definition of infinitesimal automorphisms of $\Od$, 
and related ones, are obtained imitating the 
definitions  
of Section \ref{Section : Deformation Robba}, by replacing 
everywhere $C_\varepsilon$ by $A_\varepsilon$. 
We only notice that a differential equation $\Fs$ over $\Od$ is 
$\Sigma$-compatible if and only if for all $\sigma\in\Sigma$ we have 
$\R_{\emptyset}(x_{0,1},\sigma)<
\R_{\emptyset,1}(x_{0,1},\Fs)$.
Indeed by continuity the inequality remains true over 
some unspecified segment 
$]x_{0,1-\varepsilon},x_{0,1+\varepsilon}[$ of 
$A_\varepsilon$. 

\begin{remark}
An infinitesimal  automorphism of $\mathfrak{R}$ 
naturally acts on $\Ed$. Indeed it induces an automorphism of 
$C_\varepsilon$, so the composition of $\sigma$ 
with a bounded functions on $C_\varepsilon$ remains bounded.

If an automorphism $\sigma$ of  
$\Od$ is infinitesimal, 
then it is also an infinitesimal 
automorphism of $\mathfrak{R}$ (cf. Lemma 
\ref{Lemma : localization}), and hence of $\Ed$.

If $\Sigma\subseteq \mathrm{Aut}(\Od)$ 
is non degenerate as a family of automorphisms of $\mathfrak{R}$, 
then it is so also as a family of automorphisms of $\Od$. 
The converse is unclear. 
We pay attention to the fact that it does not seem 
automatic that \emph{non degeneracy} translates as well 
from $\Od$ to $\mathfrak{R}$.
\end{remark}

\subsubsection{Extension of $\Sigma$ to Special extensions.}
\label{Action of Sigma on the logarithm.}
Let $\Ring$ be one of the rings $\Od$, 
$\Ed$, $\mathfrak{R}$.
Let $\sigma$ be an infinitesimal automorphism of 
$\Od$. 
We will need to apply $\sigma$ to the formal symbol $\log(T)$.
For this we write  $\sigma(T)/T= 1+\frac{\sigma(T)-T}{T}$. Since 
$\sigma$ is infinitesimal 
$x_{1,1}((\sigma(T)-T)/T)=x_{1,1}(\delta_\sigma(T)/T)<1$. 
Then $\sigma(T)/T$ takes values in the disc $\mathrm{D}^-(1,1)$, 
and the composite function $\log(\sigma(T)/T)$ converges in the 
annulus of definition of $\sigma(T)$. 
We are then allowed to define it as
\begin{equation}\label{eq : sigma acts on log(T)}
\log(\sigma(T))\;\;:=\;\;\log(T)+\log(\sigma(T)/T)\;,\quad 
\textrm{with }\log(\sigma(T)/T)\in \Ring\;.
\end{equation}
The action of $\Sigma$ on Special extensions is described by the 
following
\begin{lemma}\label{Lemma : Extension of sigma}
Let $\sigma$ be an infinitesimal automorphism of 
$\Od$, let 
$\widetilde{\B}/\widetilde{K}[t,t^{-1}]$ 
be a Galois Special covering, and let $\B/\Od$ (resp. 
$\B\otimes_{\Od}\Ed/\Ed$) be the corresponding Special (resp. 
unramified) extension.  

Then $\sigma$ extends uniquely, up to composition with Galois 
automorphisms of $\mathrm{Gal}(\B/\Od)
\simto\mathrm{Gal}(\widetilde{\B}/\widetilde{K}[t,t^{-1}])$, 
to a continuous automorphism of 
$\B/\Od$ and to $\B[\log(T)]$ as in 
\eqref{eq : sigma acts on log(T)}. 

In particular there exists a unique extension of $\sigma$ inducing the 
identity on the residual ring $\widetilde{\B}$ of 
$\Od$. By uniqueness $\sigma$ commutes with the action of the 
Galois group. 

The same holds for $(\B\otimes_{\Od}\Ed)/\Ed$.
\end{lemma}
\begin{proof}
It follows from the formal properties of the Henselian couples 
\cite{Ray}.
\end{proof}

\subsection{Katz-Matsuda's canonical extension}
\label{Section  Katz-Matsuda can ext}
As above we assume hypothesis \ref{Hyp : K discr}.
Now we show how to obtain the analogues of the results of 
\cite{Matsuda-unipotent} about the canonical extension by 
$\Sigma$-deformation. 

\begin{notation}\label{Notation : d-Mod and settings}
For any ring with derivation $d:\Ring\to\Ring$, we denote by 
$d-\Mod(\Ring)$ the category of locally free $\Ring$-modules of finite 
type $\Fs$ together with a connection $\nabla:\Fs\to\Fs$ satisfying 
the Leibnitz rule with respect to $d$. 

If a Frobenius 
$\phi:\Ring\to\Ring$ is given, we denote by 
$d-\Mod(\Ring)^{(\phi)}$ the full subcategory formed by those $\Fs$ 
admitting an isomorphism $\phi^*\Fs\simto\Fs$ commuting with the 
connections (morphisms are not 
supposed to commute with the Frobenius).

If $\mathcal{C}(\Ring)$ is one of the above categories, by 
$\mathcal{C}(\Ring_{K^{\mathrm{alg}}})$ we mean
the inductive limit 
$\mathcal{C}(\Ring_{K^{\mathrm{alg}}}):=
\varinjlim_{L/K}\mathcal{C}(\Ring_L)$, where $L/K$ runs in the set 
of finite extensions of $K$.
\end{notation}

We quickly recall the context. 
In \cite{Katz-local-to-global} and \cite{Katz-Can} Katz proved that, 
if $F$ is an arbitrary field of characteristic $0$, each formal differential 
equation over $F(\!(T)\!)$ comes by scalar extension from a 
differential equation over $F[T,T^{-1}]$. More precisely 
there exists a full sub-category of $d/dT-\Mod(F[T,T^{-1}])$, 
formed by the so called Special objects, which is equivalent to 
$d/dT-\Mod(F(\!(T)\!))$ via the scalar extension functor 
$d/dT-\Mod(F[T,T^{-1}])\to d/dT-\Mod(F(\!(T)\!))$.  
The section of the scalar extension functor so obtained is called 
\emph{canonical extension}: 
\begin{equation}\label{eq : cans ext Katz}
\mathrm{Can}\;:\;d/dT-\Mod(F(\!(T)\!))
\;\xrightarrow{\quad}\;
d/dT-\Mod(F[T,T^{-1}])\;.
\end{equation}

Now the rings $\Od$ and $\mathfrak{R}$ are considered 
as natural liftings in characteristic $0$ of 
$\widetilde{K}[t,t^{-1}]$ and 
$\widetilde{K}(\!(t)\!)$ respectively. 
Differential equations over $\Od$ 
and $\mathfrak{R}$ with (unspecified) Frobenoius 
structure are considered as $p$-adic analogues of the 
Katz's context. Along this analogy S.Matsuda 
(cf. \cite{Matsuda-unipotent}) shows the $p$-adic 
analogue of the above result for quasi-unipotent differential modules.
\begin{definition}
\label{Def : quasi unip}
Let $\Ring$ be one of the rings $\Od$, $\Ed$, 
$\mathfrak{R}$.

We denote by $U_n$ the free $\Ring$-module of rank $n$ with 
connection $\nabla^{U_n}:U_n\to U_n$ given in the basis 
$e_1,\ldots,e_n$ by $\nabla^{U_n}(e_i)=T^{-1}e_{i+1}$ for all 
$i=1,\ldots,n$ and $\nabla^{U_n}(e_n)=0$.
A differential module $\Fs$ over $\Ring$ is called \emph{unipotent} if 
it is isomorphic to a direct sum of modules of type $U_n$.

$\Fs$ is called \emph{quasi-unipotent} if its scalar 
extension to an unspecified special extension $\B$ of $\Ring$ is 
isomorphic to $U\otimes_\Ring\B$, where $U$ is 
unipotent over $\Ring$.
\end{definition}
The result of 
S. Matsuda states the analogous of \eqref{eq : cans ext Katz} 
for quasi-unipotent differential modules with unspecified Frobenius 
structure. 
By the $p$-adic local monodromy theorem 
(cf. \cite[Section 7.3]{An}, 
\cite{Ked}, \cite{Me}) each differential module in 
$d/dT-\Mod(\mathfrak{R})^{(\phi)}$ becomes 
quasi-unipotent 
after an unspecified scalar extension of the ground field 
$K$. Putting these two results together 
we obtain that, the scalar extension functor 
$d/dT-\Mod(\Od_{K^{\mathrm{alg}}})^{(\phi)} 
\to d/dT-\Mod(\mathfrak{R}_{K^{\mathrm{alg}}})^{(\phi)}$ 
admits a 
section called \emph{canonical extension} 
\begin{equation}
\mathrm{Can}\;\;:\;\;   
d/dT-\Mod(\mathfrak{R}_{K^{\mathrm{alg}}})^{(\phi)} \;
\to\;d/dT-\Mod(\Od_{K^{\mathrm{alg}}})^{(\phi)}\;,
\end{equation}
which is an equivalence of categories with its essential image (cf. \cite[7.3]{An}). 
In particular, this implies that after 
base change to some finite extension $L/K$ all 
differential module 
with Frobenius structure over $\mathfrak{R}$ admits a 
basis in which the matrix of the connection lies in 
$M_n(\Od)$. 

\subsubsection{Deformation of the canonical extension.}
We now resume the straightforward consequence of the above 
results.

Let $\Sigma$ be a family of infinitesimal operators of 
$\Od$. For $\Ring=\Od$, or $\Ring=\mathfrak{R}$ let 
$d/dT-\Mod(\Ring)^{(\phi),\mathrm{comp}(\Sigma)}$ be the 
full sub-category of $d/dT-\Mod(\Ring)^{(\phi)}$ formed by 
$\Sigma$-compatible differential equations over $\Ring$ (cf. 
\eqref{eq : Sigma-compatibility}). 
This is also a full-sub-category of the category
$d/dT-\Mod(\Ring)^{\mathrm{comp}(\Sigma)}$ of 
$\Sigma$-compatible differential equations, hence
the scalar extension functor 
$d/dT-\Mod(\Od)^{(\phi),\mathrm{comp}(\Sigma)} 
\to d/dT-\Mod(\mathfrak{R})^{(\phi),\mathrm{comp}(\Sigma)}$ 
commutes with the deformation functors. 
As a consequence the canonical extension functor also commutes with 
the deformations as soon as the deformations are equivalences. 
Namely assume that $\Sigma$ is a family of infinitesimal 
automorphisms of $\Od$ which is non degenerate as a family of 
automorphisms of $\mathfrak{R}$. 
Fix a sequence $(\Sigma_n)_n$, as in Proposition 
\ref{Prop: fully faithfulness over Robba}, where $\Sigma_n$ is non 
degenerate over $C_{\varepsilon_n}$. 
We have a commutative diagram:
\begin{equation}\label{Thm : Can ext Mats diff}
\xymatrix{d/dT-\Mod(\Od)^{(\phi),\mathrm{comp}
(\Sigma_n)_n}
\ar[d]^-{\wr}_-{\Def_{\Sigma}}
\ar[rr]^-{\mathrm{Can}}&&
d/dT-\Mod(\mathfrak{R})^{(\phi),\mathrm{comp}
(\Sigma_n)_n}
\ar[d]_-{\wr}^-{\Def_{\Sigma}}\\
(\Sigma_n)_n-\Mod(\Od)^{(\phi),
\mathrm{strat}}\ar@{-->}[rr]^-{\mathrm{Can}}&&(\Sigma_n)_n-\Mod(\mathfrak{R})^{(\phi),
\mathrm{strat}}}
\end{equation}
where, for $\Ring=\Od,\mathfrak{R}$, we denote by 
$d/dT-\Mod(\Ring)^{(\phi),\mathrm{comp}(\Sigma_n)_n}$ 
the full subcategory formed by $(\Sigma_n)_n$-compatible 
differential equations, and by 
$(\Sigma_n)_n-\Mod(\Ring)^{(\phi),\mathrm{strat}}$ we 
denote the corresponding category of stratified 
$(\Sigma_n)_n$-modules (i.e. its essential image by Deformation). 

\subsection{$p$-adic local monodromy Theorem}
\label{Section p-adic loc mon th}
We maintain the assumption \ref{Hyp : K discr}.
In this section we prove the $p$-adic local monodromy theorem for 
\emph{stratified} $(\Sigma_n)_n$-modules over $\mathfrak{R}$.

\begin{setting}\label{Setting : plmt-diff}
Let $\phi$ be a Frobenius of $\Od$, and let $\Sigma$ be a non 
degenerate family of infinitesimal automorphisms of $\Od$ which is 
non degenerate as a family of infinitesimal automorphisms of  
$\mathfrak{R}$. Fix a sequence $(\Sigma_n)_n$, as in Proposition 
\ref{Prop: fully faithfulness over Robba}, where $\Sigma_n$ is non 
degenerate over $C_{\varepsilon_n}$.
\end{setting}
\begin{theorem}[$p$-adic local monodromy theorem for stratified 
$\Sigma$-modules]
\label{plmt-diff}
Each stratified $(\Sigma_n)_n$-module over $\mathfrak{R}$ 
admitting an (unspecified) Frobenius structure 
becomes quasi-unipotent after a base change to 
$\mathfrak{R}_L$, for some unspecified finite extension $L/K$.
\end{theorem}
\begin{proof}
The claim says that each object $(\Fs,\Sigma)$ in 
$(\Sigma_n)_n-\Mod(\mathfrak{R}_{K^{\mathrm{alg}}})^{(\phi),
\mathrm{strat}}$ is trivialized by 
$\mathfrak{R}'[\log(T)]$, where $\mathfrak{R'}/\mathfrak{R}_L$ is 
some Special extension of $\mathfrak{R}_L$. 
This means that $(\Fs,\Sigma)$ has a complete basis of 
solutions in $\mathfrak{R}'[\log(T)]$. We know that deformation 
preserves Taylor solutions (cf. Remark 
\ref{Rk : Same Taylor solutions}), the strategy is to prove that 
deformation also preserves \emph{étale} solutions, i.e. solutions in 
some $\mathfrak{R}'[\log(T)]$.

By \cite[Cor.7.1.6]{An}, up to enlargements of  $K$, every differential 
module $\Fs$ with Frobenius structure over $\mathfrak{R}$ 
is a direct sum of sub-modules of the form $\N\otimes\U_m$ where 
$\N$ is trivialized by an étale extension $\mathfrak{R}'/\mathfrak{R}$ 
(without logarithm) of $\mathfrak{R}$, 
and $(\U_m,\nabla^{\U_m})$ is the $m$-dimensional unipotent 
differential module over $\mathfrak{R}$ 
(cf. Def. \ref{Def : quasi unip}). 
Since the deformation equivalence preserves this decomposition, we 
can assume $\Fs=\N$ or $\Fs=\U_m$. 
If we are in the first case $\Fs=\N$, we will say that $\Fs$ has 
\emph{finite local monodromy}.

We first prove the result for $\U_m$. It is well 
known that $\U_m$ is trivialized by $\Od[\log(T)]$, where 
$\log(T)$ is an indeterminate (i.e. merely a symbol).
We now consider $\log(T)$ as a function over the disc $D^-(1,1)$. It 
is not algebraic over $\Od$, so the restriction map
\begin{equation}
\Od[\log(T)]\;\xrightarrow{\quad}\;\O(D^-(1,1))
\end{equation}
is an injective ring morphism commuting with $d/dT$, 
and $\Sigma$. 
This map identifies Taylor solutions of $\U_m$ at $T=1$ with 
``\emph{abstract}'' solutions of $\U_m$ in $\Od[\log(T)]$.
Since on the right hand side these solutions are simultaneously 
solutions of the differential equation $\U_m$ and of its deformation 
$\Def_\Sigma(\U_m)$ (cf. Remark \ref{Rk : Same Taylor solutions}), 
the same holds on the left hand side. Hence 
$\Def_\Sigma(\U_m)$ is trivialized by $\Od[\log(T)]$.

We now focus on differential modules $\N$ 
with finite local monodromy. 
As for $\U_m$ we now embed the Special extension trivializing $\N$ 
into $\O(D^-(1,1))$ and we will compare Taylor solutions with étale 
solutions as above. 

Up to enlarge $K$, we may assume that $(\N,\nabla)$ is 
an $n$-dimensional differential module over $\mathfrak{R}$ 
trivialized by some étale extension $\mathfrak{R}'/\mathfrak{R}$.
Let $\B/\Od$ be the corresponding Special extension of $\Od$.
By canonical extension (cf. \cite[Cor. 5.12]{Matsuda-unipotent}) 
$\N$ comes, by scalar extension, from a differential module 
$\N_0:=\mathrm{Can}(\N)$ over $\Od$ which is trivialized by $\B$.
Let $Y_{\B}\in GL_n(\B)$ be a complete basis of solutions of 
$\N_0$ with values in $\B$. 

For all $\sigma\in\Sigma$, let $\sigma_{\B}:\B\to\B$ be the 
corresponding endo-morphism of $\B$ 
(cf. Lemma \ref{Lemma : Extension of sigma}). 
We define the matrix $A_{\sigma_{\B}}$ 
of the action of $\sigma_{\B}$ by 
\begin{equation}
\sigma_{\B}(Y_{\B})
\;=\;A_{\sigma_{\B}}\cdot Y_{\B}\;.
\end{equation}
Since $Y_{\B}$ is invertible, so does 
$A_{\sigma_{\B}}=\sigma_{\B}(Y_{\B})\cdot Y_{\B}^{-1}$. 
Moreover $A_{\sigma_{\B}}\in GL_n(\Od)$ because
the Galois group commutes with the unique extension of $d/dT$ 
to $\B$, so it acts on $Y_{\B}$ by right multiplication by matrices in 
$GL_n(K)$. Hence, since the Galois group 
also commutes with $\sigma_{\B}$,
$A_{\sigma_{\B}}$ is stable by Galois.

Now we denote by $A_\sigma\in GL_n(\Od)$ the matrix of the action 
of $\sigma$ on $\N_0$ obtained by deformation of $\nabla$. 
Namely let $x\in A_0$ be any rational point, 
and let $D\subseteq A_0$ be the open disc with radius 
$1$, centered at $x$. 
Since $\N_0$ is $\Sigma$-compatible, 
$(\N_0,\nabla)$ is trivial on $D$. 
If $Y\in GL_n(D)$ is a Taylor solution matrix of $\N_0$, 
then $A_\sigma$ is defined by 
\begin{equation}
\sigma(Y)\;=\;A_\sigma\cdot Y\;.
\end{equation}

Consider now the reduction $\widetilde{x}$ of $x$
in $\mathbb{G}_{m,\widetilde{K}}$. 
If $\widetilde{V}\to \mathbb{G}_{m,\widetilde{K}}$ is the Special 
covering 
corresponding to $\B$, then its fiber at $\widetilde{x}$ is a finite 
étale covering of $\widetilde{x}$. Up to replacing $K$ by a finite 
extension, this is given by a trivial covering of finite degree $d$. Let 
$\widetilde{y}_1,\ldots,\widetilde{y}_d$ be the points of 
$\widetilde{V}$ over $\widetilde{x}$. 

By \cite[Prop. 2.2, (i)]{Bosh-Lutk-Stable-Reduction-I} 
each $\widetilde{y}_i\in \widetilde{V}$ 
lifts into an open disc $D_i$ 
contained in the dagger affinoid $V^\dag$ 
corresponding to $\B$. Since the morphisms $\psi:V^\dag\to\mathbb{G}_{m,K}^{\dag}$ 
is finite with same degree $d$, it induces a trivial covering over $D$. In particular $\psi$ induces an isomorphism 
$\psi:D_i\simto D$ for all $D$.

Now $\sigma_{\B}$ induces an automorphism of each $D_i$ because 
the reduction of $\sigma_{\B}$ is the identity on $\widetilde{\B}$. 
Since $\sigma_{\B}$ lifts $\sigma$, then $\psi:D_i\simto D$ 
commutes with $\sigma_{\B}$ and $\sigma$.
Moreover if we identify in this way $D_i$ with $D$, 
then $Y=Y_{\B}\cdot C$, for some $C\in GL_n(K)$, by the 
uniqueness of Taylor solutions. Hence 
\begin{equation}
A_{\sigma_{\B}}\;=\;A_\sigma\;.
\end{equation}

Now the Taylor solutions of a differential equation are also solutions 
of its $\Sigma$-deformation by Remark 
\ref{Rk : Same Taylor solutions}. 
The base change by the matrix $Y$ in $(\N_0)_{|D}$ trivializes the differential equation over 
$D$, and hence simultaneously all the actions of $\sigma$ obtained 
by deformation. In the new basis we find $A_\sigma=\mathrm{Id}$ 
for all $\sigma\in\Sigma$. 

We now look at $\psi^*\N_0:=\N_0\otimes_{\Od}\B$ over $V^\dag$. 
The entries of $Y_{\B}$ are global sections on $V^\dag$ 
that coincide with $Y$ over $D_i$.
Since $A_{\sigma_{\B}} = A_\sigma$, the base change by 
$Y_{\B}$ (which trivializes the differential equation) gives a new matrix $A_{\sigma_{\B}}'$ 
which is the identity over $D$. By analytic continuation, 
the matrix $A_{\sigma_{\B}}'$ is the identity everywhere 
over $V^\dag$, so this base change 
trivializes the entire action of 
$\Sigma$ over $V^\dag$.
\end{proof}

The proof shows in particular the following result:
\begin{corollary}
Let $\M$ be a differential module over $\Od$ with unspecified 
Frobenius structure. 
Assume that $\M$ is trivialized by some Special extension 
of $\Od_L$, for some finite extension $L/K$. 
Then $\M$ has bounded Taylor solutions on each disc 
$D\subset A_0$.
\end{corollary}
\begin{proof}
Indeed $Y_{\B}$ is the restriction of a global solution over $V^\dag$, 
so it is bounded on each $D_i$.
\end{proof}
Such differential modules are unit-root by 
\cite{Matsuda-unipotent}. The fact that a unit root differential 
modules has bounded solutions is a well known result, at least since 
\cite{Katz-Travaux-de-Dwork} (see also 
\cite{Chi-Tsu-log-growth-1}, \cite{Chi-Tsu-log-growth-2}).

The proof also gives the following nice result, that could be helpful to 
work explicitly with Special extensions of $\Od$:
\begin{proposition}
If $\B$ is a Special extension of $\Od$, 
there exists an injective ring morphism
\begin{equation}
\mathrm{Tay}_1\;:\;\B_{K^{\mathrm{alg}}}\;\xrightarrow{\quad}\;
\O(D^-(1,1))_{K^{\mathrm{alg}}}
\end{equation}
commuting with $d/dT$, the Frobenius, and $\Sigma$.
\end{proposition}
\begin{proof}
With the notations of the proof of Theorem \ref{plmt-diff}, if 
$D=D^-(1,1)$, we first consider the restriction from $\B$ to 
$\O(D_i)$, and then we apply the pull-back by 
$\psi^{-1}:D\simto D_i$.
\end{proof}
As a last result, we give the following converse of Theorem 
\ref{plmt-diff}, which is a characterization of the category of stratified 
$(\Sigma_n)_n$-modules:
\begin{corollary}
\label{Corollary : Essential image Case of Robba}
We preserve the assumption \ref{Setting : plmt-diff}.
Let $\Fs$ be a $(\Sigma_n)_n$-compatible differential equation over 
$\mathfrak{R}$, together with an action of $\Sigma$. 
Assume that there exists a finite extension $L/K$ and an étale 
extension $\mathfrak{R'}/\mathfrak{R}_L$ such that 
$\Fs\otimes_{\mathfrak{R}}\mathfrak{R}'[\log(T)]$ 
has a basis on which the connection and the action of $\Sigma$ 
are both trivial.
Then the action of $\Sigma$ coincides with that obtained by 
deformation from $\nabla$.
\end{corollary}
\begin{proof}
By Theorem \ref{plmt-diff}, the action of $\Sigma$ obtained by 
deformation becomes trivial in the same basis trivializing the 
connection. So the two actions of $\Sigma$ coincide after base 
change, hence they were equal before base change.
\end{proof}

\section{Difference equations over the affine line.}
\label{Difference equations over the affine line.}
We now investigate a particular class of automorphism, 
those of the form  $\sigma_{q,h}(T)= qT+h$,
for some $q,h\in K$. If $q\neq 0$, this is an automorphism of 
$\mathbb{A}^{1,\mathrm{an}}_K$ with inverse 
$\sigma_{q^{-1},-q^{-1}h}$. 

If $q=1$ we say that $\sigma_{1,h}$ is a 
finite difference operator, and 
if $h=0$ we say that $\sigma_{q,0}$ is a 
$q$-difference operator. 
In general we say that $\sigma_{q,h}$ is a 
difference operator.

\subsection{$S$-infinitesimality of $\sigma_{q,h}$.}
\label{Section : infinitesimality of sigma_qh}
In general, for $q_1,q_2\neq 0$, 
we have
\begin{equation}
\sigma_{q_1,h_1}\circ\sigma_{q_2,h_2}\;=\;
\sigma_{q_1q_2,q_2h_1+h_2}\;.
\end{equation}
We can define a group operation on 
$\mathcal{G}:=\mathbb{G}_m^{\mathrm{an}}
\times \mathbb{A}_K^{1,\mathrm{an}}$ by 
$(q_1,h_1)(q_2,h_2):=(q_1q_2,q_2h_1+h_2)$. Since the operation 
are given by polynomials, this is a $K$-analytic group. 
Also the action $\mathcal{G}\times \mathbb{A}^{1,\mathrm{an}}_K\to 
\mathbb{A}^{1,\mathrm{an}}_K$ given 
by $((q,h);T)\mapsto qT+h$ is 
given by polynomials, so it is a morphism of $K$-analytic 
spaces 
as in Section \ref{Analyticity of the action}.
If $q\neq 1$, $\sigma_{q,h}$ has a unique fixed rigid point which is 
\begin{equation}\label{eq : a=-h/(q-1)}
a\;:=\;-h/(q-1)\;.
\end{equation}
Moreover by a translation sending 
$a$ into $0$, $\sigma_{q,h}$ 
become just the multiplication by $q$: 
$\sigma_{q,h}(T-a)=q(T-a)$. 
This often  permits to reduce to the case where $h=0$. 
We then deduce the following
\begin{lemma}
$\sigma_{q,h}$ extends to an automorphism of 
$\mathbb{P}^{1,\mathrm{an}}_K$.
If $q=1$, then $+\infty$ is its unique fixed rational point in 
$\mathbb{P}^{1,\mathrm{an}}_K(K)$. 
If $q\neq 1$, then $a$ and 
$+\infty$ are its unique fixed points in 
$\mathbb{P}^{1,\mathrm{an}}_K(K)$. \hfill$\Box$
\end{lemma}
\if{
\begin{proof}
If $q=1$, express $\mathbb{P}^{1,\mathrm{an}}_K$ as 
$D^+(0,1)$ pasted with itself via the map $T\mapsto T^{-1}=Z$ 
along $\{|T|(x)=1\}$. Call $D_\infty$ the unit closed disc at 
$+\infty$. 
Then $\sigma_{1,h}(Z)=\sigma_{1,h}(T^{-1})=(T+h)^{-1}=T^{-1}(1+h/T)^{-1}= Z\sum_{n\geq 0}(-hZ)^n$.

Assume now that $q\neq 1$. Let $a=-h/(q-1)$. 
Express $\mathbb{P}^{1,\mathrm{an}}_K$ as a closed unit disc $D$
centered at $a$ pasted with itself by the map 
$(T-a)\mapsto (T-a)^{-1}$ along the annulus $|T-a|=1$. 
Denote by $D_\infty$ the corresponding closed unit disc at the infinity 
so obtained. Now $\sigma_{q,h}$ acts as a 
multiplication by $q$ on $D$ (with respect to the coordinate 
$T-a$), 
and by multiplication by $q^{-1}$ on $D_\infty$ (with respect to 
$(T-a)^{-1}$).
\end{proof}
}\fi
\begin{lemma}[Disks that are stable under 
$\sigma_{q,h}$]
\label{Lemma : discs stable by sigma_q,h}
The following hold:
\begin{enumerate}
\item If $|q|\neq 1$, then the restriction of 
$\sigma_{q,h}$ to a disc in 
$\mathbb{P}^{1,\mathrm{an}}_\Omega$ 
is never an automorphism of the disc. 
In particular $\sigma_{q,h}$ is not isometric.
\item If $|q|=1$, and if $|q-1|=1$, 
then the family formed by the open/closed discs 
$D^\pm(a,\rho)$, $\rho\geq 0$, $a=-h/(q-1)$, and by their 
complements in $\mathbb{P}^{1,\mathrm{an}}_\Omega$,
is the unique family of (open or closed) discs in 
$\mathbb{P}^{1,\mathrm{an}}_\Omega$ on which 
$\sigma_{q,h}$ induces an automorphism.
\item If $|q-1|<1$, the family of discs in 
$\mathbb{P}^{1,\mathrm{an}}_\Omega$ on which 
$\sigma_{q,h}$ induces an automorphism is formed by the discs 
$D^-(c,\rho)\subseteq \mathbb{A}^{1,\mathrm{an}}_\Omega$ 
(resp. $D^+(c,\rho)\subseteq \mathbb{A}^{1,\mathrm{an}}_\Omega$) 
satisfying  
\begin{equation}
|\sigma_{q,h}(c)-c|\;<\;\rho\qquad
(\textrm{resp. }|\sigma_{q,h}(c)-c|\;\leq\; \rho)\;,
\end{equation}
and by their complements in 
$\mathbb{P}^{1,\mathrm{an}}_\Omega$.
\item
In particular, if $D$ is a virtual open (resp. closed) disc with 
boundary $x$ which is stable under $\sigma_{q,h}$, then
each virtual open disc with boundary in $[x,+\infty[$ 
(resp. $]x,+\infty[$) is stable by $\sigma_{q,h}$.
\hfill$\Box$
\end{enumerate}
\end{lemma}
In the situation of point ii) of the Lemma \ref{Lemma : discs stable by sigma_q,h}, 
the unique differential equation which is 
$\sigma_{q,h}$-compatible over a discs centered at $a$, or over its complement, is the trivial 
one. So this case is not interesting from the point of view of this paper.

\begin{hypothesis}
From now on we assume $|q-1|<1$. 
In particular, if $q\neq 1$, the absolute 
value of $K$ is not trivial.
\end{hypothesis}

\begin{proposition}
\label{Prop.  : characterization of analytic domains stable by sigma_q,h}
Let $X\subseteq \mathbb{P}^{1,\mathrm{an}}_K$ 
be a connected 
analytic domain distinct from 
$\mathbb{P}^{1,\mathrm{an}}_K$, and 
$\mathbb{P}^{1,\mathrm{an}}_K-\{t\}$ 
for any point $t\in \mathbb{P}^{1,\mathrm{an}}_K$ of type $1$ or $4$. Then:
\begin{enumerate}
\item
The analytic skeleton $\Gamma_X$ of $X$ is 
the skeleton 
of a (not unique) weak triangulation (cf. Section 
\ref{Quasi-smooth Berkovich curves}).
Each other weak triangulation 
$S$ of $X$ verifies $\Gamma_X\subseteq\Gamma_S$.
\item Let $Y$ be a connected component of $X_\Omega$. 
Then each connected component of the complement of $Y$ in 
$\mathbb{P}^{1,\mathrm{an}}_\Omega$ is either 
an open or closed disc, or it is reduced to 
a point $x$ such that 
$x\in\overline{\Gamma_Y}-\Gamma_Y$, where 
$\overline{\Gamma_Y}$ is the closure of 
$\Gamma_Y$ in $\mathbb{P}^{1,\mathrm{an}}_K$. In particular 
$x$ is an end point of $\overline{\Gamma_X}$ of type $1$ or $4$.
\item $\sigma_{q,h}$ induces an automorphism of 
$X$ if and only if it induces an automorphism of the complement of $X_\Omega$ in 
$\mathbb{P}^{1,\mathrm{an}}_\Omega$.
Moreover, if $S$ is a weak triangulation of $X$, the action of 
$\sigma_{q,h}$ on $X$ is $S$-infinitesimal if and 
only if for all connected component $Y$ of $X_\Omega$ the following hold
\begin{enumerate}
\item $\sigma_{q,h}$ induces an automorphism of 
each connected component of $\mathbb{P}^{1,\mathrm{an}}_\Omega-Y$;
\item If $D$ is a virtual open disc in $X$ such that 
$S\cap D\neq\emptyset$, then for each point 
$x\in S_\Omega\cap D_\Omega$ which is an end-point of 
$\Gamma_{S_\Omega}$ there exists an open disc 
$D_x\subset D_\Omega$ with boundary $x$ which is globally 
fixed by $(\sigma_{q,h})_\Omega$.\hfill$\Box$
\end{enumerate}
\end{enumerate}
\end{proposition}
\if{
\begin{proof}
i)
If $S$ is any weak triangulation of $X$, then 
$X-\Gamma_S$ is a disjoint 
union of discs in $X$, so 
$\Gamma_X\subseteq\Gamma_S$. 
By assumption 
$\mathbb{P}^{1,\mathrm{an}}_K-X$ has at least two points.
This implies that $\Gamma_X\neq\emptyset$. Let $x_0$ be one of them.
If $x\in\Gamma_X$ then the whole 
segment 
$[x,x_0[\cap X$ is contained in $\Gamma_X$ 
because if a virtual open disc $D\subseteq X$ intersects $[x,x_0[\cap X$, 
then $D$ contains $x$. So $\Gamma_X$ is a tree, and since 
$\Gamma_S$ is locally finite, so is $\Gamma_X$. 

If now $D$ is a virtual open disc in $X$ with boundary in 
$x\in\Gamma_X$ such that $\Gamma_S\cap D\neq\emptyset$, the 
set $S':=(S\cup\{x\})-(S\cap D)$ is again a weak triangulation of $X$ 
that do not intersect $D$. In this way one sees that 
there exists a weak triangulation $S'$ of $X$ such that 
$\Gamma_X=\Gamma_{S'}$. This proves i).

ii) 
By the arguments of Section \ref{Extension of scalars.},
the existence of the map $\sigma_L : X_{\Ka}\to X_L$ implies a bijection 
between $\Gamma_{S_L}$ and $\Gamma_{S_{\Ka}}$. In particular
each connected component $Y$ of $X_{\Ka}$ remains connected after scalar 
extension to $\Omega$. 

Since $Y$ is an analytic domain, for all $x\in \Gamma_Y$ there are only 
a finite number of branches out of $x$ that do not belong to $Y$. 
Each such a branch corresponds to a disc in 
$\mathbb{P}^{1,\mathrm{an}}_K$. The claim ii) follows.

iii) $\sigma_{q,h}$ is an automorphism of $X$ if and only if 
$(\sigma_{q,h})_\Omega$ is an automorphism of $X_\Omega$.
The complement of $X_\Omega$ in 
$\mathbb{P}_\Omega^{1,\mathrm{an}}$ is a
$\Omega$-analytic space. 
Since $(\sigma_{q,h})_\Omega$ is an automorphism, $X_\Omega$ is 
stable under $(\sigma_{q,h})_\Omega$ if and only if its complement 
is.

The assertions (a) and (b) are proved as follows. 
Assume first that $\sigma_{q,h}$ acts $S$-infinitesimally on $X$. 
Then $\Gamma_S$ (resp. $\Gamma_{S_\Omega}$) 
is fixed by $\sigma_{q,h}$ (resp. $(\sigma_{q,h})_\Omega$). 

(a)
Let $C$ be a connected component of 
$\mathbb{P}^{1,\mathrm{an}}_\Omega-Y$ 
(either a open/closed disc or a point). 

If $C$ is an open disc, then its 
boundary $x$ in $\mathbb{P}^{1,\mathrm{an}}_\Omega$ 
lies in $\Gamma_{X_\Omega}$ by definition. Since 
$\Gamma_{X_\Omega}\subseteq\Gamma_{S_\Omega}$, 
then $x$ is fixed by $\sigma_{q,h}$. 
Moreover almost all open discs of 
$\mathbb{P}^{1,\mathrm{an}}_\Omega$ with boundary 
$x$ are connected component of 
$X_\Omega-\Gamma_{X_\Omega}$. So they are stable by 
$\sigma_{q,h}$ by the infinitesimality condition. 
Hence $C$ is stable too by point iv) of Lemma 
\ref{Lemma : discs stable by sigma_q,h}.

If now $C$ is a closed disc (resp. a point $x$), 
and if $x$ is its Shilov boundary, 
there is $y\in X_\Omega$ such that $]x,y[\subset X_\Omega$. 
If $X_\Omega$ is the complement of $C$ in 
$\mathbb{P}^{1,\mathrm{an}}_\Omega$, there is nothing to prove.
Otherwise $X_\Omega$ is not a disc, so any disc intersecting 
$]x,y[$ contains $C$. Hence we may find $y$ close 
enough to $x$ in order to have 
$]x,y[\subseteq\Gamma_{X_\Omega}$. 
Now the segment $]x,y[$ is fixed by $\sigma_{q,h}$, and by 
continuity, so is $x$. 
Then $\sigma_{q,h}(C)\cap C$ is not empty, and since $C$ is a 
connected component of the complement of $X_\Omega$, 
we have $\sigma_{q,h}(C)=C$.

(b) If $\Gamma_S=\Gamma_X$ we are done. 
Otherwise the disc $D_x$ is a connected component of 
$X_\Omega-\Gamma_{S_\Omega}$ so it is stable by the 
infinitesimality condition.

We now assume that (a) and (b) both hold, and we prove 
that $\sigma_{q,h}$ acts $S$-infinitesimally. 
Since $\infty$ is a fixed point by $\sigma_{q,h}$, 
we can replace $X$ by 
$X':=X\cap\mathbb{A}^{1,\mathrm{an}}_K$. 
If $\sigma_{q,h}$ acts infinitesimally on $X'$ then so it 
does on $X$. 

If $D_x$ is stable by $\sigma_{q,h}$ as in $(b)$, then 
$[x,+\infty[\cap Y\subseteq \Gamma_{S_\Omega}$. 

Otherwise if $C$ is a connected component of 
$\mathbb{A}^{1,\mathrm{an}}_\Omega-Y$ as in (a), 
then $C$ is an open disc (resp. a closed disc, or a point), and if $x$ is 
its boundary (resp. Shilov boundary), then
$I_x:=[x,+\infty[\cap Y\subseteq 
\Gamma_{S_\Omega}$
(resp. $I_x:=]x,+\infty[\cap Y\subseteq 
\Gamma_{S_\Omega}$). 

Now $\Gamma_{S_\Omega}$ is the union of all those 
segments $I_x$, and by point iv) of 
Lemma \ref{Lemma : discs stable by sigma_q,h} one sees that 
each open disc $D$ with boundary in $\Gamma_{S_\Omega}$ is stable by 
$\sigma_{q,h}$. So $\sigma_{q,h}$ acts infinitesimally on $X$.
\end{proof}
}\fi
\if{
\subsection{Deformation.}
Let $X$ be an analytic domain of 
$\mathbb{P}^{1,\mathrm{an}}_K$, with a weak 
triangulation $S$. 
Assume that $\sigma_{q,h}$ acts $S$-infinitesimally on 
$X$ (cf. Prop. 
\ref{Prop.  : characterization of analytic domains stable 
by sigma_q,h}). 

If $\Fs$ is a differential equation over $X$, the 
compatibility condition
is $\R_{S,1}(x,\Fs)>\R_{S}(x,\sigma)$ for all $x\in X$. 
We recall that we can test such a condition on a locally 
finite set of 
points and germs of segments 
(cf. Lemma \ref{Lemma : infinitesimality on S}). 

Moreover, if $\infty\notin X$ we have a global coordinate 
$T$ on $X$, 
and we notice that the function $\R_{S}(-,\sigma)$ is 
given by
\begin{equation}
\R_{S}(x,\sigma)\;=\;\frac{|(q-1)T+h|(x)}{\rho_{S,T}(x)},
\end{equation}
where $\rho_{S,T}(x)$ is defined as in Remark 
\ref{Remark : trivializing Radius over A^1}.
}\fi

\subsection{Non degeneracy and fully faithfulness}
\subsubsection{$q$-Taylor expansion.}
For all natural number  $n\geq 1$ we set 
\begin{equation}
[n]_q\;:=\;1+q+q^2+\cdots+q^{n-1}\;,\qquad
[n]_q^!\;:=\;[1]_q[2]_q[3]_q\cdots[n]_q\;.
\end{equation}
\if{
For $0\leq i\leq n$, we define the $q$-binomial $\tbinom{n}{i}_q$ by 
the relation
\begin{equation}
(1-T)(1-qT)\cdots(1-q^{n-1}T)\;=\;
\sum_{i=0}^n(-1)^i
\tbinom{n}{i}_q\cdot q^{\frac{i(i-1)}{2}}\cdot T^i\;,
\end{equation}
where it is understood that for $i=0,1$ the symbol 
$q^{\frac{i(i-1)}{2}}$ is equal to $1$. 
It follows from the definition that for all $1\leq i\leq n$ we have
\begin{equation}
\tbinom{n}{i}_q\;=\;\tbinom{n-1}{i-1}_q+q^i\tbinom{n-1}{i}_q
\;=\; 
q^{n-i}\tbinom{n-1}{i-1}_q+\tbinom{n-1}{i}_q\;.
\end{equation}
}\fi
For all $c\in \Omega$ we set $(T-c)_{q,h}^{[0]}=1$ and, for all 
$n\geq 1$ we set
\begin{equation}
(T-c)_{q,h}^{[n]}\;:=\;(T-c)(T-\sigma_{q,h}(c))
(T-\sigma_{q,h}^2(c))\cdots(T-\sigma_{q,h}^{n-1}(c))\;.
\end{equation}
We define the twisted $(q,h)$-derivative as 
$d_{q,h}(f):=
\frac{f\circ\sigma_{q,h}-f}{(q-1)T+h}$.
In particular $d_{q,h}$ is a $K$-linear map satisfying 
$d_{q,h}((T-c)_{q,h}^{[n]})=[n]_q(T-c)_{q,h}^{[n-1]}$, for all 
$n\geq 1$.
The denominator of $d_{q,h}$ has a zero. 
The following Proposition shows, in particular, that $d_{q,h}$ is 
well defined around that zero.
\begin{proposition}\label{(q,h)-Taylor expansion}
Let $D^-(c,R)\subset\mathbb{A}^{1,\mathrm{an}}_\Omega$ 
be a disc invariant under $\sigma_{q,h}$.
Let $f(T):=\sum_{n\geq 0}a_i(T-c)^i\in\O(D^-(c,R))$. 
Then
\begin{enumerate}
\item $f(T)$ can be uniquely written as 
$f(T)=\sum_{n\geq 0}
\widetilde{a}_n(T-c)^{[n]}_{q,h}\in\O(D^-(c,R))$. 
In particular, if $q\neq 1$, $d_{q,h}$ is defined around 
$a=-h/(q-1)$;
\item For all $\rho$ satisfying $|(q-1)c+h|<\rho<R$ one has $|f|(x_{c,\rho}):=
\sup_{n\geq 0}|a_n|\rho^n=
\sup_{n\geq 0}|\widetilde{a}_n|\rho^n$;
\item The radius of convergence of $f$ at $c$ is given by the formula:
\begin{equation}
Radius(f(T))\;\;:=\;\;\liminf_{n}|a_n|^{-1/n}\;\;=\;\;\liminf_{n}|\widetilde{a}_n|^{-1/n}\;.
\end{equation}
\item  If $q=1$ assume that $h\neq 0$, if $q\neq 1$ assume that 
$q$ is not a root of unity. 
Then one has the $(q,h)$-Taylor expansion formula 
$f(T)=
\sum_{n\geq 0}d_{q,h}^n(f)(c)\cdot 
\frac{(T-c)^{[n]}_{q,h}}{[n]_q^!}$.
\end{enumerate}
\end{proposition}
\begin{proof}
The proof follows closely \cite[Lemma 5.3]{Pu-q-Diff}, we omit it for 
expository reasons.
\end{proof}
\begin{corollary}[Non degeneracy of $\sigma_{q,h}$]
\label{Cor : Non degeneracy}
Let $X$ be an analytic domain of 
$\mathbb{P}^{1,\mathrm{an}}_K$, together with a weak 
triangulation $S$, such that   
$\sigma_{q,h}$ acts $S$-infinitesimally on $X$. 
Assume that $X$ is not an open 
disc centered at $\infty$ with empty weak triangulation. 
If $q=1$ assume that $h\neq 0$, and if $q\neq 1$ 
assume that $q$ is not a root of unity. Then the action of 
$\sigma_{q,h}$ is non degenerate on $X$.

Under the same assumptions, 
$\sigma_{q,h}$ is non degenerate if we replace $X$ by 
the Robba ring. 
\end{corollary}
\begin{proof}
By the assumption there exists a point $x\in X$ 
such that $\infty\notin D(x,S)$. 
Let $D\subseteq X_\Omega$ be an open disc such that 
$D^+(x,\sigma)\subset D\subseteq D(x,S)$. 
We can write each $f\in\O(D)$ as $f(T)=
\sum_{n\geq 0}d_{q,h}^n(f)(c)\cdot 
\frac{(T-c)^{[n]}_{q,h}}{[n]_q^!}$. Now $\sigma_{q,h}(f)=f$ 
means $d_{q,h}(f)=0$, which holds if and only if $f$ is constant. 
\end{proof}
\begin{remark}
The action of the group $\bs{\mu}_{p^n}$ of $p^{n}$-th roots of unity, 
and also of $\bs{\mu}_{p^{\infty}}=\cup_n\bs{\mu}_{p^n}$, are 
always degenerate. This is due to the existence of the 
function $\ell_x(T):=\log(T/t_x)$, 
for all $x\in\mathbb{A}^{1,\mathrm{an}}_K$, 
that verifies $\ell_x(qT)=\ell_x(T)$ for all 
$q\in\bs{\mu}_{p^{\infty}}$. 
So \eqref{eq : O(D)^Sigma=Omega} can not occur.
\end{remark}

\subsection{Confluence.}\label{confluence}
In this section we give a characterization of the essential image of the 
deformation functor $\Def_{\sigma_{q,h}}$, and we define a quasi 
inverse functor \emph{confluence}. For this we need to fix a global 
coordinate on $X$, so we are induced to make the following 
assumption:

\begin{hypothesis}
From now on we assume that $X$ is an analytic domain of 
$\mathbb{A}^{1,\mathrm{an}}_K$.
\end{hypothesis}

\subsubsection{$(q,h)$-Taylor solution.}
Let $X$ be an analytic domain of 
$\mathbb{A}^{1,\mathrm{an}}_K$, with a weak 
triangulation $S$, and an $S$-infinitesimal action of 
$\sigma_{q,h}$.  We now give a criterion for a 
$\sigma_{q,h}$-difference module 
$(\Fs,\sigma_{q,h})$ to be the deformation of a 
differential equation $\nabla$, under the assumptions of 
non degeneracy of Corollary \ref{Cor : Non degeneracy}. 
We will need to use the action of $d_{q,h}$ on $\Fs$ 
(cf. \eqref{eq : Delta(Y)=GY}), this 
introduces a pole at $a=-h/(q-1)$ 
(cf. \eqref{eq : a=-h/(q-1)}), 
so we assume that $a\notin X$ 
(i.e. that $\sigma_{q,h}$ has no fixed points in $X$).

In this situation 
we can always find a $\Gamma_S$-covering of $X$ 
formed by quasi-Stein analytic domains of $X$ 
on which $\Fs$ and 
$\Omega^1_X$ are free. 
By Corollary 
\ref{Cor : Non degeneracy} if $\sigma_{q,h}$ is non 
degenerate on $X$, it is so on each open of the covering. 
So, by fully-faithfulness, the differential equation will be 
unique on the intersections. So the local pieces 
glue to a global 
differential equation over $X$.\footnote{Namely the 
intersections are 
quasi-Stein, and the matrix of $\nabla$ is given by 
$G=d/dT_1(Y_\chi)\cdot Y_\chi^{-1}$ 
(cf. \eqref{eq : Y_chi}). Over an intersection 
the two matrices of the stratifications differs by 
multiplication by a matrix killed by $d/dT_1$, 
so they furnishes the same $G$.} 
\begin{hypothesis}
We 
assume that $\Fs$ and $\Omega^1_X$ are free, 
that $X$ is quasi-Stein\footnote{Conjecturally every  
connected analytic domain of 
$\mathbb{A}^{1,\mathrm{an}}_K$ is quasi-Stein.},
and that $a\notin X$. 
\end{hypothesis}

With this assumption the action of $\sigma_{q,h}$ 
corresponds (in a basis $\e$ of $\Fs$) to 
an equation $\sigma_{q,h}(Y)=A(q,h;T)\cdot Y$, 
with $A(q,h;T)\in GL_n(\O(X))$. 
Equivalently we have an equation of type 
\begin{equation}\label{eq : Delta(Y)=GY}
d_{q,h}(Y)\;=\;G_{[1]}(q,h;T)\cdot Y\;,
\end{equation}
where $G_{[1]}:=
\frac{A-\mathrm{Id}}{(q-1)T+h}\in 
M_n(\O(X))$.\footnote{Notice 
that $G_{[1]}$, and also $G_{[n]}$, has a denominator. 
It belong however to $M_n(\O(X))$ because $a\notin X$.} 
If \eqref{eq : Delta(Y)=GY} 
admits a complete basis of solutions 
$Y_D\in GL_n(\O(D))$, 
over some $\sigma_{q,h}$-invariant open disc $D$, 
then we can express it as $Y_D(T)=
\sum_{n\geq 0}d_{q,h}^n(Y_D)(c)\cdot 
\frac{(T-c)^{[n]}_{q,h}}{[n]_q^!}$. 
Now, iterating \eqref{eq : Delta(Y)=GY}, we find 
$d_{q,h}^n(Y_D)=G_{[n]}(q,h;T)\cdot Y_D$, where 
$G_{[n]}$ are inductively defined by the relations 
$G_{[0]}=\mathrm{Id}$, and 
$G_{[n+1]}=\sigma_{q,h}(G_{[n]})\cdot 
G_{[1]}+d_{q,h}(G_{[n]})$. 

Assume for a moment that $(\Fs,\sigma_{q,h})=
\mathrm{Def}_{\sigma_{q,h}}(\Fs,\nabla)$ is obtained 
by $\sigma_{q,h}$-deformation from a differential 
equation. Then the matrix of the stratification 
$\chi$ associated to $\nabla$ can be written as
\begin{equation}\label{eq : (q,h)-cocycle}
Y_\chi\;=\;\sum_{n\geq 0}G_{[n]}(q,h;T_2)
\frac{(T_1-T_2)_{q,h}^{[n]}}{[n]^!_q}\;
\in\;GL_n(\O(\mathcal{T}))\;,
\end{equation}
where $\mathcal{T}$ is an 
admissible open neighborhood of the diagonal containing 
$\Delta_{\sigma_{q,h}}(X)$.

We now come back to the general case of a possibly not stratified 
equation \eqref{eq : Delta(Y)=GY}. In this case we consider 
\eqref{eq : (q,h)-cocycle} as a (possibly divergent) series of functions 
over some unspecified admissible open neighborhood $\mathcal{T}$ 
of the diagonal. 
We now investigate whether \eqref{eq : (q,h)-cocycle} converges to the matrix 
of a stratification corresponding to a $\sigma_{q,h}$-compatible 
differential equation. 
We define for all $x\in X$
\begin{equation}\label{eq : (q,h)-Radius non intrinsic}
\R^{\Fs,\sigma_{q,h}}(x,\e)\;:=\;
\min\Bigl(\;\rho_{S,T}(x)\;,\;
\liminf_n\Bigl(|G_{[n]}(q,h;T)|(x)/
[n]^!_q\Bigr)^{-1/n}\Bigr)\;,
\end{equation}
where $\rho_{S,T}(x)$ in the function of Remark 
\ref{Remark : trivializing Radius over A^1}.

\begin{corollary}\label{Corollary : quasi-stein conf}
Assume that, for 
all $x\in X$, we have (cf. \eqref{eq : R_S(x,sigma)})
\begin{equation}\label{eq : condition for confluence}
\R^{\Fs,\sigma_{q,h}}(x,\e)\;>\;\R^{\sigma_{q,h}}(x)\;.\footnote{Recall that 
$\R^{\sigma_{q,h}}(x)=
\R_{S}(x,\sigma_{q,h})\cdot\rho_{S,T}(x)=
|(q-1)t_x+h|(x)$, as in the proof of Lemma 
\ref{Thm : Finiteness : the case of the line}.}
\end{equation}
Then \eqref{eq : (q,h)-cocycle} converges in 
$M_n(\O(\mathcal{T}))$, for some admissible neighborhood of the 
diagonal containing $\Delta_{\sigma_{q,h}}(X)$,
and it lies in $GL_n(\O(\mathcal{T}))$. 
Moreover it is the matrix of a stratification corresponding 
to a $\sigma_{q,h}$-compatible differential 
 equation $(\Fs,\nabla)$
whose $\sigma_{q,h}$-deformation is 
$(\Fs,\sigma_{q,h})$, and one has
\begin{equation}
\R^{(\Fs,\nabla)}(x)\;=\;\R^{\Fs,\sigma_{q,h}}(x,\e)\;.
\end{equation}
\end{corollary}
\begin{proof}
We begin by the following 

\begin{lemma}\label{Lemma : q-tay is a cocycle}
Assume that $X$ is an affinoid domain, and that 
there exists $R$ such that 
\begin{equation}\label{eq : Raiidiid}
\max_{x\in X}\R_{S}(x,\sigma_{q,h})\cdot
\rho_{S,T}(x)\;<\;R\;<\;
\min_{x\in X} 
\{\textrm{Radius of }(D(x,S))\}\;.\footnote{$X_{\Ka}$ 
is a disjoint union of affinoid domains of the type 
$Y=D^+(c_0,R_0)-\cup_{i=1}^sD^-(c_i,R_i)$, for 
which 
$\min_{x\in Y} \{\textrm{Radius of }(D(x,S))\}=
\min(R_0,R_1,\ldots,R_s)$.}
\end{equation}
Then the following are equivalent:
\begin{enumerate}
\item For all $x\in X$ we have 
$\R^{\Fs,\sigma_{q,h}}(x)>R$;
\item \eqref{eq : (q,h)-cocycle} converges in 
$M_n(\O(\mathcal{T}))$, with 
$\mathcal{T}:=\mathcal{T}(X,T,R)$, 
and it lies in $GL_n(\O(\mathcal{T}))$. 
Moreover it is the matrix of a stratification over $X$ 
corresponding to 
a $\sigma_{q,h}$-compatible differential equation 
$(\Fs,\nabla)$ whose $\sigma_{q,h}$-deformation 
is $(\Fs,\sigma_{q,h})$.
\end{enumerate}
\end{lemma}
\begin{proof}
If $q\neq 1$, by a translation we can assume that $h=0$ 
(cf. \eqref{eq : a=-h/(q-1)}). In this case the Proposition 
is proved in \cite[Lemma 5.16]{Pu-q-Diff}. If $q=1$ and 
$h\neq0$, the proof follows similarly.
\end{proof}

The proof of Corollary 
\ref{Corollary : quasi-stein conf} then goes as follows.
We find a covering of $X$ on which Lemma 
\ref{Lemma : q-tay is a cocycle} applies. More precisely, 
if $x\in\Gamma_S$, we consider a 
star-shaped affinoid neighborhood of $x$ in $X$ of the form 
$Y_x=\tau_{S}^{-1}(\Lambda_x)$, similarly as in Definition 
\ref{Def : Gamma_S-cov}. 
By construction $Y_x$ is stable by $\sigma_{q,h}$, and 
for all $y\in Y_x$ we have 
$\rho_{S,T}(y)=\rho_{S_{Y_x},T}(y)$,  
$\R^{\Fs,\sigma_{q,h}}(y)=
\R^{\Fs_{|Y_x},\sigma_{q,h}}(y)$, and also 
$\R_{S_{Y_x}}(y,\sigma_{q,h})=\R_S(y,\sigma_{q,h})$. 
Up to shrinking $\Lambda_x$,  by continuity we have 
\eqref{eq : Raiidiid}, and i). 
This proves the existence of a good differential equation 
over $Y_x$. 
Since we are in the affine line, this is actually a covering 
of $X$ as soon as $X$ is not an open disc with empty 
weak triangulation (i.e. $\Gamma_S=\emptyset$).  
And one sees that we can assume that the intersection of 
three affinoids of the covering is empty, so the local data 
glue over $X$.

If $\Gamma_S=\emptyset$, by 
Lemmas \ref{Lemma : discs stable by sigma_q,h} 
and \ref{Lemma: D(t,sigma)}, 
$\sigma_{q,h}$ is actually $S'$-infinitesimal with respect to a 
convenient weak triangulation given by a point $x\in X$. If $x$ is 
close enough to the open boundary of the disc $X$, 
replacing $S=\emptyset$ by $S'=\{x\}$ doesn't cause 
any trouble. And we can apply the above proof.
\end{proof}

\begin{corollary}
Under the assumptions of Corollary \ref{Corollary : quasi-stein conf}, 
the matrix of $\nabla$ is given by the formula
\begin{equation}
G(T)\;:=\;\lim_{n\to+\infty}
\frac{A(q^{p^n},[p^n]_q\cdot h;T)-\mathrm{Id}}
{(q^{p^n}-1)T+[p^n]_q\cdot h}\;.
\end{equation}
\end{corollary}
\begin{proof}
If $Y_\chi(T_1,T_2)$ is the matrix of the stratification, we 
have $G(T)=\frac{d}{dT_1}(Y_\chi)\cdot Y_\chi^{-1}$ and 
$A(q,h;T)=Y_\chi(qT+h,T)$. The claim follows from the fact that 
$\lim_{\sigma\to 1}\frac{\sigma-1}{\sigma(T)-T}=\frac{d}{dT}$, 
and recalling that $\sigma_{q,h}^{p^n}=
\sigma_{q^{p^n},[p^n]_q\cdot h}$ tends to $1$ as $n\to +\infty$.
\end{proof}
\begin{remark}\label{Remark : decreasing on discs}
If $D\subset (X-\Gamma_S)$ is a disc, and if $I$ is a 
segment in $D$ oriented as outside $D$, then 
$\R^{\Fs,\sigma_{q,h}}(x,\e)$ is logarithmically 
not increasing along $I$ since each function 
$x\mapsto|G_{[n]}|(x)$ is not $\log$-decreasing, and 
$\rho_{S,T}$ is locally constant outside 
$\Gamma_S$. 
If $X$ is connected, and if $\Gamma_S\neq\emptyset$, 
this shows that \eqref{eq : condition for confluence} 
holds for all $x\in X$ if and only if it holds for all 
$x\in\Gamma_S$. If $\Gamma_S=\emptyset$, then $X$ 
is a virtual open disc and it is enough to 
test \eqref{eq : condition for confluence} at the 
open boundary of the disc.
\end{remark}
\subsubsection{}
Let $\mathcal{G}$ be the group structure on 
$\mathbb{G}_m^{\mathrm{an}}\times
\mathbb{A}_m^{1,\mathrm{an}}$ defined in 
section \ref{Section : infinitesimality of sigma_qh}.
For all $0<\tau< 1$ and $\nu>0$ the product of discs 
$\mathcal{G}_{\tau,\nu}:=
D^-(1,\tau)\times D^-(0,\nu)$ is a 
$K$-analytic subgroup.
The results of section \ref{Analyticity of the action} 
apply to $\mathcal{G}_{\tau,\nu}$. If its action is 
$S$-infinitesimal,
by Corollary \ref{Cor : Non degeneracy} it is also non degenerate.

The following proposition shows how to recover the 
differential equation from its 
$\mathcal{G}_{\tau,\nu}$-deformation.
\begin{proposition}\label{Prop : 9.13}
Assume that $\mathcal{G}_{\tau,\nu}$ acts
 $S$-infinitesimally on $X$, 
let $Y'=G(T)Y$ be a $\mathcal{G}_{\tau,\nu}$-compatible 
differential equation,
and let $\{\sigma_{q,h}(Y)=A(q,h;T)Y\}_{(q,h)\in 
\mathcal{G}_{\tau,\nu}}$ be 
the corresponding  
$\mathcal{G}_{\tau,\nu}$-deformation. Then
for all $(a,b)\in K^{2}-\{(0,0)\}$ one has
\begin{equation}\label{eq : G(T)= d/dq A+d/dh A}
G(T)\;\;=\;\; 
(aT+b)^{-1}\cdot\Bigl[\;a\cdot\frac{\partial}{\partial 
q}\Bigr.A(q,h;T)\Bigr|_{q=1,h=0}\;+\;b\cdot
\frac{\partial}{\partial h}\Bigr.A(q,h;T)\Bigr|_{q=1,h=0}\;\Bigr] \;.
\end{equation}
In particular $G(T)= T^{-1}[\frac{d}{dq}A(q,h;T)]_{|q=1,h=0} = 
[\frac{d}{dh}A(q,h;T)]_{|q=1,h=0}$.
\end{proposition}
\begin{proof}
As in the proof of Proposition 
\ref{Prop : non degeneracy --> fully faith}, 
we can assume that $X$ is equal to the open disc $D$ 
centered at $t_x$ as in Definition 
\ref{Def : non degeneracy}. 
If $Y_\chi(T_1,T_2)$ is the matrix of the stratification, 
then $A(q,h;T)=Y_\chi(qT+h,T)$. If 
$\gamma_{(a,b)}:D^-(0,\varepsilon)\to 
\mathcal{G}_{\tau,\nu}$ is 
the path $r\mapsto (1+ar,br)$, then the limit 
$\frac{d}{dr}(\sigma_{\gamma_{(a,b)}(r)}):=
\lim_{r\to 0}\frac{\sigma_{1+ar,br}-\mathrm{Id}}{r}$ 
converges to  $(aT+b)\frac{d}{dT}$.
The claim then follows  
from 
$G=\frac{d}{dT_1}(Y_\chi)\cdot Y_\chi^{-1}$.
\end{proof}

\begin{remark}
If an individual action of $\sigma_{q,h}$ 
satisfying  
Corollaries \ref{Cor : Non degeneracy} and 
\ref{Corollary : quasi-stein conf} is given, 
this produces a $\sigma_{q,h}$-compatible differential 
equation $(\Fs,\nabla)$. 
If now $X$ is an affinoid domain in 
$\mathbb{A}^{1,\mathrm{an}}_K$ as in 
\cite{Pu-q-Diff} (though this work also over a 
more general class of analytic domains), 
the differential equation so obtained is 
$\mathcal{G}_{\tau,\nu}$-compatible for some pair $(\tau,\nu)$, 
and so, by $\mathcal{G}_{\tau,\nu}$-deformation of $\nabla$, the original action of $\sigma_{q,h}$ 
extends to an action of $\mathcal{G}_{\tau,\nu}$. 
\end{remark}

\subsection{An example on the Tate curve}
A Tate curve $X_a$ is obtained as a quotient of 
$\mathbb{G}_{m,K}^{\mathrm{an}}$ by the action of 
$a^{\mathbb{Z}}$, where $a\in K$ has norm $|a|<1$.
The analytic skeleton of $X_a$ is a loop, and it 
is the skeleton of a weak 
triangulation $S$.
Now, $\mathbb{G}_{m,K}^{\mathrm{an}}\subseteq
\mathbb{A}^{1,\mathrm{an}}_K$ is stable under 
the action of $\mathcal{G}_{1,0}=D^-(1,1)$. 
That action commutes with the multiplication by $a$, and 
it defines an $S$-infinitesimal non degenerate 
action on $X_a$.

It has been shown in \cite{NP-III} that all differential 
equation $\Fs$ over $X_a$ has constant radius 
$\R_{S,1}(-,\Fs)$.
On the other hand it is easy to see that for all 
$q\in D^-(1,1)$, one has 
$\R_{S}(-,\sigma_{q,0})=|q-1|$. The action of 
$\sigma_{q,0}$ is visibly non degenerate as soon as $q$ 
is not a root of unity. 
This permits to describe all differential equations of 
$X_a$ as semi-linear analytic $G_{\tau,0}$-modules for 
some $\tau>0$.

\section{Morita's $p$-adic Gamma function and 
Kubota-Leopolodt's $L$-functions}\label{Gamma}

In this section we apply the previous theory to a 
particular difference equation satisfied by the Morita's 
$p$-adic $\Gamma$-function $\Gamma_p$. 
We firstly prove some useful results in section 
\ref{Section : smallradius}.

\subsection{Small radius for rank one $(q,h)$-difference 
equations}\label{Section : smallradius}
In section \ref{eq : rough estimation} we have defined 
the number $\omega=\lim_n|n!|^{1/n}$. 
For all $q\in K^{\times}$, $|q-1|<1$, we now set 
$\omega_q := \lim_{n\to\infty}|[n]_q^!|^{1/n}$. 
One verifies (cf. \cite[3.5]{DV-Dwork}) that, 
if $\kappa\geq 1$ is the smallest integer such that 
$|q^\kappa-1|<\omega$, then $\omega_q=\omega$ if 
$\kappa=1$ and
$\omega_q=\bigl([\kappa]_q\cdot\omega\bigr)^{1/\kappa}$ 
if $\kappa\geq 2$. In particular $\omega_1=\omega$.

Let $X$ be an affinoid domain of 
$\mathbb{A}^{1,\mathrm{an}}_K$, with a weak 
triangulation $S$, and an $S$-infinitesimal non 
degenerate action of $\sigma_{q,h}$ 
(cf. Corollary \ref{Cor : Non degeneracy}). 
\if{
If $K=\Ka$, 
$X$ is a disjoint union of connected affinoid 
domains of the form 
$X=D^+(c_0,R_0)-\cup_{i=1}^nD^-(c_i,R_i)$,
where $c_0,\ldots,c_n\in K$ satisfy $|c_i-c_0|\leq R_0$ 
for all $i$, and $0<R_1,\ldots,R_n\leq R_0$. 

If $K$ is general, an affinoid domain $X$ of 
$\mathbb{A}^{1,\mathrm{an}}_K$ is the quotient of 
$X_{\Ka}$ by $\mathrm{Gal}(\Ka/K)$ 
(cf. \cite[2.2.2]{Ber}). Without 
loss of generality we will always assume that $X$ is 
connected. 
So the holes of $X_{\Ka}$ form an orbit 
under $\mathrm{Gal}(\Ka/K)$,\footnote{Including the 
holes of 
$X_{\Ka}$ at $+\infty$, i.e. the complements in 
$\mathbb{P}^{1,\mathrm{an}}_{\Ka}$ of the disc 
$D^+(c_0,R_0)$.} 
which acts isometrically.
Hence we can speak about the holes of $X$, and of their 
radii $R_1,\ldots,R_n$. And also about the larger virtual 
disc containing $X$ whose radius is $R_0$. 
Such notations are fixed from now on.
}\fi
Let $x\in X$ be a point such that 
$D(x)$ is fixed by 
$(\sigma_{q,h})_\Omega$.\footnote{By Lemma 
\ref{Lemma : discs stable by sigma_q,h} this 
means $|(q-1)t_x+h|<r(x)$.} Since $\sigma_{q,h}$ is an 
automorphism of $D(x)$, its composite with a bounded 
function remains bounded. 
So $d_{q,h}$ acts on the ring $\mathcal{B}(D(x))$ of 
bounded functions over $D$. We have an isometric 
inclusion $\H(x)\to\mathcal{B}(D(x))$ given by the 
Taylor expansion at $t_x$: 
$f\mapsto\sum_{n\geq 0}f^{(n)}(t_x)(T-t_x)^n/n!$.

Let $d$ be $d_{q,h}$ or  $d/dT$. 
We set 
$\|d\|_{\mathcal{B}(D(x))}:=
\max_{f\in \mathcal{B}(D(x)), f\neq 0}
\|d(f)\|_{D(x)}/\|f\|_{D(x)}$, 
where $\|.\|_{D(x)}$ is the sup-norm on $D(x)$. 
It follows from Proposition \ref{(q,h)-Taylor expansion}  
that $\|d\|_{\mathcal{B}(x)}=r(x)^{-1}$.
\begin{lemma}[explicit small radius]
\label{Lemma : small radius}
Let $\Fs$ be a $\sigma_{q,h}$-module.  
With the notation of 
\eqref{eq : (q,h)-Radius non intrinsic} we have 
$\liminf_{n}(|G_{[n]}|(x)/|[n]_q^!|
)^{-1/n}\geq
\frac{\omega_q}{\max(|G_{[1]}|(x),
\|d_{q,h}\|_{\mathcal{B}(D(x))})}$.
Moreover if $\mathrm{rank}(\Fs)=1$, then 
$| G_{[1]}|(x)> 
\|d_{q,h}\|_{\mathcal{B}(D(x))}$ if and only if 
$\liminf_n(|G_{[n]}|(x)/
|[n]^!_q|)^{-1/n}<\omega_q\cdot
\|d_{q,h}\|_{\mathcal{B}(D(x))}^{-1}$. In this case 
\begin{equation}
\liminf_{n}\Bigl(|G_{[n]}|(x)/|[n]_q^!|
\Bigr)^{-1/n}\;=\;\frac{\omega_q}{|G_{[1]}|(x)}\;.
\end{equation}

The same statements hold for rank one differential 
equations replacing 
$d_{q,h},\omega_q,[n]^!_q,G_{[1]}$ 
with $d/dT,\omega,n!,G_1$, where $G_1$ is the 
matrix of \eqref{eq : Y_chi}.
\end{lemma}
\begin{proof} By
$G_{[n+1]}=
d_{q,h}(G_{[n]})+\sigma_q(G_{[n]})G_{[1]}$ 
we inductively have
$|G_{[n]}|(x)\leq \max(|G_{[1]}|(x),
\|d_{q,h}\|_{\mathcal{B}(D(x))})^n$, 
and equality holds if 
$|G_{[1]}|(x)>\|d_{q,h}\|_{\mathcal{B}(D(x))}$ and if 
$\Fs$ has rank one. Now, since the
sequence $|[n]_q^{!}|^{1/n}$ is convergent to 
$\omega_q$, one has
$\liminf_{n}(|G_{[n]}|(x)/|[n]_q^!|)^{-1/n} = 
\omega_q\cdot
\liminf_{n}|G_{[n]}|(x)^{-1/n}$. 
\end{proof}

\if{
\begin{remark}
The analogous lemma is well known for differential equations (cf. \cite[Cor.6.7]{Astx}).
\end{remark}
}\fi

\subsection{Morita's $p$-adic Gamma function as solution of a difference equation}\label{diff eq of Gamma sections}

In this section $K=\mathbb{Q}_p$. 
Assume that $p\neq 2$ is a prime number. 
The Morita's
$p$-adic Gamma function 
$\Gamma_p:\mathbb{Z}_p\to\mathbb{Z}_p$, 
is the unique continuous function on
$\mathbb{Z}_p$ verifying $\Gamma_p(0)=1$, and 
the functional equation $\Gamma_p(x+1)=\left\{\sm{-
x\Gamma_p(x)&\textrm{ if }& |x|=1\phantom{\;.}\\ 
&&\\
-\Gamma_p(x)&\textrm{ if }& |x|<1\;,}\right.$.
Its values on natural numbers $n\geq 1$ are 
given by $\Gamma_p(n)=(-1)^{n}\cdot
\prod_{i=1, (i,p)=1}^{n-1} i$. 
It is known since \cite{Morita} 
that $\Gamma_p(T)$ is
locally analytic with local radius greater than $|p|$. 
Subsequently Dwork  \cite{Dwork-Radius-Gamma}, 
applying non cohomological methods introduced by 
D.Barsky \cite{Barsky-Gamma}, was able to compute the 
exact radius of convergence of $\Gamma_p(T)$ 
in a neighborhood of the points $0,\ldots,p-1$, which is 
$\omega|p|^{1/p}$. Denote by 
$\Gamma_p^{i}(T):=1+
\sum_{n\geq 1}\gamma_{n}^{i}\cdot
(T-i)^n$ the Taylor expansion of $\Gamma_p$ at 
$T=i\in\{0,\ldots,p-1\}$.
Clearly $\Gamma_p^{i}(T+i)=
(-1)^{i}(T+1)(T+2)\cdots(T+i-1)
\Gamma_p^{0}(T)$. 
From the functional equation, for all $n\geq 1$, we have
\begin{equation}\label{iterated equation Gamma}
\sigma_{1,p^n}(\Gamma^{0}_p(T))\;\;=\;\;A(1,p^n;T)\cdot\Gamma_p^{0}(T)\;,\qquad\qquad
A(1,np;T)\;=\;-\prod_{i=1,(i,p)=1}^{np-1}(T+i)\;.
\end{equation}
\if{
These equations are defined everywhere since 
$A(1,p^n;T)$ is a polynomial. But each 
one of them verifies the assumptions of Theorem \ref{Theorem over an annulus ...} over a disc $\mathrm{D}^-(0,r_n)$, centered at $0$ (cf. section 
\ref{Formal cocycle}). 

By Theorem \ref{Theorem over an annulus ...} and section \ref{Formal cocycle} the function $\Gamma^0_p(T)$ is then solution of a differential equation
\begin{equation}
\Gamma_p^0(T)'\;\;=\;\; g_0(T)\cdot\Gamma_p^0(T)\;.
\end{equation}
Let $Y_0(x,y)$ denote the analytic cocycle attached to this differential equation. In this section we are interested in 
\begin{enumerate}
\item finding the convergence disc of $g_0(T)$,
\item describing the functions $\rho\mapsto |g_0(T)|_{\rho}$,
\item describing the functions $\rho\mapsto Rad_x(Y_0(x,y),|.|_{\rho})$.
\end{enumerate}

The disc $\mathrm{D}^-(0,r_n)$ can be defined as the biggest disc over which $Y_0(x,y)$ is $\sigma_{1,p^n}$-compatible. Since the $\sigma_{1,p^n}$-compatibility 
condition is $|\delta_{\sigma_{1,p^n}}|_\rho=|p|^n<Rad_x(Y_0(x,y),\rho)$, we have $\mathrm{D}^-(0,r_n)\subset\mathrm{D}^-(0,r_{n+1})$ for all $n\geq 1$. 
Moreover, if $Y_{p^n}(x,y)$ denotes the restriction of $Y_0(x,y)$ to $\mathrm{D}^-(0,r_n)\times\mathrm{D}^-(0,r_n)$, then $Y_{p^n}(x,y)$ is by Theorem 
\ref{Theorem over an annulus ...}, solution of the equation $Y_{p^n}(\sigma_{1,p^n}(x),y)=A(1,p^n;x)\cdot Y_{p^n}(x,y)$. We will prove that 
\begin{equation}
\cup_{n\geq 1}\mathrm{D}^-(0,r_n) \;\;=\;\; \mathrm{D}^-(0,1)\;.
\end{equation} 
As a matter of facts $Y_0(x,y)$ will be constructed by gluing the analytic cocycles $Y_{p^n}(x,y)$ defined step by step, over 
$\mathrm{D}^-(0,r_n)\times\mathrm{D}^-(0,r_n)$, by the $(1,p^h)$-th equation $\sigma_{1,p^n}(Y)=A(1,p^n;T)\cdot Y$ by the expression 
\eqref{cocycle of q-h equation}. 
The function $\Gamma_p^0(T)$ will hence be  equal to $Y_0(T,0)$.
}\fi
\begin{theorem}\label{Gammmammmma}
The function $\Gamma_p^0(T)$ is the Taylor solution at 
$T=0$ of a rank one differential equation 
$Y'=g_0(T)\cdot Y$ such that 
$g_0(T)\in\O(D^-(0,1))\subset\mathbb{Q}_p[[T]]$. 
If $\Fs$ is the associated differential module, and if 
$D^-(0,1)$ has the empty weak triangulation, then:
\begin{equation}
\R^\Fs(x_{0,\rho})=\left\{\sm{
\omega |p|^{1/p}&\textrm{ if }&0&\leq&\rho&\leq& 
r_0 &\\
\frac{\omega|p|^{n}}{\rho^{p^{n-1}(p-1)}}&
\textrm{ if }&r_{n-1}&\leq&\rho&\leq&
r_n\;,&\quad n\geq 1.
}\right.
\end{equation}
where $r_{0}=|p|^{1/p}$, and  
$r_n=\omega^{\frac{1}{p^{n-1}(p-1)}}$, 
for all $n\geq 1$.
Moreover, the Small Radius Lemma 
\ref{Lemma : small radius} gives
\begin{equation}\label{eq : norm-rho of g_0}
|g_0(T)|(x_{0,\rho})=\left\{\sm{
|g_0(T)|(x_{0,\rho})\leq\rho^{-1}&
\textrm{ if }&0&\leq&\rho&\leq& r_0 &\\
\rho^{p^{n-1}(p-1)}/|p|^{n}    &\textrm{ if }&\quad
r_{n-1}&\leq&\rho&\leq&r_n\;,& \quad n\geq 1.
}\right.
\end{equation}
\end{theorem}
\begin{proof} 
Let $(\Fs,\sigma_{1,p^n})$ be the difference 
module associated to \eqref{iterated equation Gamma} 
in the basis $\e\in\Fs$.
Every $A(1,p^n;T)$ converges everywhere, because
it is a polynomial. We may think that 
$(\Fs,\sigma_{1,p^n})$ is defined over a conveniently 
large disc with empty weak triangulation, 
so $\rho_{S,T}$ does not play any role in 
\eqref{eq : (q,h)-Radius non intrinsic}, which 
is always computed by the infimum limit.
So, for all $\rho\geq 0$ we set 
$R(n,\rho):=\liminf_s(|G_{[s]}(1,p^n;T)|_\rho/|[s]_q^!|
)^{-1/s}$.
By Remark \ref{Remark : decreasing on discs}, the locus 
of points where \eqref{eq : condition for confluence} 
holds is an open disc $D^-(0,r_n)$. And $r_n$ is the 
supremum value of $\rho$ satisfying 
$R(n,\rho)>\R^{\sigma_{1,p^n}}(x_{0,\rho})$.
Over that disc Corollary 
\ref{Corollary : quasi-stein conf} applies, and 
we have a differential equation $(\Fs,\nabla_n)$, whose 
$\sigma_{1,p^n}$-deformation is 
$(\Fs,\sigma_{1,p^n})$, 
satisfying $\R^{(\Fs,\nabla_n)}(x_{0,\rho})=
R(n,\rho)$ for all $\rho<r_n$.
By definition of the deformation, the 
$\sigma_{1,p^{n+1}}$-deformation of $(\Fs,\nabla_n)$ 
is $(\sigma_{1,p^{n}})^p:\Fs\simto\Fs$. 
Hence $R(n+1,\rho)=R(n,\rho)$ for all $\rho<r_n$, and 
by concavity of $R(n+1,\rho)$ we have $r_n<r_{n+1}$.
Now, for $n+1$, the range of application of Corollary 
\ref{Corollary : quasi-stein conf} is the disc 
$D^-(0,r_{n+1})$, and it is clear that 
$(\Fs,\nabla_{n+1})_{|D^-(0,r_n)}=(\Fs,\nabla_n)$. 
Since $\Gamma_p^0(T)$ is a solution of 
$\sigma_{1,p^n}$ and of $\sigma_{1,p^{n+1}}$, 
it is also a solution of 
$\nabla_n$ and of $\nabla_{n+1}$, 
hence the matrix of the 
two connections in the basis $\e$ of $\Fs$ coincide.
This proves that the matrix of $\nabla_n$ in the basis 
$\e$ actually lies over $D^-(0,r_{n+1})$. 
We now prove that $\lim_nr_n=1$, and 
inductively compute the function $R(n,\rho)$.
The proof consists in computing inductively 
the small values of the 
radii $R(n,\rho)$ by Lemma \ref{Lemma : small radius}, 
and also $r_n$, and $|g_0|(x_{0,\rho})$ by the same 
Lemma.

\begin{lemma}\label{verrii}
Let $G_{[1]}(1,p^n;T):=\frac{A(1,p^n;T)-1}{p^n}$.
For all $n\geq 1$ one has
\begin{equation}
\label{valuation of H_1((1,p),T)=rho^p^n-p^n-1}
|G_{[1]}(1,p^n;T)|(x_{0,\rho})\;=\;
\frac{\rho^{\mathrm{deg}(G_{[1]}(1,p^n;T))}}{|p|^n}=\frac{\rho^{p^{n-1}(p-1)}}{|p|^n}\;,\qquad\textrm{for all }\rho\geq
1\;,
\end{equation}
and for $n=1$ the equality holds for all 
$\rho\geq \omega$.
\end{lemma}
\begin{proof}
If $\sum a_iT^i\in\mathbb{Z}[T]$ is a polynomial of 
degree $n$, with $|a_n|=1$, then 
$x_{0,\rho}(f)=\rho^n$ for all $\rho\geq 1$. 
Then \eqref{valuation of H_1((1,p),T)=rho^p^n-p^n-1} 
follows from the fact that the degree of 
$G_{[1]}$ is $p^{n-1}(p-1)$. 
For $n=1$, the reduction of $A(1,p;T)$ in 
$\mathbb{F}_p[T]$ is the cyclotomic polynomial $1-T^{p-1}$. 
Then $A(1,p;T)-1=-T^{p-1}+a_{p-2}T^{p-2}+\cdots+a_{0}$, with
$|a_i|\leq |p|$, for all $i=0,\ldots,p-2$.  Hence $|A(1,p;T)-1|_{\rho}=\max(\rho^{p-1},|a_{p-2}|\rho^{p-2},\ldots,|a_0|)$,
and $\rho^{p-1}\geq |p|\rho^i$ for all $i=0,\ldots,p-2$ if and
only if $\rho\geq \omega=|p|^{1/(p-1)}$.
\end{proof}

\begin{lemma}
We have $R(1,\rho)=\omega |p|^{1/p}$ for 
$\rho\leq|p|^{1/p}$, and 
$R(1,\rho)=\omega|p|/\rho^{p-1}$, for all 
$\rho\geq |p|^{1/p}$.
\end{lemma}
\begin{proof}
For all $n\geq 1$, the function 
$R(p^n,\rho)=\R^{(\Fs,\nabla)}(x_{0,\rho})$ is 
constant over $D^-(0,\omega|p|^{1/p})$, 
with value $\omega|p|^{1/p}$,
because this is the disc of convergence of 
$\Gamma_p^0(T)$. By Lemmas \ref{verrii} and 
\ref{Lemma : small radius} we have 
$R(1,\rho)=\omega|p|^{1/p}/\rho^{p-1}$, for all 
$\rho> |p|^{1/p}$. Now, since $\ln(\rho)\mapsto 
\ln(R(1,\rho))$ is concave and continuous, it must be 
constant for all $\rho\leq |p|^{1/p}$ since 
$R(1,|p|^{1/p})=\lim_{\rho\to(|p|^{1/p})^+}R(1,\rho)=\omega|p|^{1/p}=R(1,0)$. 
\end{proof}

As explained $r_1=\sup(\rho\textrm{ such that }
R(1,\rho)>
|p|)=\omega^{\frac{1}{p-1}}$.

We now inductively assume that, for $n\geq 1$, 
$R(n,\rho)=\omega|p|^{n}/\rho^{p^{n-1}(p-1)}$, 
for all $r_{n-1}\leq\rho\leq r_n$. 

Now we know that $R(n+1,\rho)=R(n,\rho)$ for all 
$\rho\leq r_n$, and by Lemmas \ref{verrii} and 
\ref{Lemma : small radius} we have 
$R(n+1,\rho)=\omega|p|^{n+1}/\rho^{p^n(p-1)}$, for 
all $\rho\geq 1$. The values for $\rho\in [r_n,1]$ are 
deduced by continuity and concavity. Indeed the function 
$\rho\mapsto \omega|p|^{n+1}/\rho^{p^n(p-1)}$ 
is logarithmically a line, and its value at $\rho=r_n$ is
$|p|^{n}=R(n+1,r_n)$. So we have 
$\R(n+1,\rho)=\omega|p|^{n+1}/\rho^{p^n(p-1)}$, 
for all $\rho\geq r_n$. 
Again Corollary \ref{Corollary : quasi-stein conf} 
 implies $r_{n+1}=\sup(\rho\textrm{ such that }
R(n+1,\rho)>
|p|^{n+1})=\omega^{\frac{1}{p^n(p-1)}}$. 
This concludes the computation of the radius.
Now \eqref{eq : norm-rho of g_0} is an 
immediate consequence of Lemma 
\ref{Lemma : small radius}.
\end{proof}

Define the Newton polygon of 
$g_0(T):=\sum_{n\geq 0}
a_nT^n\in\mathbb{Q}_p[[T]]$ as the convex hull 
of the points $\{(n,v_p(a_n))\}_{n\geq 0}\cup
\{(0,+\infty)\}$, where $v_p$ is the $p$-adic valuation 
normalized by $v_p(p)=1$.

\begin{corollary}\label{Newton Polygon}
The wedges $(i,x_i)$ of the Newton polygon of 
$g_0(T)$ such that $i\geq p-1$ are the points
$\{(p^{n-1}(p-1),-n)\}_{n\geq 1}$. In particular $v_p(a_{p^{n-1}(p-1)})=-n$, for all $n\geq 1$.\hfill\CVD
\end{corollary}
\begin{remark}
For all $k\geq 0$ one has
\begin{equation}
v_p(a_k)\;\;\geq\;\;\left\{
\begin{array}{rcrcccl}
0&\textrm{ if }&0&\leq&k&\leq&p-2\;,\\
-n&\textrm{ if }&p^{n-1}(p-1)&\leq&k&<&p^{n}(p-1)\;,\quad n\geq 1\;.
\end{array}
\right.
\end{equation}
as illustrated in the following picture:
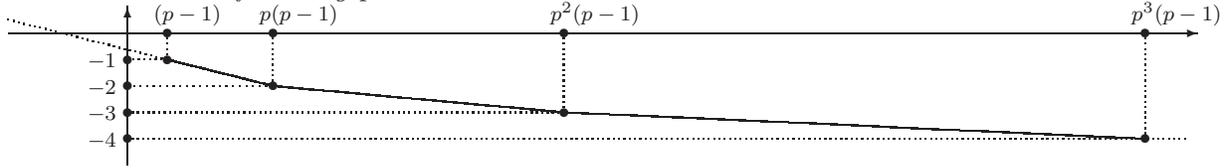
\begin{center}
\begin{picture}(400,38)
\put(0,-6){
\begin{picture}(400,53)
\put(-30,40){\vector(1,0){450}}
\put(15,-10){\vector(0,1){60}}

\put(28,38){\begin{scriptsize}$\bullet$\end{scriptsize}}\put(25,45){\begin{scriptsize}$(p-1)$\end{scriptsize}}
\put(68,38){\begin{scriptsize}$\bullet$\end{scriptsize}}\put(65,45){\begin{scriptsize}$p(p-1)$\end{scriptsize}}
\put(178,38){\begin{scriptsize}$\bullet$\end{scriptsize}}\put(175,45){\begin{scriptsize}$p^2(p-1)$\end{scriptsize}}
\put(398,38){\begin{scriptsize}$\bullet$\end{scriptsize}}\put(395,45){\begin{scriptsize}$p^3(p-1)$\end{scriptsize}}

\put(28,28){\begin{scriptsize}$\bullet$\end{scriptsize}}\qbezier[5](30,30)(30,35)(30,40)
\put(68,18){\begin{scriptsize}$\bullet$\end{scriptsize}}\qbezier[10](70,20)(70,30)(70,40)
\put(178,8){\begin{scriptsize}$\bullet$\end{scriptsize}}\qbezier[15](180,10)(180,25)(180,40)
\put(398,-2){\begin{scriptsize}$\bullet$\end{scriptsize}}\qbezier[20](400,0)(400,20)(400,40)

\qbezier[30](-30,45)(0,37.5)(30,30)

\qbezier(30,30)(50,25)(70,20)
\qbezier(70,20)(125,15)(180,10)
\qbezier(180,10)(290,5)(400,0)
\qbezier[7](15,30)(22.5,30)(30,30)\put(13,28){\begin{scriptsize}$\bullet$\end{scriptsize}}
\qbezier[22](15,20)(42.5,20)(70,20)\put(13,18){\begin{scriptsize}$\bullet$\end{scriptsize}}
\qbezier[82](15,10)(97.5,10)(180,10)\put(13,8){\begin{scriptsize}$\bullet$\end{scriptsize}}
\qbezier[200](15,0)(215,0)(415,0)\put(13,-2){\begin{scriptsize}$\bullet$\end{scriptsize}}

\put(0,27){\begin{scriptsize}$-1$\end{scriptsize}}
\put(0,17){\begin{scriptsize}$-2$\end{scriptsize}}
\put(0,7){\begin{scriptsize}$-3$\end{scriptsize}}
\put(0,-3){\begin{scriptsize}$-4$\end{scriptsize}}
\end{picture}
}
\end{picture}
\end{center}

\end{remark}

\subsection{Applications to Kubota-Leopoldt's $p$-adic 
$L$-functions}
\label{Section : Applications -KL}
We preserve the assumption $p\neq 2$. It has been 
known since Y.Morita \cite{Morita} and J.Diamond 
\cite[Theorem 10]{Diamond} (see also \cite[p.376]{Robert}) that 
$\log(\Gamma_p)$ admits the 
following Taylor expansion for $|T|\leq |p|$:
\begin{equation}
\log(\Gamma_p^0(T))=\lambda_0 T - \sum_{m\geq 1}
L_p(1+2m,\overline{\omega}_p^{2m})\frac{T^{1+2m}}{1+2m}
\end{equation}
where $\overline{\omega}_p$ 
is the inverse of the Teichm\"uller Dirichlet 
character corresponding to the prime $p$ and where 
$L_p(1+2m,\overline{\omega}_p^{2m})$ is the value at 
$s=1+2m$ of 
the $p$-adic Kubota-Leopoldt's $L$-function 
corresponding to the character $\overline{\omega}_p^{2m}$. 
The constant $\lambda_0$ is the constant coefficient 
appearing in the Taylor expansion of the Zeta function 
$\zeta_p(s)$ at $s=1$: 
$\zeta_p(s)=\sum_{n\geq -1}\lambda_n(s-1)^n$.
We note that 
\begin{equation}\label{...L...L...L}
g_0(T)=\frac{d}{dT}(\log(\Gamma_p^0(T)))=\lambda_0  - \sum_{m\geq 1}
L_p(1+2m,\overline{\omega}_p^{2m})T^{2m}\;.
\end{equation}
The Newton polygon of
$g_{0}(T)$ have been computed in Corollary 
\ref{Newton Polygon}. It gives the following estimate on 
the values of the $L$-functions appearing in 
\eqref{...L...L...L}:
\begin{corollary}\label{corollary abs value of zeta at 1+e}
For all $n\geq 1$ one has
\begin{equation}
v_p(\;L_p(1+p^{n-1}(p-1),\overline{\omega}_p^{p^{n-1}(p-1)})\;)\;\;=\;\;
v_p(\;\zeta_p(1+p^{n-1}(p-1))\;)\;\;=\;\;
-n\;.
\end{equation}
Moreover for all $m\geq 0$
\begin{equation*}
v_p(\;L_p(1+2m,\overline{\omega}_p^{2m})\;)\;\;\geq\;\; 
\left\{ 
\begin{array}{rcrcccl}
0&\textrm{ if }& 0&\leq &2m&\leq&(p-1)\\
-n&\textrm{ if }&p^{n-1}(p-1)&\leq &2m&<&p^{n}(p-1)\;,\quad n\geq 1\;.
\end{array}
\right.
\end{equation*}
\end{corollary}
Indeed, this constitutes a proof of the existence of a pole 
of $\zeta_p$ at $s=1$.

\subsection{An application to sums of powers.} 
\label{Section : Applications -SP}
We now apply the above computations to some sums of 
powers. The following computations have been obtained 
in collaboration with Daniel Barsky.

For all integers $\ell,k\geq 1$,  set
\begin{equation}
S_\ell(k)\;\;:=\;\;
\sum_{i=1, (i,p)=1}^{k-1}\frac{1}{i^\ell}\;.
\end{equation}
This and similar sums have been studied by several 
authors modulo powers of $p$ 
\cite[pp. 95-103]{Dickson}. 
A result of L.Washington \cite{Washington} expresses 
it as sum of values at certain positives integers of some 
Kubota-Leopoldt's $p$-adic $L$-functions. Similar 
expressions have been found by D.Barsky 
\cite{Barsky-cong}. 

The following corollary gives another 
proof of \cite[Theorem 1,(a)]{Washington} 
(cf. Remark \ref{Rk : link with Was}).  

\begin{corollary}\label{Cor : congruences}
For all integers $n,\ell\geq 1$ we have
\begin{equation}\label{eq : S_l}
\frac{(-1)^{\ell-1}}{\ell}\cdot S_\ell(np)
\;\;=\;\;
-\sum_{m\geq \ell/2}\binom{1+2m}{\ell}(np)^{(1+2m-\ell)}\cdot
\frac{L_p(1+2m,\overline{\omega}_p^{2m})}{1+2m}\;.
\end{equation}
Moreover for $\ell=1$ we have 
$S_1(np)=g_0(np)-g_0(0)$. In particular, for 
$n=p^{k-1}$, we recover the relation 
$\lim_{k\to\infty}p^{-k}
\sum_{i=0,(i,p)=1}^{p^k-1}i^{-1}=0$ 
because $g_0'(0)=0$.
\if{
\begin{enumerate}
\item One has 
$\log(\Gamma_p(np))=np\lambda_0-\sum_{m\geq 1}(np)^{1+2m}\cdot \frac{
L_p(1+2m,\overline{\omega}_p^{2m})}{1+2m}$;
\item For $\ell=1$ one has
\begin{equation}\label{eq : S_1}
S_1(np) \;\;:=\;\;\sum_{i=1, (i,p)=1}^{np-1}\frac{1}{i}\;\;=\;\; 
-\sum_{m\geq 1}(np)^{2m}\cdot L_p(1+2m,\overline{\omega}_p^{2m}) \;\;=\;\;
g_0(np)-g_0(0)\;.
\end{equation}
In particular, for $n=p^{k-1}$, we recover the relation $\lim_{k\to\infty}p^{-k}\sum_{i=0,(i,p)=1}^{p^k-1}i^{-1}=0$ because $g_0'(0)=0$.
\item For all $\ell\geq 2$ one has
\begin{equation}\label{eq : S_l}
\frac{(-1)^{\ell-1}}{\ell}\cdot S_\ell(np)
\;\;=\;\;
-\sum_{m\geq \ell/2}\binom{1+2m}{\ell}(np)^{(1+2m-\ell)}\cdot
\frac{L_p(1+2m,\overline{\omega}_p^{2m})}{1+2m}\;.
\end{equation}
\end{enumerate}
}\fi
\end{corollary}
\begin{proof}
We have
$\Gamma^0_p(T)=\exp\Bigl(\lambda_0T-\sum_{m\geq
1}L_p(1+2m,\overline{\omega}_p^{2m})\cdot\frac{T^{1+2m}}{1+2m}\Bigr)$.
The functional equation gives
\begin{equation}\label{function al djklh}
\Gamma_p^0(T+np)/\Gamma_p^0(T)\;\;=\;\;A(1,np;T)\;\;=\;\;
-\prod_{\begin{smallmatrix}i=1,\\(i,p)=1\end{smallmatrix}}^{np-1}(T+i)\;.
\end{equation}
On the left hand side one finds
\begin{equation}
\Gamma_p^0(T+np)/\Gamma_p^0(T)=
\exp\Bigl(\lambda_0np-\sum_{m\geq
1}L_p(1+2m,\overline{\omega}_p^{2m})\cdot\frac{(T+np)^{1+2m}-T^{1+2m}}{1+2m}\Bigr)\;.
\end{equation}
We now compute the argument of the exponential. To 
simplify the notations let 
$b_{1+2m}:=L_p(1+2m,\overline{\omega}_p^{2m})/(1+2m)$, 
then
\begin{eqnarray}
\sum_{m\geq
1}b_{1+2m}\cdot((T+np)^{1+2m}-T^{1+2m})&=&\sum_{m\geq
1}b_{1+2m}\cdot\sum_{k=1}^{1+2m}\binom{1+2m}{k}(np)^{k}T^{1+2m-k}\\
&=&\sum_{\ell\geq 0}\Bigl(\sum_{m\geq\max(1,
\ell/2)}\binom{1+2m}{\ell}(np)^{(1+2m-\ell)}\cdot b_{1+2m}\Bigr)T^\ell\nonumber
\end{eqnarray}
Now taking $\log$ of both sides of \eqref{function al djklh} one finds
\begin{equation}
\lambda_0np-\sum_{\ell\geq 0}\;\Bigl(\;\sum_{m\geq\max(
\ell/2,1)}\binom{1+2m}{\ell}(np)^{(1+2m-\ell)}\cdot b_{1+2m}\;\Bigr)\;T^{\ell}=
\log\Bigl(-\prod_{\begin{smallmatrix}i=1,\\(i,p)=1\end{smallmatrix}}^{np-1}(T+i)\Bigr)\;.
\end{equation}
Then write $-\prod_{i=1,(i,p)=1}^{np-1}(T+i)=\Gamma_p(np)\cdot\prod_{i=1,(i,p)=1}^{np-1}(1+\frac{T}{i})$.
\if{
\begin{equation*}
-\prod_{i=1,(i,p)=1}^{p^n-1}(T+i)\;\;=\;\;\Gamma_p(p^n)\cdot\prod_{i=1,(i,p)=1}^{p^n-1}(1+\frac{T}{i})\;.
\end{equation*}
}\fi
Since $|\Gamma_p(np)-1|\leq|p|$,  it has a meaning to consider $\log(\Gamma_p(np))$. Then
\begin{eqnarray}
\log\Bigl(-\prod_{\begin{smallmatrix}i=1,\\(i,p)=1\end{smallmatrix}}^{np-1}(T+i)\Bigr)
&\;\;=\;\;&\log(\Gamma_p(np))+\sum_{i=1,(i,p)=1}^{np-1}\log(1+\frac{T}{i})\\
&=&\log(\Gamma_p(np))+\sum_{\ell\geq
1}\frac{(-1)^{\ell-1}}{\ell}S_\ell(np)T^\ell
\end{eqnarray}
This  proves the corollary. 
\end{proof}

\if{
\begin{corollary}\label{Cor : congruences}
For all $n\geq 1$ the following holds:
\begin{enumerate}
\item For $\ell=1$ one has
\begin{equation}\label{eq : S_1}
S_1(p^n) \;\;:=\;\;\sum_{i=1, (i,p)=1}^{p^n-1}\frac{1}{i}\;\;=\;\; 
-\sum_{m\geq 1}p^{2mn}\cdot L_p(1+2m,\overline{\omega}_p^{2m}) \;\;=\;\;
g_0(p^{n})-g_0(0)\;.
\end{equation}
In particular we recover the relation $\lim_{n\to\infty}p^{-n}\sum_{i=0,(i,p)=1}^{p^n-1}i^{-1}=0$ because $g_0'(0)=0$.
\item For all $\ell\geq 2$ one has
\begin{equation}\label{eq : S_l}
\frac{(-1)^{\ell-1}}{\ell}\cdot S_\ell(p^n)
\;\;=\;\;
-\sum_{m\geq \ell/2}\binom{1+2m}{\ell}p^{n(1+2m-\ell)}\cdot
\frac{L_p(1+2m,\overline{\omega}_p^{2m})}{1+2m}\;.
\end{equation}
\end{enumerate}
\end{corollary}
\begin{proof}
We have
$\Gamma^0_p(T)=\exp\Bigl(\lambda_0T-\sum_{m\geq
1}L_p(1+2m,\overline{\omega}_p^{2m})\cdot\frac{T^{1+2m}}{1+2m}\Bigr)$.
The functional equation gives
\begin{equation}\label{function al djklh}
\Gamma_p^0(T+p^n)/\Gamma_p^0(T)\;\;=\;\;A(1,p^n;T)\;\;=\;\;
-\prod_{\begin{smallmatrix}i=1,\\(i,p)=1\end{smallmatrix}}^{p^n-1}(T+i)\;.
\end{equation}
On the left hand side one finds
\begin{equation}
\Gamma_p^0(T+p^n)/\Gamma_p^0(T)=\exp\Bigl(\lambda_0p^n-\sum_{m\geq
1}L_p(1+2m,\overline{\omega}_p^{2m})\cdot\frac{(T+p^n)^{1+2m}-T^{1+2m}}{1+2m}\Bigr)
\end{equation}
We now compute the argument of the exponential. To 
simplify the notations let 
$b_{1+2m}:=L_p(1+2m,\overline{\omega}_p^{2m})/(1+2m)$, 
then
\begin{eqnarray}
\sum_{m\geq
1}b_{1+2m}\cdot((T+p^n)^{1+2m}-T^{1+2m})&=&\sum_{m\geq
1}b_{1+2m}\cdot\sum_{k=1}^{1+2m}\binom{1+2m}{k}p^{nk}T^{1+2m-k}\\
&=&\sum_{\ell\geq 0}\Bigl(\sum_{m\geq\max(1,
\ell/2)}\binom{1+2m}{\ell}p^{n(1+2m-\ell)}\cdot b_{1+2m}\Bigr)T^\ell\nonumber
\end{eqnarray}
Now taking $\log$ of both sides of \eqref{function al djklh} one finds
\begin{equation}
\lambda_0p^n-\sum_{\ell\geq 0}\;\Bigl(\;\sum_{m\geq\max(
\ell/2,1)}\binom{1+2m}{\ell}p^{n(1+2m-\ell)}\cdot b_{1+2m}\;\Bigr)\;T^{\ell}=
\log\Bigl(-\prod_{\begin{smallmatrix}i=1,\\(i,p)=1\end{smallmatrix}}^{p^n-1}(T+i)\Bigr)\;.
\end{equation}
Then write $-\prod_{i=1,(i,p)=1}^{p^n-1}(T+i)=\Gamma_p(p^n)\cdot\prod_{i=1,(i,p)=1}^{p^n-1}(1+\frac{T}{i})$.
\if{
\begin{equation*}
-\prod_{i=1,(i,p)=1}^{p^n-1}(T+i)\;\;=\;\;\Gamma_p(p^n)\cdot\prod_{i=1,(i,p)=1}^{p^n-1}(1+\frac{T}{i})\;.
\end{equation*}
}\fi
Since $|\Gamma_p(p^n)-1|\leq|p|$,  it has a meaning to consider $\log(\Gamma_p(p^n))$. Then
\begin{eqnarray}
\log\Bigl(-\prod_{\begin{smallmatrix}i=1,\\(i,p)=1\end{smallmatrix}}^{p^n-1}(T+i)\Bigr)
&\;\;=\;\;&\log(\Gamma_p(p^n))+\sum_{i=1,(i,p)=1}^{p^{n}-1}\log(1+\frac{T}{i})\\
&=&\log(\Gamma_p(p^n))+\sum_{\ell\geq
1}\frac{(-1)^{\ell-1}}{\ell}S_\ell(p^n)T^\ell
\end{eqnarray}
This  proves the corollary. 
\end{proof}
}\fi
\begin{remark}
\label{Rk : link with Was}
Equality \eqref{eq : S_l} is
equivalent to the following relation given in 
\cite[Theorem 1,(a)]{Washington}: for all integers $n,\ell\geq 1$ one has
\begin{equation}\label{eq : Was}
S_\ell(np+1)\;=\;\sum_{i=1,(i,p)=1}^{np}\frac{1}{i^{\ell}}\;=\;
-\sum_{k\geq 1}\binom{-\ell}{k}(np)^k\cdot 
L_p(\ell+k,\omega_p^{1-k-\ell})\;.
\end{equation}
We now prove the equivalence with  
\eqref{eq : S_l}. First notice that  
$S_\ell(np+1)=S_\ell(np)$ because it is a sum over 
natural numbers that are prime to $p$. 
Moreover observe that
$(-1)^k\cdot\binom{-\ell}{k}
=(-1)^{\ell-1}\cdot\binom{k-1}{\ell-1}$, that 
$L_p(s,\overline{\omega}_p^k)=0$ if $k$ is odd, and 
that by definition $\omega_p^{1-k-\ell} 
=\overline{\omega}_p^{\ell-k-1}$.
Equation \eqref{eq : Was} is then equivalent to
\begin{equation}\label{eq : Was-2}
S_\ell(np)\;=\;
-\sum_{k\geq 1}(-1)^{k}\binom{\ell+k-1}{k}(np)^k\cdot 
L_p(\ell+k,\overline{\omega}_p^{\ell+k-1})\;.
\end{equation}
Since $L_p(\ell+k,\overline{\omega}_p^{\ell+k-1})=0$ 
if $\ell+k-1$ is odd, we have 
$L_p(\ell+k,\overline{\omega}_p^{\ell+k-1})=
(-1)^{\ell+k-1}
L_p(\ell+k,\overline{\omega}_p^{\ell+k-1})$, moreover 
$\binom{\ell+k-1}{k}=\binom{\ell+k-1}{\ell-1}=\binom{\ell+k}{\ell}\cdot\frac{\ell}{\ell+k}$. 
Equation \eqref{eq : Was-2} is then equivalent to
\begin{eqnarray}
S_\ell(np)&\;=\;&
-\sum_{k\geq 1}(-1)^{2k+\ell-1}\ell\binom{\ell+k}{\ell}
(np)^k\cdot 
\frac{L_p(\ell+k,\overline{\omega}_p^{\ell+k-1})}{\ell+k}\\
&=&
-\frac{\ell}{(-1)^{\ell-1}}\cdot
\sum_{k\geq 1}\binom{\ell+k}{\ell}
(np)^k\cdot 
\frac{L_p(\ell+k,\overline{\omega}_p^{\ell+k-1})}{\ell+k}\\
&=&-\frac{\ell}{(-1)^{\ell-1}}\cdot
\sum_{s\geq \ell}\binom{s+1}{\ell}
(pn)^{s-\ell+1}\frac{L_p(s+1,
\overline{\omega}_p^{s})}{1+s}\;.
\end{eqnarray}
Since $L_p(s+1,
\overline{\omega}_p^{s})=0$ for $s$ odd, this is 
equivalent to \eqref{eq : S_l}.
\if{
\begin{equation}
\frac{(-1)^{\ell-1}}{\ell}S_\ell(np)\;=\;
-
\sum_{m\geq \ell/2}\binom{2m+1}{\ell}
(pn)^{2m-\ell+1}\frac{L_p(2m+1,
\overline{\omega}_p^{2m})}{1+2m}
\end{equation}
}\fi
\end{remark}

\subsection{Note}
Examples of $ \Sigma$-deformation appear in 
several  process. Here are some examples:

The deformation appears  in \cite[7.1]{Astx} to 
show the independence to the Frobenius.

The $\sigma_{q,0}$-deformation, with 
$q\in\bs{\mu}_p$ a $p$-th root of unity, appears in \cite[10.4.2]{Kedlaya-book} 
under the name of ``\emph{Taylor series}'' to show the 
existence of the antecedent by Frobenius of a differential 
equation $\Fs$, which is the 
sub-space of $\Fs$ fixed points under the action of 
$\bs{\mu}_p$ on $\Fs$ obtained by deformation.
 
It also appears in \cite[3.2.2, 3.2.6]{Kedlaya-draft} to 
show the existence of rank one submodules. 

\if{

\section{\Huge{Compliance with Ethical Standards.}}
\begin{center}
\textbf{\Huge {Disclosure of potential conflicts of interest}}.
\end{center}
{\Huge{\textbf{Conflict of Interest}: 
The author declares that he has no conflict of interest.}
}
}\fi
\bibliographystyle{amsalpha}
\bibliography{bib}

\end{document}